\documentclass[english,11pt,a4paper,leqno]{amsart}
\usepackage{amsmath,amstext, amsthm, amssymb}
\usepackage{latexsym, amsthm, amsfonts, amsmath,amssymb,graphicx,xcolor}
\usepackage[latin1]{inputenc} 
\usepackage[T1]{fontenc} 
\usepackage{enumerate}
\usepackage[english]{babel} 
\usepackage{enumitem}

\newtheorem{thm}{Theorem}[section]

\newtheorem{lem}[thm]{Lemma}
\newtheorem{coro}[thm]{Corollary} 

\newtheorem{conj}[thm]{Conjecture}
\theoremstyle{definition}
\newtheorem{defi}[thm]{Definition}

\theoremstyle{remark} 
\newtheorem{rem}[thm]{Remark}
\theoremstyle{definition}

\newtheorem{Fact}{Fact}
\newcounter{constant}
\newcommand{\newconstant}[1]{\refstepcounter{constant}\label{#1}}
\newcommand{\useconstant}[1]{c_{\ref{#1}}}
\newcommand{\defconstant}[1]{ \newconstant{c_{#1}}\expandafter\newcommand\csname c#1\endcsname{\useconstant{c_{#1}}} }  

\newcommand{\mZ}{{\mathcal Z}}

\newcommand{\Q}{\overline{\mathbb Q}}

\renewcommand{\r}{\boldsymbol r}

\newcommand{\C}{\mathbb{C}}
\newcommand{\Z}{\mathbb{Z}}

\newcommand{\N}{\mathbb{N}}
\newcommand{\R}{\mathcal R}
\newcommand{\M}{\mathcal M}

\newcommand{\K}{\mathcal K}
\newcommand{\A}{\mathcal A}

\newcommand{\I}{\mathcal I}

\newcommand{\J}{\mathcal J}

\newcommand{\E}{\mathcal E}

\newcommand{\be}{{\boldsymbol{e}}}
\newcommand{\x}{{\boldsymbol{x}}}

\newcommand{\z}{{\boldsymbol{z}}}
\newcommand{\lambd}{{\boldsymbol{\lambda}}}
\newcommand{\f}{{\boldsymbol f}}

\renewcommand{\S}{\mathcal{S}}
\renewcommand{\k}{{\boldsymbol{k}}}

\newcommand{\balpha}{{\boldsymbol{\alpha}}}
\newcommand{\X}{\boldsymbol{X}}

\newcommand{\Y}{\boldsymbol{Y}}

\newcommand{\bmu}{{\boldsymbol{\mu}}}
\newcommand{\bnu}{{\boldsymbol{\nu}}}

\newcommand{\bxi}{{\boldsymbol{\xi}}}

\newcommand{\bphi}{{\boldsymbol{\phi}}}

\newcommand{\gamm}{{\boldsymbol{\gamma}}}
\newcommand{\thet}{{\boldsymbol{\theta}}}

\newcommand{\kapp}{{\boldsymbol{\kappa}}}
\newcommand{\btau}{{\boldsymbol{\tau}}}

\renewcommand{\a}{{\boldsymbol{a}}}

\newcommand{\bR}{{\boldsymbol{R}}}
\newcommand{\bTheta}{{\boldsymbol{\Theta}}}
\numberwithin{equation}{section} 

\title[Mahler's method in several variables ]{Mahler's method in several variables and finite automata}

\author{Boris Adamczewski}
\address{
Univ Lyon, Universit\'e Claude Bernard Lyon 1\\
 CNRS UMR 5208, Institut Camille Jordan \\
 F-69622 Villeurbanne Cedex, France}
\email{Boris.Adamczewski@math.cnrs.fr}

\author{Colin Faverjon}
\address{
Univ Lyon, Universit\'e Claude Bernard Lyon 1\\
 CNRS UMR 5208, Institut Camille Jordan \\
 F-69622 Villeurbanne Cedex, France}
\email{colin.faverjon@riseup.net}
\date{}

\thanks{This project has received funding from the European Research Council (ERC) under the 
European Union's Horizon 2020 research and innovation programme 
under the Grant Agreement No 648132. }
\begin{document}

\begin{abstract}  
We develop a theory of linear Mahler systems in several variables from the 
perspective of transcendence 
and algebraic independence, which also includes the possibility of 
dealing with several  systems associated with sufficiently independent matrix transformations.  
Our main results go far beyond the existing literature, also surpassing those of two unpublished preprints 
the authors made available 
on the arXiv in 2018. 
 The main new feature is that they 
apply now without any restriction on the matrices defining the corresponding Mahler systems.  
As a consequence, we settle several  problems concerning  
expansions of numbers in multiplicatively independent bases.    
For instance, we prove that no irrational real number can be automatic in two multiplicatively independent 
integer bases, and we give a new proof and a broad algebraic generalization of Cobham's theorem in automata theory.  
We also provide a new proof and a multivariate generalization of Nishioka's theorem, a landmark result in Mahler's method. 
 \end{abstract}

\bibliographystyle{abbvr}
\maketitle
\setcounter{tocdepth}{1}
\tableofcontents

\section{Introduction }\label{sec: introduction}

It is commonly expected that expansions of numbers in multiplicatively independent bases, 
such as $2$ and $10$, should have no common structure. 
However, it seems extraordinarily difficult to confirm this naive heuristic principle 
in some way or another. In the late 1960s, Furstenberg \cite{Fu67,Fu70} 
suggested a series of conjectures, which became famous, and aim to capture  
this heuristic (see  Conjecture \ref{conj: F} in Appendix \ref{preamble}). 
Despite recent remarkable progress by Shmerkin \cite{Sh19} and Wu \cite{Wu19}, Conjecture \ref{conj: F}   
remains totally out of reach of the current methods. 
 As always when mathematicians have to face such an enormous gap between 
heuristic and knowledge, it becomes essential to find out \emph{good problems}. 
By that, we mean problems which, 
on the one hand, formalize and express the general heuristic, and, on the other hand, whose solution 
does not seem desperately  out of reach. 
While Furstenberg's conjectures take place in a dynamical setting,  
we use instead the language of automata theory to formulate some related conjectures 
that, hopefully, belong to the above category.  These conjectures are introduced and discussed in 
Appendix \ref{preamble}. 
Thanks to the work of Cobham \cite{Co68},  
various problems involving 
numbers generated by finite automata can be translated and extended to problems 
concerning transcendence and algebraic  
independence of values of $M$-functions. 
Furthermore,  
such problems 
fall into the scope of  \emph{Mahler's method}.   
In the end, we are able to settle our conjectures after proving the apparently unrelated 
Theorem \ref{thm: main}.

Let $\Q\subset \mathbb C$ denote the field of algebraic numbers and, 
given a field $\mathbb K\subset \mathbb C$, let $\mathbb K\{z\}$ denote 
the ring of convergent power series with coefficients in $\mathbb K$. 
 Given an integer $q\geq 2$,  $f(z)\in \overline{\mathbb Q}\{z\}$ is 
said to be a  {\it $q$-Mahler function}  if there exist polynomials 
$p_0(z),\ldots , p_m(z)\in \overline{\mathbb Q}[z]$, 
not all zero,  such that  
\begin{equation} \label{eq:Mahler1}
 p_0(z)f(z)+p_1(z)f(z^q)+\cdots + p_m(z)f(z^{q^m}) \ = \ 0. 
 \end{equation}  
 If $f(z)$ is $q$-Mahler for some $q$, we simply say that $f(z)$ is a Mahler function, or an $M$-function.  
The coefficients of an $M$-function generate only a finite field extension of $\mathbb Q$. 
Let us also recall that nonzero complex numbers 
$x_1,\ldots,x_r$ are multiplicatively  
 independent if there is no nonzero tuple of integers $(n_1,\ldots,n_r)$ such that 
$x_1^{n_1}\cdots x_r^{n_r}=1$.

\begin{thm}\label{thm: main}
Let $r\geq 1$ be an integer and $\mathbb K\subseteq \Q$ be a field. 
For every integer $i$, $1\leq  i\leq r$, we let $q_i\geq 2$ be an integer, 
$f_i(z)\in \mathbb K\{z\}$ be a $q_i$-Mahler function, 
and $\alpha_i\in \mathbb K$, $0 <\vert\alpha_i\vert <1$, be such that $f_i(z)$ 
is well-defined at $\alpha_i$.  Let us assume that one of the two following 
properties holds. 

\begin{itemize}

\item[{\rm (i)}] The numbers $\alpha_1,\ldots, \alpha_r$ are multiplicatively independent. 
 
\smallskip

\item[{\rm (ii)}] The numbers  $q_1,\ldots, q_r$ are pairwise multiplicatively independent. 
\end{itemize}
Then  $f_1(\alpha_1),f_2(\alpha_2),\ldots,f_r(\alpha_r)$ are algebraically independent over $\Q$, unless  
one of them belongs to $\mathbb K$.   

\end{thm}

Until now, Theorem \ref{thm: main} was only proved when $r=1$.  
This special case, conjectured by Cobham \cite{Co68} 
in 1968 and settled by the authors \cite{AF1}  in 2017, 
implies that the decimal expansion of algebraic 
irrational numbers cannot be generated by  finite automata\footnote{This result was first proved 
by Bugeaud and the first author \cite{AB07} by means of the subspace theorem.}.  
Also, an algorithm to determine whether 
the numbers $f_i(\alpha_i)$  belong to $\mathbb K$ 
or not is described in \cite{AF2}.

Let us point out 
the four main difficulties we have to face when trying to prove Theorem \ref{thm: main}.

\begin{itemize}

\item[(i)] We have to consider a bunch of \emph{arbitrary $M$-functions}. 
In contrast, many results in the past 
where restricted to inhomogeneous order one equations (see Section \ref{sec: mahlersv}). 
That is, equations of the form $p_{-1}(z)+p_0(z)f(z)+p_1(z)f(z^q)=0$. Being able to deal 
with arbitrary equations becomes essential for applications involving automata. 

\item[(ii)] Given an $M$-function, we have to consider its values at \emph{arbitrary algebraic points} 
where it is well-defined, while a classical feature of results in this framework is that they are 
only available for points which 
are \emph{regular}\footnote{See Definition \ref{def:reg}.} with respect to the underlying Mahler system.  

\item[(iii)]  We have to consider simultaneously values of $M$-functions at 
\emph{different algebraic points}. In the setting of Siegel $E$-functions, the study of algebraic relations 
between values of  $E$-functions at different algebraic points can 
be achieved by considering different $E$-functions at the same point. 
Indeed, if $f(z)$ is an $E$-function and $\alpha$ is an algebraic number, then 
the function $f(\alpha z)$ is still an $E$-function.  However, this trick no longer works for 
$M$-functions. 

\item[(iv)] We have to consider $M$-functions associated with 
\emph{different transformations} ({\it i.e.}, $z\mapsto  z^{q}$ with different $q$). 

\end{itemize}

Thanks to the work of Ku.\ Nishioka \cite{Ni90}, the transcendence 
theory of linear Mahler systems in one variable 
is well-developed. It has even reached a rather 
definitive stage after the recent works of Philippon \cite{PPH} and the authors \cite{AF1}.  
These new results provide tools to overcome (ii), and also (i) in some situations.   
However, Theorem \ref{thm: main} does not fall into the scope 
of Mahler's method in one variable. In particular, the problem raised by (iii) requires 
a major development of Mahler's method in several variables. Partial results in this direction are due to 
Mahler \cite{Ma30b}, Kubota \cite{Ku77}, Loxton and van der Poorten \cite{LvdP77II,LvdP82}, 
and Nishioka \cite{Ni96}.  
Last but not least, (iv) is a source of well-known difficulties and only limited results, though of great interest,  
have been obtained by Nishioka \cite{Ni94} and Masser \cite{Mas99}. 

In what follows, we develop a theory of linear Mahler systems in several variables 
from the perspective of transcendence and algebraic independence, which also includes the possibility of 
dealing with several  systems associated with sufficiently independent matrix transformations.  
It is condensed in three main general results, Theorems \ref{thm: permanence}, \ref{thm: purete2}, 
and \ref{thm: purity}, which go far beyond the existing literature.   
These results also surpass those of two unpublished preprints 
\cite{AF3,AF4}  
the authors made  
available on the arXiv in 2018. 
The main new feature with respect to these two preprints is that our results   
applies now without any restriction on the matrices defining the systems 
under consideration.  
{\it In fine}, the new approach we follow allows us to overcome all the aforementioned difficulties.         
To tell the truth, proving Conjectures \ref{conj: weakv} and \ref{conj: strongv}, and Corollary \ref{coro: fonctions} 
stated in Appendix \ref{preamble}, was our initial goal.    
In order to measure the relevance of the theory eventually developed in Section \ref{sec: mahlersv} to reach this goal, 
the reader is thus encouraged to look at Appendix \ref{preamble}.  
However, in our opinion, this theory is equally valuable in its own right.

\subsection*{Organization of the paper}

In Section \ref{sec: mahlersv}, we state our main results concerning the study of 
linear Mahler systems in several variables, namely Theorems \ref{thm: permanence}, 
\ref{thm: purete2}, and \ref{thm: purity}. We also discuss the three main new ingredients of our approach 
in Section \ref{sec: ingredient}. 
Some notation are introduced in Section \ref{sec: notation}. 
As made clear in Section \ref{sec: mahlersv}, the strength of our results strongly depends on 
our ability to provide simple and natural conditions that ensure certain admissibility conditions. 
This problem is addressed in Section \ref{sec: admissibility} where concrete and optimal conditions are given. 
In Section \ref{sec: vanishing} we prove a new vanishing theorem that is a key ingredient for proving 
Theorem \ref{thm: purity}.  
In Section \ref{sec: families}, we state Theorem \ref{thm: families}, a general result dealing with 
families of linear Mahler systems associated with sufficiently independent transformations.  
Some preliminary results for proving Theorem \ref{thm: families} are gathered in Section \ref{sec: matrix}. 
Then Theorem \ref{thm: families} is proved in Section \ref{sec: mainproof}, while     
Theorems \ref{thm: permanence},  \ref{thm: purete2}, \ref{thm: purity}, and Corollaries \ref{thm: Nishioka} and  
\ref{cor:2}  are derived from Theorem \ref{thm: families}  
in Section \ref{sec: final}. 
Finally, we deduce Theorem \ref{thm: main} from Theorems  \ref{thm: purete2} and \ref{thm: purity} in Section \ref{sec: theorem1}, and 
Conjectures \ref{conj: weakv}, \ref{conj: strongv}, and Corollary \ref{coro: fonctions} from Theorem \ref{thm: main}
in Appendix \ref{preamble}.  

\section{Mahler's method in several variables}\label{sec: mahlersv}

Let $n\geq 1$ be an integer and $\z=(z_1,\ldots,z_n)$ be a $n$-tuple of indeterminates. 
We let $\Q\{\z\}$ denote the ring of $n$ variables convergent power 
series with algebraic coefficients,      
and we set $\Q^\star:=\Q\setminus \{0\}$.  Given a field extension $\mathbb L$ of a field 
$\mathbb K$, and 
$a_1,\ldots,a_m$ in $\mathbb L$, 
we let ${\rm tr.deg}_{\mathbb K}(a_1,\ldots,a_m)$ denote the transcendence degree  of  $\mathbb K(a_1,\ldots,a_m)$ over 
$\mathbb K$. 

Given 
$f_1(\z),\ldots,f_m(\z)\in \Q\{\z\}$ related by a system  
of the form   
\eqref{eq:mahler} (see below) and $\balpha\in (\Q^\star)^n$ a point at which these functions are well-defined,    
 Mahler's method aims at transferring results about the absence of algebraic (resp., linear) relations  between 
$f_1(\z),\ldots,f_m(\z)$ over $\Q(\z)$ to the absence of algebraic 
(resp., linear) relations over $\overline{\mathbb Q}$ between the complex numbers 
$f_1(\balpha),\ldots,f_m(\balpha)$. In particular, a reoccurring theme consists in establishing the equality 
\begin{equation}\label{eq:degretranscendance0}
{\rm tr.deg}_{\Q}(f_1(\balpha),\ldots,f_m(\balpha))= {\rm tr.deg}_{\Q(\z)}(f_1(\z),\ldots,f_m(\z)) \, ,
\end{equation}
under some reasonable assumptions on $A(\z)$, $T$, and $\balpha$. 
This problem goes back to the pioneering work of 
Mahler \cite{Ma29,Ma30a,Ma30b} at the end of the 1920s. 
\subsection{Mahler's transformations and linear Mahler systems}

Let  $T=(t_{i,j})_{1\leq i,j\leq n}$ be an $n\times n$ matrix with 
nonnegative integer coefficients. 
We set 
$$
T \z = (z_1^{t_{1,1}}z_2^{t_{1,2}}\cdots z_n^{t_{1,n}},\ldots,
z_1^{t_{n,1}}z_2^{t_{n,2}}\cdots z_n^{t_{n,n}})\, , 
$$ 
and we let also $T$ act on $\mathbb C^n$ in a similar way.

\begin{defi}
A \emph{linear $T$-Mahler system}, or simply a \emph{Mahler system}, is a system 
of functional equations of the form 

\begin{equation}
\label{eq:mahler}
\left(\begin{array}{c} f_1(T\z) \\ \vdots \\ f_m(T\z) \end{array} \right) = 
A(\z)\left(\begin{array}{c} f_1(\z) \\ \vdots \\ f_m(\z) 
\end{array} \right)\, ,
\end{equation}
where $A(\z)\in{\rm GL}_m(\Q(\z))$. 
 A \emph{Mahler function} $f(\z)\in\Q\{\z\}$ is a coordinate of a vector representing a solution to 
a linear Mahler system.   
\end{defi}

\begin{defi}\label{def:reg}
A point $\balpha \in (\Q^\star)^n$ is said to be \emph{regular} with respect to the Mahler system 
\eqref{eq:mahler} if the matrix $A(\z)$ is well-defined and invertible at $T^k\balpha$ for all nonnegative integers $k$. 
\end{defi} 

\subsection*{Warning}
Independently of the choice of the Mahler system \eqref{eq:mahler}, there are some unavoidable 
restrictions that one has to impose on the matrix transformation $T$ and on the algebraic point $\balpha$. 
When these conditions are fulfilled, the pair $(T,\balpha)$ is said to be \emph{admissible}.  We postpone the definition of an admissible pair to Section \ref{sec: admissibility}, but let us just already say that, in this respect, our results are as general as possible.  
With this formalism, all our results are concerned with values at some algebraic point 
$\balpha$ of some Mahler functions $f_1(\z),\ldots,f_m(\z)$ related by a system  
of the form \eqref{eq:mahler} under the 
assumption that: 

\begin{itemize}
\item[(a)] the pair $(T,\balpha)$ is admissible,  

\item[(b)] $\balpha$ is regular with respect to \eqref{eq:mahler}. 
\end{itemize}
As discussed in \cite{Ad19}, these conditions are typical in Mahler's method. 

\subsection{The lifting theorem}

As a first contribution, we prove the following result.  
Let us recall that a field extension $\mathbb L$ of a field $\mathbb K$ is said to be \emph{regular}\footnote{The reader will take care that we use two totally different notions of regularity in this paper.} if $\mathbb K$ is algebraically closed in $\mathbb L$ and $\mathbb L$ is separable over $\mathbb K$. 
If $\balpha\in \Q^n$, we let 
$\overline{\mathbb Q(\z)}_{\balpha}$ denote the algebraic closure of $\Q(\z)$ in $\Q\{\z-\balpha\}$\footnote{ The ring $\Q\{\z-\balpha\}$ is the ring of convergent power series in $(\z-\balpha)$.}.

\begin{thm}[Lifting]
\label{thm: permanence}
Let $f_1(\z),\ldots,f_m(\z)\in \Q\{\z\}$ be related by a system of functional equations of the form  
\eqref{eq:mahler}. 
Let us assume that $\balpha\in  (\Q^\star)^n$ is a regular point with respect to \eqref{eq:mahler} and 
that the pair $(T,\balpha)$ is admissible. Then for every homogeneous polynomial 
$P \in \Q[X_1,\ldots,X_m]$ such that 
$$
P(f_{1}(\balpha),\ldots,f_{m}(\balpha)) = 0\,,
$$
there exists a homogeneous polynomial $Q\in\overline{\mathbb Q(\z)}_{\balpha}[X_1,\ldots,X_m]$  
such that 
\begin{eqnarray*}
Q(\z,f_{1}(\z),\ldots,f_{m}(\z)) = 0 & \mbox{ and } &
Q(\balpha,X_1,\ldots,X_m)=P(X_1,\ldots,X_m).
\end{eqnarray*}
Furthermore, if  $\Q(\z)(f_1(\z),\ldots,f_m(\z))$ is a regular extension of $\Q(\z)$, then there exists such a polynomial $Q$ in $\Q[\z,X_1,\ldots,X_m]$. 
\end{thm}

Theorem \ref{thm: permanence} is the first that applies to \emph{all} 
linear Mahler systems in several variables, that is, without any restriction on the matrix $A(\z)$. Furthermore,  the quantitative Equality \eqref{eq:degretranscendance0} is replaced by a qualitative statement: 
any algebraic relation over $\Q$ between the values $f_1(\balpha),\ldots,f_m(\balpha)$  can be lifted to 
a similar algebraic relation over $\overline{\mathbb Q(\z)}$ between the functions $f_1(\z),\ldots,f_m(\z)$.  
Such a qualitative refinement is a key for applications. 

\begin{rem} 
Theorems \ref{thm: permanence}  also applies to nonhomogeneous polynomials, for we can always  
turn an inhomogeneous relation into an homogeneous one by adding the constant function 
$f_{0}\equiv 1$ to the system and replacing the matrix $A(\z)$ by 
$$\left(\begin{array}{c|ccc} 
1 &  &0&
\\ \hline 
\\ 
 0 &  &A(\z) &
 \\ &&&
\end{array}\right)\, .$$
\end{rem}

As a corollary of the lifting theorem, we deduce the following result.   

\begin{coro}
\label{thm: Nishioka}
Let $f_1(\z),\ldots,f_m(\z)\in \Q\{\z\}$ be related by a system of the form  
\eqref{eq:mahler}. 
Let us assume furthermore that $\balpha\in  (\Q^\star)^n$ is a regular point with respect to \eqref{eq:mahler} 
and 
that the pair $(T,\balpha)$ is admissible. Then 
\begin{equation}\label{eq:degretranscendance2}
{\rm tr.deg}_{\Q}(f_1(\balpha),\ldots,f_m(\balpha))= {\rm tr.deg}_{\Q(\z)}(f_1(\z),\ldots,f_m(\z)) \,.
\end{equation}
\end{coro}

Now, let us compare Theorem \ref{thm: main} with previous results on the subject.

\subsubsection*{The case $n=1$}   In that case, the operator $T$ takes the simple 
form $z\mapsto z^q$, where $q\geq 2$ is an integer, and the pair $(T,\alpha)$ is admissible as soon as $0<\vert \alpha\vert <1$.  
Furthermore, the field extension $\Q(z)(f_1(z),\ldots,f_m(z))$ is always regular. 
 After several partial results due to Mahler, Kubota, and 
Loxton and van der Poorten, Ku.\ Nishioka \cite{Ni90} finally proved in 1990 that 
\begin{equation}\label{eq: degtr}
{\rm tr.deg}_{\Q}(f_1(\alpha),\ldots,f_m(\alpha))= {\rm tr.deg}_{\Q(z)}(f_1(z),\ldots,f_m(z)) \, 
\end{equation} 
for all matrices $A(z)$ and all regular points $\alpha \in \Q$, $0< \vert \alpha\vert<1$. 
This is certainly a landmark result in Mahler's method. 
The proof of Nishioka's theorem is based on some technics from commutative algebra first introduced in the framework of algebraic independence by Nesterenko.    
More recently, Philippon \cite{PPH} and then the authors \cite{AF1} refine 
Nishioka's theorem by proving the case $n=1$ of Theorem \ref{thm: permanence}, which we refer 
to as Philippon's lifting theorem.    
Similar lifting theorems 
have first been obtained in the framework of linear differential equations ({\it e.g.}, Siegel $E$-functions) by Nesterenko and Shidlovskii \cite{NS96}, 
by Beukers \cite{Beu06} using some results of Andr\'e \cite{An1,An2} on the theory of $E$-operators,  
and then by Andr\'e \cite{An3}. A recent proof of Philippon's lifting theorem in the spirit of \cite{An3} 
is also given in \cite{NS}. In \cite{PPH,AF1,NS}, the latter  
is derived from Nishioka's theorem, while  
our proof of Theorem \ref{thm: permanence} has little 
in common with these papers and \cite{Ni90}.  
In particular, it provides a new and more elementary way to prove the theorems of Nishioka and Philippon.

\subsubsection*{The case $n\geq 2$}  
Unfortunately, the method used for proving Nishioka's theorem hardly generalizes to higher dimension. 
In 1982, Loxton and van der Poorten \cite{LvdP82}  
published a paper claiming that 
\begin{equation}\label{eq:degretranscendance}
{\rm tr.deg}_{\Q}(f_1(\balpha),\ldots,f_m(\balpha))= {\rm tr.deg}_{\Q(\z)}(f_1(\z),\ldots,f_m(\z)) \,
\end{equation}
when the matrix $A(\boldsymbol 0)$ is well-defined and nonsingular, the pair $(T,\balpha)$ is admissible, 
and $\balpha$ is a regular algebraic point.  It was the main result published in this area, but 
unfortunately some argument in their proof is flawed. 
This is reported, for instance, by Nishioka in \cite{Ni90}. 
In the end, Mahler's method in several variables has been applied successfully only for the two following 
much more restricted classes of matrices. 
 In 1977, Kubota \cite{Ku77} proved  that Equality \eqref{eq:degretranscendance} holds 
true when the matrix $A(\z)$ is 
\emph{almost diagonal}, that is, when 
the functions $f_i(\z)$ satisfy a system of equations of the form 
\begin{equation}\label{eq:degre1}
\left(\begin{array}{c} 1\\ f_1(T\z) \\ \vdots \\ f_m(T\z)  \end{array} \right) = 
\left(\begin{array}{c|ccc} 
1 & 0\cdots &  & 0\\
\hline 
 b_1(\z)& a_1(\z)&&
\\  \vdots && \ddots &
\\ b_m(\z) &&& a_m(\z) 
 \end{array}\right)\left(\begin{array}{c} 1\\ f_1(\z) \\ \vdots \\ f_m(\z) 
\end{array} \right)\, 
\end{equation}
where $a_i(\z), b_i(\z) \in \Q(z)$ have no pole at $\boldsymbol 0$, and $a_i(\boldsymbol 0) \neq 0$.  
Such systems are precisely those arising from the study of several inhomogeneous equations of order one. 
A variant of this result is due to Nishioka \cite{Ni96}, who proved in 1996 that Equality \eqref{eq:degretranscendance} also holds true when the 
matrix $A(\z)$ is \emph{almost constant}, that is,  
for systems of the form  
\begin{equation}\label{eq:constante}
\left(\begin{array}{c} 1\\ f_1(T\z) \\ \vdots \\ f_m(T\z)  \end{array} \right) = 
\left(\begin{array}{c|ccc} 
1 & 0&\cdots   & 0
\\ 
\hline
b_1(\z) &&& 
\\
\vdots   & &B & 
\\ b_m(\z) &&& 
 \end{array}\right)\left(\begin{array}{c} 1\\ f_1(\z) \\ \vdots \\ f_m(\z) 
\end{array} \right)\, 
\end{equation}
where $B\in {\rm GL}_m(\Q)$, and $b_i(\z)\in \Q(\z)$ have no pole at $\boldsymbol 0$.  
 The proof of these results (and also of the failed attempt by  
Loxton and van der Poorten) follow closely the approach initiated by Mahler in \cite{Ma30b}. 
We stress that, so far, this remained the only available strategy to tackle this problem 
(see  Section \ref{sec: ingredient}).


\subsection{The two purity theorems}
According to the lifting theorem, the study of the algebraic relations between the values of Mahler functions 
related by a system of equations of the form  
\eqref{eq:mahler}
can be reduced to the easier study of the algebraic relations between the functions themselves.  
 However,  easier   
does not necessarily mean \emph{easy}, and, so far, only the linear relations between 
$M$-functions have been fully understood \cite{AF1,AF2}. 
Our second main result is of a different nature. It states that, when evaluated at sufficiently independent 
algebraic points, Mahler functions associated with transformations 
having the same spectral radius \emph{always} behave independently.  
The main feature of this result is that there is no need to check any kind of independence 
between the Mahler functions under consideration.  
 
To state this result properly, we first need some notation. 
Let us consider several tuples of complex numbers 
$$\mathcal  E_1=(\zeta_{1,1},\ldots,\zeta_{1,s_1}),\ldots, \mathcal  E_r
=(\zeta_{r,1},\ldots,\zeta_{r,s_r}) \,.$$ 
With every $i$, $1\leq i\leq r$, we associate 
a vector of indeterminates $\X_i=(X_{i,1},\ldots,X_{i,s_i})$, and we let 
$$
{\rm Alg}_{\Q}(\mathcal E_i):= \left\{P(\X_i)\in\Q[\X_i] : 
P(\zeta_{i,1},\ldots,\zeta_{i,s_i})=0\right\} \,
$$ 
denote the ideal of algebraic relations over $\Q$ between the coordinates of $\mathcal E_i$.  
We also consider the tuple  $\mathcal E=(\zeta_{1,1},\ldots,\zeta_{r,s_r})$ obtained by concatenation of the tuples $\mathcal E_i$, and we set $\X:=(\X_1,\ldots,\X_r)$ and  
$$
{\rm Alg}_{\Q}(\mathcal E):= \left\{P(\X)\in\Q[\X] : P(\zeta_{1,1},\ldots,\zeta_{r,s_r})
=0\right\} \,.
$$ 
We say that  $P\in {\rm Alg}_{\Q}(\mathcal E)$  
is a \emph{pure algebraic relation} with respect to $\mathcal E_i$ if 
it belongs to the extended ideal 
$$
{\rm Alg}_{\Q}(\mathcal E_i \mid \mathcal E) 
:=  {\rm span}_{\Q[\X]} \{ P(\X_i) : P\in {\rm Alg}_{\Q}(\mathcal E_i)  \}\, .
$$  
Our second main result reads as follows. 

\begin{thm}[Purity--Independent points]\label{thm: purete2} 
Let $r\geq 2$ be an integer. For every integer $i$, $1\leq i \leq r$, let us consider a linear Mahler system   
\begin{equation}\stepcounter{equation}
\label{eq:mahleri}\tag{\theequation .$i$}
\left(\begin{array}{c} f_{i,1}(T_i\z_i) \\ \vdots \\ f_{i,m_i}(T_i\z_i) \end{array} \right) = 
A_i(\z_i)\left(\begin{array}{c} f_{i,1}(\z_i) \\ \vdots \\ f_{i,m_i}(\z_i) 
\end{array} \right) 
\end{equation}
where $A_i(\z_i)$ belongs to ${\rm GL}_{m_i}(\Q(\z_i))$, $\z_i:=(z_{i,1},\ldots,z_{i,n_i})$ is a tuple of 
indeterminates, and $T_i$ is an $n_i\times n_i$ matrix
with nonnegative  integer coefficients and with spectral radius 
$\rho(T_i)$. Let   $\balpha_i =(\alpha_{i,1},\ldots,\alpha_{i,n_i})\in (\Q^\star)^{n_i}$, 
$\mathcal E_i$ be a subtuple of 
$(f_{i,1}(\balpha_i),\ldots,f_{i,m_i}(\balpha_i))$,    
and $\mathcal E=(\mathcal E_1,\ldots,\mathcal E_r)$. 
Suppose that the two following conditions hold. 

\begin{enumerate}

\item[\rm{(i)}]   For every $i$, $\balpha_i$ is regular w.r.t. \eqref{eq:mahleri} and $(T_i,\balpha_i)$ is admissible. 

\item[\rm{(ii)}]   $\rho(T_1)=\cdots = \rho(T_r)$ and there is no nonzero tuple $(\bmu_1,\ldots,\bmu_r)\in \mathbb Z^N$, 
$N=n_1+\cdots+n_r$, such that $(T_1^k\balpha_1)^{\bmu_1}\cdots (T_r^k\balpha_r)^{\bmu_r}=1$, 
for all $k$ in an arithmetic progression.  

\end{enumerate}
Then
$$
{\rm Alg}_{\Q}(\mathcal E) = \sum_{i=1}^r {\rm Alg}_{\Q}(\mathcal E_i\mid \mathcal E)\,.
$$
\end{thm}

In other words, the only algebraic relations between the coordinates of $\mathcal E$ are those that 
can be trivially derived from the pure algebraic relations with respect to the coordinates of each $\mathcal E_i$. 

The first results dealing with values of Mahler functions at independent points are due to Mahler \cite{Ma30a} and are limited to linear independence over $\Q$.  
Some generalization are due to Kubota \cite{Ku77} and to Loxton and van der Poorten \cite{LvdP77II}.  
All these results are restricted to the study of several  inhomogeneous equations of order one.

\begin{rem}\label{rem: ReformulationTindependece}
Condition  (ii) is clearly satisfied when all the algebraic numbers $\alpha_{1,1},\ldots,\alpha_{r,n_r}$ 
are multiplicatively independent. 
\end{rem}

Let us turn to our third main result.  
It states that values at algebraic points of Mahler functions associated with  
sufficiently independent transformations  
\emph{always} behave independently.   As with Theorem \ref{thm: purete2}, the main advantage 
is that there is no need to 
check any kind of functional independence.   
Again, this result  
is expressed in terms of purity.

\begin{thm}[Purity--Independent transformations]
\label{thm: purity}
We continue with the notation of Theorem \ref{thm: purete2}. 
Suppose that the two following conditions hold.

\begin{enumerate}

\item[\rm{(i)}]  For every $i$, $\balpha_i$ is regular w.r.t. \eqref{eq:mahleri} and $(T_i,\balpha_i)$ is admissible.

\item[\rm{(ii)}]   The spectral radii $\rho(T_1),\ldots,\rho(T_r)$ are pairwise multiplicatively independent. 

\end{enumerate}
Then
$$
{\rm Alg}_{\Q}(\mathcal E) = \sum_{i=1}^r {\rm Alg}_{\Q}(\mathcal E_i\mid \mathcal E)\,.
$$
\end{thm}

In 1976, Kubota \cite{Ku76} and van der Poorten \cite{vdP76}, 
first envisaged the possibility of extending  Mahler's method in order to consider simultaneously 
several Mahler systems associated with independent transformations.  
In \cite{Ku76}, Kubota gave a sketch of proof in a very specific case and announced a paper on this problem, 
but the latter never appeared in print. 
Then Loxton and van der Poorten \cite{LvdP78} stated some related result, but the corresponding proof is  incomplete (see  \cite[p.\ 89]{Ni94}). 
In 1987, van der Poorten \cite{vdP87} made this guess more ambitious and  precise, 
pointing out several striking consequences that would follow from results he expected to 
prove in his collaboration with Loxton. However, these authors did not publish any new 
paper on this problem.  
In the end, only examples limited to the study of several inhomogeneous equations of order one 
have been worked out by Nishioka \cite{Ni94} and Masser \cite{Mas99}.    
In contrast, Theorem \ref{thm: purity}  
applies to arbitrary linear Mahler systems, 
and to a much larger class of transformation matrices and algebraic points.

Of course, Theorems \ref{thm: purete2} and \ref{thm: purity} are  
 strong statements about algebraic independence.   

\begin{coro}\label{cor:2}
We continue with the assumptions of Theorems \ref{thm: purete2} or \ref{thm: purity}. 
The following equality holds true:
$$
{\rm tr.deg}_{\Q}(\mathcal E) = \sum_{i=1}^r {\rm tr.deg}_{\Q}(\mathcal E_i)\,.
$$
\end{coro}

\begin{rem}
In geometric terms, Theorems \ref{thm: purete2} and \ref{thm: purity} can be rephrased by saying 
that the affine $\Q$-variety associated with the ideal 
${\rm Alg}_{\Q}(\mathcal E)$ is isomorphic to the cartesian product of the affine $\Q$-varieties 
associated with the ideals ${\rm Alg}_{\Q}(\mathcal E_i)$, $1\leq i\leq r$. Indeed, 
we prove that their coordinate rings are isomorphic. That is,  
$$
\frac{\Q[\X]}{ {\rm Alg}_{\Q}(\mathcal E)} \cong \frac{\Q[\X_1]}{{\rm Alg}_{\Q}(\mathcal E_1)} \otimes_{\Q} \cdots \otimes_{\Q} \frac{\Q[\X_r]}{{\rm Alg}_{\Q}(\mathcal E_r)}\, \cdot 
$$
\end{rem}
\subsection{Main new ingredients}\label{sec: ingredient}
As already mentioned, all previous results concerning the transcendence theory of linear Mahler systems 
in several variables are very much inspired by the early 
work of Mahler \cite{Ma30b}. 
We also start with the same initial strategy, but we add a number of fundamental new ingredients, including 
Hilbert's Nullstellensatz, tools from ergodic Ramsey theory, and a new vanishing 
theorem. 
 
 In all previous works, a crucial step consists in expressing the coordinates of the 
 iterated matrix  
$A_k(\z):=A(\z)A(T\z)\cdots A(T^{k-1}\z)$ associated with the Mahler system \eqref{eq:mahler}
in terms of linear combination some convergent power series of the form $g_i(T^k\z)$, 
possibly twisted by some multivariate 
exponential polynomials. 
This is really of great importance for one can then apply some vanishing theorems to the power series $g_i(\z)$. 
This step has gradually become more difficult in the aforementioned works,  
as the matrices under consideration have taken a more general form.   
Its complexity culminated in \cite{AF3}. 
Unfortunately, one cannot expect to find this kind of expression  
when $A(\z)$ is not regular singular in the sense of \cite[Definition 1.1]{AF3}. 
Hence this strategy suffers from an intrinsic limitation, which prevents from dealing with  
arbitrary Mahler systems. We overcome this main deficiency 
by defining the so-called \emph{relation matrices} in Section \ref{sec: matrix}. Their existence 
and main properties  
are obtained by means of Hibert's Nullstellensatz and the notion of piecewise syndetic set.   
Introducing these matrices is a cornerstone of the present work and the main novelty with respect to 
our two unpublished preprints 
\cite{AF3,AF4}.  

In order to apply our results to transformations $T$ and points $\balpha$ that are as general 
as possible, it is of great importance to prove suitable \emph{vanishing theorems}.  
That is, results that guarantee the nonvanishing of arbitrary multivariate analytic functions at special sets of points  
(typically, certain subsets of  $\{T^{k}\balpha,\, k\geq 0\}$). 
In the case of a single transformation, Masser \cite{Mas82} solved this problem 
in a rather definitive way.  
We note that Masser's vanishing theorem (in fact a refinement using the notion of piecewise syndetic set) is already strong enough to prove Theorems \ref{thm: permanence} and \ref{thm: purete2}.   
In fact, we only need the identity theorem for reproving Nishioka's theorem and Philippon's lifting theorem. 
Unfortunately, Masser's vanishing theorem  is not suited to 
deal with Mahler systems associated with independent transformations. 
First results towards this goal were proved 
by Nishioka \cite{Ni94} and, again, by Masser \cite{Mas99}. Unfortunately,  
they remain too restricted for proving Theorem \ref{thm: purity}.  
In 2005, Corvaja and Zannier \cite[Theorem 3]{CZ05} deduced from the subspace theorem a general  result 
concerning the vanishing at $\mathcal S$-units of analytic multivariate power series with algebraic coefficients. 
They already noticed that it could be relevant for Mahler's method. 
Using the flexibility of their result and the notion of piecewise syndetic set, we cook up 
in Section \ref{sec: vanishing} our own vanishing theorem, 
which is specifically shaped for our purpose.

\subsection{Relevance of Mahler's formalism} 

To end this section, we recall two major advantages that this multivariate 
formalism offers. 

First, adding variables makes it possible to deal with values of $M$-functions at 
\emph{different algebraic points}.  
Let us give a basic example. With the function $\mathfrak f(z)=\sum_{n=0}^{\infty} z^{2^n}$, 
we can associate  
the two variables linear $T$-Mahler system 
\begin{equation}\label{eq:ex}
\begin{pmatrix}
1 \\ \mathfrak f(z_1^{2}) \\ \mathfrak f(z_2^{2})
\end{pmatrix} =   
\begin{pmatrix} 
1 & 0 & 0\\
-z_1 & 1& 0 \\
-z_2 & 0 & 1 \\
 \end{pmatrix} \begin{pmatrix}
 1 \\ \mathfrak  f(z_1) \\  \mathfrak f(z_2)
 \end{pmatrix} \, , \quad \mbox {where }\; T=\begin{pmatrix}
2& 0\\
0 & 2 \\
 \end{pmatrix}\,.
\end{equation}
By Theorem \ref{thm:masser}, the point $\balpha:=(1/2,1/3)$ is regular w.r.t.\ \eqref{eq:ex}  
and the pair $(T,\balpha)$ is admissible. 
The key point is that the transcendence of $\mathfrak f(z)$ gives  
\emph{for free} the algebraic independence over 
$\overline{\mathbb Q}(z_1,z_2)$ 
of the functions $\mathfrak  f(z_1)$ and $\mathfrak f(z_2)$. 
By Corollary \ref{thm: Nishioka}, it follows that  
$\mathfrak f(1/2)$ and $\mathfrak f(1/3)$ are algebraically independent over $\Q$.   
This important principle really takes shape, and acquires great generality, 
with Theorems \ref{thm: main} and  \ref{thm: purete2}.  

The second advantage of Mahler's multivariate formalism comes from the possibility of dealing with  
 a much larger class of one-variable functions obtained by 
suitable specializations of Mahler functions in several variables. 
Mahler's favorite example is the family of the Hecke-Mahler functions   
$$
\mathfrak f_{\omega}(z)=\sum_{n=0}^{\infty}\lfloor n\omega\rfloor z^n \,,
$$ 
where $\omega$ is a quadratic irrational real number. Though  $\mathfrak f_{\omega}(z)$ 
is not  an $M$-function\footnote{This fact only very recently became known. It is an easy consequence of Theorem \ref{thm: purity}, but it can also be obtained by combining the results in \cite{ABS} and \cite{Du11}.},  we have that $\mathfrak f_{\omega}(z)= F_{\omega}(z,1)$, where 
$$
F_{\omega}(z_1,z_2)=\sum_{n_1=0}^{\infty} \sum_{n_2=0}^{\lfloor n_1\omega\rfloor}z_1^{n_1}z_2^{n_2}
$$ 
is a Mahler function in two variables.  
In another direction, Cobham \cite{Co68} proved that generating functions of morphic sequences 
are specializations of the form $F(z,\ldots,z)$ for some multivariate Mahler functions $F(z_1,\ldots,z_n)$.  
Some related applications of our main results can be found in \cite{AF4}.

\section{Notation}\label{sec: notation} 
We fix here some notation that we will use all along this paper. 
Let $d$ be a positive integer and $\balpha=(\alpha_1,\ldots,\alpha_d)\in (\mathbb C\setminus \{0\})^d$. 
If $T=(t_{i,j})_{1\leq i,j\leq d}$ is a $d\times d$ matrix with 
nonnegative integer coefficients, we let $\rho(T)$ denote its spectral radius. Furthermore, we recall that  
$$
T \balpha=(\alpha_1^{t_{1,1}}\alpha_2^{t_{1,2}}\cdots \alpha_d^{t_{1,d}},\ldots,
\alpha_1^{t_{d,1}}\alpha_2^{t_{d,2}}\cdots \alpha_d^{t_{d,d}})\,
$$ 
and that we let  $T(\x)$  denote the usual matrix product between $T$ and a column vector $\x\in\C^d$.  
If $T_1,\ldots,T_d$ are matrices, we let $T_1\oplus \cdots \oplus T_d$ denote the direct sum of these matrices. 
That is, 
$$
T_1\oplus \cdots \oplus T_d =   \left(\begin{array}{ccc} T_1 & & \\ & \ddots & \\ && T_d \end{array}\right) \, . 
$$
Let  
$\k=(k_1,\ldots,k_d)\in\mathbb Z^d$, then $\balpha^{\k}$ stands for $\alpha_1^{k_1}\cdots \alpha_d^{k_d}$, so that $(T\balpha)^\k  =\balpha^{\k T}$. 
Given a $d$-tuple of natural numbers 
$\k=(k_1,\ldots,k_d)$, we set $\vert \k\vert =k_1+\cdots +k_d$. 
We use the same symbol $\Vert\cdot\Vert$ to denote both the maximum norm of $\C^d$ and the maximum norm of $\M_d(\C)$. We let $\vert \xi\vert$ denote the module of the complex number $\xi$. 

Let $\lambd= (\lambda_1,\ldots,\lambda_d) \in \N^d$ and $\gamm = (\gamma_1,\ldots,\gamma_d) \in \N^d$. We define a partial order on $\N^n$ by setting 
$\lambd \leq \gamm$ if $\lambda_i \leq \gamma_i$, $1\leq i\leq d$. 
We also set  
$$
\binom{\lambd}{\gamm} = \prod_{i=1}^d \binom{\lambda_i}{\gamma_i}\, ,
$$
the product of binomial coefficients associated  with each coordinate of  $\lambd$ and $\gamm$.
Given  a positive integer $h$, a matrix $M=(m_{i,j})_{1 \leq i,j \leq h}$  with coefficients in some ring, and 
a matrix $\bmu=(\mu_{i,j})_{1\leq i,j\leq h}$  with nonnegative integer coefficients, we set 
$$
M^\bmu = \prod_{1\leq i,j,\leq h} m_{i,j}^{\mu_{i,j}}\, .
$$

We use the standard Landau notation $\mathcal O$. We also use the notation $\gg$ as follows.  
Writing that some property holds true for all integers 
$\lambda_1\gg \lambda_2,\lambda_3$ means 
that the corresponding property holds true for all $\lambda_1$ that is sufficiently large w.r.t.\ 
$\lambda_2$ and $\lambda_3$, while writing that some property holds true for all integers 
$\lambda_1\gg \lambda_2\gg \lambda_3$ means 
that the corresponding property holds true for all $\lambda_1$ that is sufficiently large w.r.t.\ $\lambda_2$, 
assuming that $\lambda_2$ is itself sufficiently large w.r.t.\ $\lambda_3$. 

Given a positive real number $R$ and $\balpha \in \C^d$, we let
$$
\mathcal D(\balpha,R)=\{\thet \in \C^d\ : \ \Vert \thet - \balpha \Vert < R\}\, 
$$
denote the  open polydisc of center $\balpha$ and radius $R$.  
By definition, an element $g \in \Q\{\z\}$ has a unique expansion of the form  
$$
g(\z)=\sum_{\lambd \in \N^d} g_\lambd \z^\lambd \, ,
$$
which converges in some neighborhood of the origin. 
The \emph{radius of convergence} of $g$ is defined as the supremum of the positive real numbers $R$ 
such that the power series defining $g$ is convergent on $\mathcal D(\boldsymbol{0},R)$. 
By \cite[Proposition 2.2]{Cartan}, when the radius of convergence of $g$ is finite and equal to $R_0$, the power series defining $g$ is absolutely convergent on $\mathcal D(\boldsymbol{0},R_0)$.  
By specialization, we deduce from the Cauchy-Hadamard theorem that 
if  $g(\z)=\sum_{\lambd \in \N^d} g_\lambd \z^\lambd\in\Q\{\z\}$, then 
\begin{equation}
\label{eq:CauchyHadamard}
\vert g_\lambd \vert = \mathcal O\left(R^{-\vert \lambd \vert }\right)\, ,
\end{equation}
for all positive real numbers $R$ smaller than the radius of convergence of  $g$.

We let $H(\cdot)$ denote the \emph{absolute Weil height} over the projective space $\mathbb P^d(\Q)$.  
Given $\boldsymbol \beta = (\beta_1,\ldots,\beta_d) \in \Q^d$, we also write $H(\boldsymbol \beta)$ instead of $H(\beta_1:\cdots:\beta_d:1)$. We will only use basic properties of the Weil height and we refer the interested reader to \cite[Chapter 3]{Miw} for more details.


\section{Admissibility conditions}\label{sec: admissibility}

A well-known feature of Mahler's method is that, 
independently of the choice of the matrix $A(\z)$ defining the system \eqref{eq:mahler}, 
some unavoidable restrictions on the transformation $T$ and 
on the point $\balpha$ are required.

\begin{defi}\label{def:admissible}
Let $T$ be an $n\times n$ matrix with nonnegative integer coefficients and $\balpha\in(\Q^\star)^n$. 
The pair $(T,\balpha)$ is said to be \emph{admissible} if there exist two real numbers $\rho>1$ and 
$c>0$ such that the following three conditions hold.  

\begin{enumerate}[label=(\alph*)]
\item \label{condition: A} 
The coefficients of the matrix $T^k$ belong to $\mathcal O(\rho^k)$.   

\item \label{condition: B} Set $T^k{\boldsymbol \alpha}:=(\alpha_1^{(k)},\ldots,\alpha_n^{(k)})$. 
Then 
$\log \vert \alpha_i^{(k)}\vert \leq -c\rho^k$, 
for every integer $i$, $1\leq i\leq n$, and all sufficiently large integers $k$.  

\item \label{condition: C} If $f({\z})\in \mathbb C\{{\z}\}$ is nonzero, then there 
are infinitely many integers $k$ such that $f(T^k\boldsymbol \alpha)\not=0$. 
\end{enumerate}
\end{defi}


The strength of our results strongly depends on our ability to 
provide simple and natural conditions 
that imply   
Conditions (a), (b), and (c). The latter are in fact necessary 
to apply Mahler's method (see \cite{Ma29}). 
Though they appear naturally in proofs, 
it is not that easy, at first glance, to see how to check them. 
We provide here a simple characterization of matrices and algebraic points 
satisfying these conditions, gathering results of Kubota \cite{Ku77}, 
Loxton and van der Poorten 
\cite{LvdP77,LvdP77II}, and mainly Masser \cite{Mas82}.

\begin{defi} 
Let $T$ be an $n\times n$ matrix with nonnegative integer coefficients and with spectral radius $\rho(T)$. 
We say that $T$ belongs to the class $\M$ if it satisfies the following three conditions.  

\smallskip

\begin{itemize}
\item[(i)] It is nonsingular. 

\smallskip

\item[(ii)] None of its eigenvalues  are roots of unity.

\smallskip

\item[(iii)] There exists an eigenvector with positive coordinates associated with the eigenvalue $\rho(T)$. 
\end{itemize}
\end{defi}

In particular, if $T$ belongs to the class $\M$, then $\rho(T) > 1$. 

\begin{rem}
Let us consider $r$ Mahler systems associated with transformations $T_1,\ldots,T_r$ in $\M$, all  
having the same spectral radius $\rho$. Then the direct sum $T_1 \oplus  \cdots \oplus T_r$ 
also belongs to the class $\M$.  More generally, given $r$ Mahler 
systems, associated with transformation matrices $T_1,\ldots,T_r\in \M$ with pairwise multiplicatively dependent spectral radius, it is possible to gather them into a larger Mahler system whose transformation matrix 
also belongs to the class $\M$, and then to apply  Theorem \ref{thm: permanence}.  
\end{rem}

Given a one-variable Mahler system associated with a matrix $A(z)$, we can  
consider the same  
system twice but with different variables. 
That is, the system associated with the matrix
$$
\left(\begin{array}{cc} A(z_1) & 0 \\ 0 & A(z_2) \end{array} \right) \,.
$$
This shows that some kind of minimal independence between the coordinates of the point 
$\balpha=(\alpha_1,\alpha_2)$ 
is required in order to apply Mahler's method. Typically, we cannot consider a point 
of the form $(\alpha,\alpha)$ in that case. 

\begin{defi}\label{def: independent}
A point $\balpha\in (\Q^\star)^n$ is  said to be {\it $T$-independent} 
if there is no nonzero $n$-tuple of integers $\bmu$ for which $(T^k\balpha)^\bmu=1$ for all $k$ 
in an arithmetic progression.  
\end{defi} 

\begin{rem}
According to Definition \ref{def: independent}, 
Condition (ii) of Theorem \ref{thm: purete2} is equivalent to the fact that 
the point $\balpha:=(\balpha_1,\ldots,\balpha_r)$ is $T$-independent with respect to the 
direct sum $T=T_1 \oplus  \cdots \oplus T_r$.
\end{rem}

With these definitions, we have the following characterization 
of admissibility, which makes our main results very convenient to apply.

\begin{thm}
\label{thm:masser}
Let $T$ be an $n\times n$ matrix with nonnegative integer coefficients and $\balpha\in (\Q^\star)^n$.  
Then the pair $(T,\balpha)$ is admissible if and only if  
 $T$ belongs to the class $\M$,  $\lim_{k\to \infty}T^k\balpha=\boldsymbol 0$, and $\balpha$ is $T$-independent. 
 \end{thm}

\begin{rem}
There is no difficulty in checking whether or not a matrix belongs to the class $\M$.  
Furthermore, if $\alpha_1,\ldots,\alpha_n\in \mathbb C \setminus\{0\}$ are multiplicatively independent 
complex numbers,  
then $\balpha=(\alpha_1,\ldots,\alpha_n)$ is \emph{a fortiori} $T$-independent.  
\end{rem}


\begin{proof}[Proof of Theorem \ref{thm:masser}] 
We first prove the reverse direction. 
Let us assume that $T$ belongs to $\M$, $\balpha\in(\Q^\star)^n$ is $T$-independent, and 
that $\lim_{k\to\infty}T^k\balpha=0$. Then there exists $k_0$ such 
that $\Vert T^{k_0}\balpha\Vert< 1$. We observe that if the pair $(T,T^{k_0}\balpha)$ is admissible, 
then so is the pair $(T,\balpha)$. Thus, we can assume without any loss of generality that 
$\Vert \balpha\Vert< 1.$ We let $\x$ 
denote the transpose of the vector 
$$
 \left( - \log |\alpha_{1}|,\ldots,-\log |\alpha_{n}| \right) \, ,
$$
whose coordinates are all positive. By assumption, $T$ has a positive eigenvector associated with the eigenvalue $\rho(T)$. 
Let us choose such an eigenvector $\bmu$ whose coordinates are all smaller than those of $\x$. 
We have
$$
- \log \Vert T^{k}\balpha\Vert  = \Vert  T^{k}(\x) \Vert  = \Vert T^{k}(\bmu) 
+ T^{k}(\x-\bmu)\Vert  > \Vert  T^{k}(\bmu)\Vert  = \rho(T)^{k}\Vert \bmu\Vert  \, ,
$$
for all $k\in \mathbb N$, because $T^k(\x-\bmu)$ has positive coordinates. Condition (b) is thus satisfied with $\rho = \rho(T)$. Following Loxton and van der Poorten \cite{LvdP77},  Condition (a) is satisfied for 
the matrix $T$ belongs to $\M$.  
Finally, Masser vanishing theorem \cite{Mas82} implies that Condition (c) holds since $\balpha$ is $T$-independent. Hence, the 
pair $(T,\balpha)$ is admissible.  

Now, we prove the forward direction. Let $(T,\balpha)$ be an admissible pair. We first note that Condition (b)  
implies that $\lim_{k\to \infty} T^k\balpha = 0$. 

Replacing $\balpha$ by $T^k\balpha$ for some $k$ if 
necessary, we can thus assume without any loss of generality that $\Vert \balpha \Vert < 1$. By \cite{Ku77}, 
Condition (c) 
implies that the matrix $T$ is nonsingular and that none of its eigenvalues is a root of unity. Hence 
its spectral radius $\rho(T)$ is larger than $1$. 
Let us prove that there exists an eigenvector with positive coordinates associated with the eigenvalue 
$\rho(T)$.  
Let $\rho$ be as in Definition \ref{def:admissible}.  We first infer from Conditions (a) and (b) that 
$\rho = \rho(T)$. Since the coefficients of $T$ are nonnegative integers, for every eigenvalue  $\rho'$ 
with $\vert \rho'\vert =\rho$, there is a root of unity $\mu$ such that $\rho'=\mu \rho$ (see, for instance, \cite[Theorem 2, p.\ 65 \& Chapter III \textsection 4]{Grantmacher}). Replacing $T$ by some power if necessary, we can assume 
that $\rho$ is larger than every other eigenvalue of $T$. 
Let $E_\rho$ denote the eigenspace associated with $\rho$. Condition (a) implies that the characteristic space associated with $\rho$ is equal to $E_\rho$, since otherwise the sequence $T^k/\rho^k$ would not be bounded.  Hence $E_\rho$ has a $T$-invariant complement, say $E_\rho^\perp$. 
From Condition (b), we infer  the existence of a positive real number $\gamma$ such that every coordinate of $T^k(\x)$ is larger than $\gamma \rho^k$. Set 
$$
\x = \boldsymbol e + \boldsymbol e^\perp,\, \text{ where } \boldsymbol e \in E_\rho \text{ and } \boldsymbol e^\perp \in E_\rho^\perp\,.
$$
Suppose that there is no vector in $E_\rho$ with positive coordinates. Then, for some $j$, the $j$th 
coordinate of $\boldsymbol e$ is nonpositive. Since $T^k(\x)=\rho^k \boldsymbol e + T^k(\boldsymbol e^\perp)$, 
we deduce that  the $j$th coordinate of $T^k(\boldsymbol e^\perp)$ is larger than $\gamma \rho^k$. 
Since all eigenvalues of $T$ on $E_\rho^\perp$ are smaller than $\rho$, we obtain a contradiction. 
This shows that $T$ belongs to the class $\mathcal M$.

 Finally, $\balpha$ is $T$-independent. Indeed, 
otherwise there would exist two tuples of nonnegative integers $(s_1,\ldots, s_n)$ and $(t_1,\ldots, t_n)$, not both zero, such that $P(T^k\balpha)=0$ for infinitely many $k$, where $P(\z)= z_1^{s_1}\cdots z_n^{s_n}$ - $z_1^{t_1}\cdots z_n^{t_n}$, providing a contradiction with Condition (c). 
\end{proof}


\section{A new vanishing theorem}\label{sec: vanishing}

As already mentioned, it is of great importance to find natural conditions that ensure   
nonvanishing properties similar to Condition (c) in Definition \ref{def:admissible}.  
Of course, our goal is to obtain a \emph{vanishing theorem} 
that can be applied to transformation matrices and points which are as general as possible. 
Our contribution to this problem is Theorem \ref{thm:lemmedezero}.

In the framework of Mahler's method, several vanishing theorems have been formulated 
by saying that a nonzero multivariate power series cannot vanish  
at all points in some well-structured large sets. The latter are obtained by iteration of the transformation 
matrix and usually involve arithmetic progressions. In order to prove our main theorems, we need 
to replace these \emph{well-structured sets} by sets which remain large but offer much more flexibility.  
We use the notion of  piecewise syndetic set, which is classical in Ramsey theory,  
especially in its ergodic counterpart. As we just said, it can be though of as a notion of largeness 
for subsets of $\mathbb N$. Furthermore, Brown's lemma (see (ii) in Lemma \ref{lem:syndetique}) 
shows that such sets are partition regular, and thus much more robust in terms of partitions 
than arithmetic progressions. 

\begin{defi}\label{def:full}
A set $\mZ\subset\N$ is said to be \emph{piecewise syndetic} if there exists a natural number 
$B\geq 1$  such that  for any given integer $C\geq 2$  there exist $l_1 < \cdots < l_C$ in $\mZ$ 
such that
$$
l_{i+1} - l_i \leq B, \qquad 1 \leq i < C\, .
$$
In this case, we say that $B$ is a {\it bound}  for $\mZ$. A set $\mZ\subset \N$ is said to be 
\emph{negligible} if it is not piecewise syndetic, while it is said to be \emph{full} if its complement is negligible. 
\end{defi}

Let us recall that a subset of $\mathbb N$ is  said to be {\it syndetic}, or sometimes 
\emph{relatively dense}, if it has bounded gaps. A subset of $\mathbb N$ is said to be \emph{thick} 
if it contains arbitrarily long intervals. Thus piecewise syndetic sets are those that 
can be obtained as the intersection of  a syndetic set and a thick set. In the rest of this section, 
as well as all along Section \ref{sec: matrix}, we will use heavily the following results. 

\begin{lem}\label{lem:syndetique}
Let $\mZ \subset \N$ be a piecewise syndetic set with bound $B$. 
Then the following properties hold.  

\begin{enumerate}[label=(\roman*)]
\item[{\rm (i)}] If $\mZ \subset \mZ' \subset \N$, then $\mZ'$ is also piecewise syndetic.
\item[{\rm (ii)}] If $\mZ \subset \cup_{i=1}^s \mZ_i$, then at least one of the sets $\mZ_i$  
is piecewise syndetic.
\item[{\rm (iii)}] Let $l_0$ be a natural number. The set 
$$
\mZ_0:=\left\{l \in \mathcal Z : (l+\{l_0,\ldots,l_0+B\}) \cap \mZ \neq \emptyset \right\}
$$
is piecewise syndetic.
\item[{\rm (iv)}] The set $\mZ$ contains arbitrarily long arithmetic progressions. 
\end{enumerate} 
\end{lem}

\begin{proof}
The point (i) immediately follows from the definition, while points (ii) and (iv)  correspond to 
classical results respectively known as  Brown's lemma (see \cite{Brown}) and Szemer\'edi's 
theorem \cite{Sz}. Let us prove (iii). Let $l_0$ and $C$ be two natural numbers and let $a$ 
be the smallest integer such that $aB \geq l_0$. 
Since $\mZ$ is piecewise syndetic, there exists a sequence $l_1 < l_2 < \cdots < l_{C+aB}$ 
of integers in $\mZ$ such that $l_{i+1}-l_i < B$. Let $i\in \{1,\ldots,C\}$. 
Then, $l_i+l_0\leq l_i+aB \leq l_{i+aB}$. 
There thus exists an integer $j \leq aB$ such that $l_i+l_0 \leq l_{i+j} \leq l_i+l_0+B$. 
Hence 
$l_i \in \mZ_0$. Thus,  $l_1,\ldots,l_C$ 
all belong to the set $\mZ_0$, which proves that this set is piecewise syndetic.  
\end{proof}

Part of Lemma \ref{lem:syndetique} can be naturally rephrased as follows. 

\begin{lem}\label{lem:syndetique2}
The following properties hold.  

\begin{enumerate}[label=(\roman*)]
\item[{\rm (i)}] A finite union of negligible sets is negligible. 
\item[{\rm (ii)}] A finite intersection of full sets is full. 
\item[{\rm (iii)}] If $\mZ_1$ is full and $\mZ_2$ is negligible, then $\mZ_1\setminus \mZ_2$ is full. 
\end{enumerate} 
\end{lem}

\begin{proof}
Point (i) follows directly from Point (ii) of Lemma \ref{lem:syndetique}.  Let $\mZ_1,\ldots,\mZ_r$ be full sets. 
By (i), we obtain that $(\cap \mZ_i)^c=\cup (\mZ_i^c)$ is negligible, which proves (ii). 
Let us prove (iii). By assumption, $\mZ_1^c$ is negligible. By (i), the set $\mZ_2 \cup \mZ_1^c$ is also negligible. Since $\mZ_2 \cup \mZ_1^c=(\mZ_1 \setminus \mZ_2)^c$, we obtain that  
$\mZ_1\setminus \mZ_2$ is full, as wanted.  
\end{proof}

We are now ready to state our vanishing theorem. 

\begin{thm}
\label{thm:lemmedezero}
Let $T_1,\ldots,T_r$ be matrices in the class $\M$ whose spectral radii $\rho(T_1),\ldots,\rho(T_r)$ 
are pairwise multiplicatively independent.   
Let $n_i$ denote the size of the matrix $T_i$ and set $N:=\sum_{i=1}^r n_i$.  
Set 
\begin{equation}\label{eq:Theta}
\Theta:=\left( \frac{1}{\log \rho(T_1)},\ldots,\frac{1}{\log \rho(T_r)}\right) \, .
\end{equation}
For every $l \in \mathbb N$, we let 
$\k_l:=(k_{l,1},\ldots,k_{l,r})$ denote a $r$-tuple of positive integers. Let us assume that  
\begin{equation}\label{eq:bounded}
\Vert \k_l -l \Theta \Vert = \mathcal O(1) \, .
\end{equation}
Let $\balpha:=(\balpha_1,\ldots,\balpha_r)\in(\Q^\star)^N$ be
such that the pair $(T_i,\balpha_i)$ is admissible for every $i$, and   
let  
$g(\z) \in \Q\{\z\}$ be nonzero. Then the set
$$
\left\{ l \in \mathbb N : g(T_1^{k_{l,1}}\balpha_1,\ldots,T_r^{k_{l,r}}\balpha_r) = 0 \right\}
$$
is negligible. 
\end{thm}

 Applying Mahler's method to several Mahler systems requires some uniform speed of 
 convergence to the origin for the orbits of each algebraic point $\balpha_i$ under the matrix 
 transformations $T_i$.  As noticed by van der Poorten \cite{vdP87}, one way to overcome 
 this difficulty is to iterate  each transformation $T_i$  $k_i$-times, and to choose the iteration vector 
 $\k =(k_1,\ldots,k_r)$ so that asymptotically the matrices $T_i^{k_i}$ have essentially the same radius of convergence. As explained by Lemma \ref{lem:uniformiteconvergencefamilles}, 
 this forces us to consider only iteration 
 vectors $\k$ that remain at bounded distance from the real line $\mathbb R\Theta$, where $\Theta$ 
 is defined by \eqref{eq:Theta}. This explains why the assumption \eqref{eq:bounded} 
 is natural in this framework.  In the rest of this section, we set 
 $$
T_{\k} :=  T_1^{k_{1}} \oplus \cdots \oplus T_r^{k_{r}} \; \mbox{ and }\; T_{\k}\balpha:=(T_1^{k_{1}}\balpha_1,\ldots,T_r^{k_{r}}\balpha_r).
$$

 \begin{lem}\label{lem:uniformiteconvergencefamilles}
Let $T_1,\ldots,T_r$, $\balpha_1,\ldots,\balpha_r$ be as in Theorem \ref{thm:lemmedezero} 
and let $(\k_l)_{l \in \N}$ be an arbitrary sequence with values in  $\N^r$. 
Then the following properties are equivalent.
\begin{itemize}
\item There exist a real numbers $\rho>1$ and  $c>0$ such that
$$
\left\Vert T_{\k_l} \right\Vert = \mathcal O(\rho^l)\, \mbox{ and } \log \left\Vert T_{\k_l}\balpha\right\Vert \leq -c\rho^l\, \mbox{ for all large enough $l$.}
$$
\item There exists a real number $\lambda$ such that $\Vert \k_l - \lambda l \Theta \Vert = \mathcal O(1)$, 
where $\Theta$ is defined as in \eqref{eq:Theta}.
In that case, one has $\rho = e^\lambda$. In particular, if $\lambda=1$, then $\rho = e$.
\end{itemize}
\end{lem}

\begin{proof}
Let us assume that there exists a real number $\lambda$ such that 
$\Vert \k_l - \lambda l \Theta \Vert = \mathcal O(1)$.
There thus exists a positive real number $B$, such that, for every $l\in \N$ and $i$, $1\leq i \leq r$, we have 
$k_{i,l} =\lambda l/\log(\rho(T_i)) + \epsilon(i,l)$, for some  real number $\epsilon(i,l)$ with 
$\vert \epsilon(i,l) \vert \leq B$. 
By Conditions (a) and (b) in Definition 
\ref{def:admissible}, we  have, on the one hand, that 
$$
\left\Vert T_i^{k_{i,l}} \right\Vert = \mathcal O\left(\rho(T_i)^{k_{i,l}}\right) = 
\mathcal O\left(\rho(T_i)^{\lambda l/\log(\rho(T_i))}\right) = \mathcal O\left(e^{\lambda l}\right)\, ,
$$
while, on the other hand,  
$$
\log \left\Vert  T_i^{k_{i,l}}\balpha_i \right\Vert \leq -c_i\rho(T_i)^{k_{i,l}} \leq -c'_ie^{\lambda l}\, ,
$$
for all large enough $l$, where $c_i$ and $c'_i$ are positive real numbers. Setting $c:=\min_i\{c'_1,\ldots,c'_r\}$, 
we obtain the desired estimate.

\medskip Now, let us assume that 
$$
\left\Vert T_{\k_l} \right\Vert = \mathcal O(\rho^l)\, \mbox{ and } 
\log \left\Vert T_{\k_l}\balpha\right\Vert \leq -c\rho^l \, \mbox{ for all large enough $l$.} 
$$
Set $\lambda := \log(\rho)$ and let $i$ be an integer with $1\leq i \leq r$. 
Since the pair $(T_i,\balpha_i)$ satisfies Conditions (a) and (b) of Definition \ref{def:admissible}, 
we have 
$$
\left\Vert T_i^{k_{i,l}} \right\Vert= \mathcal O\left(\rho(T_i)^{k_{i,l}}\right)\, \mbox{ and }\log \left\Vert T_i^{k_{i,l}}\balpha_i \right\Vert \leq -c_i\rho(T_i)^{k_{i,l}} \, ,
$$
for some positive real number $c_i$ and all large enough $l$. There thus exist two positive real number $\kappa_i$ and $\gamma_i$ 
such that 
$$
\kappa_i e^{\lambda l} \leq \rho(T_i)^{k_{i,l}} \leq \gamma_i e^{\lambda l} \, ,
$$
for all $l \in \N$. Taking the logarithm, it follows that 
$$
\log(\kappa_i) + \lambda l \leq \log(\rho(T_i)) k_{i,l} \leq \log(\gamma_i) + \lambda l\, .
$$
Dividing by $\log(\rho(T_i))$, we see that $k_{i,l}$ remains at bounded distance from 
$\lambda l /\log(\rho(T_i))$. This ends the proof. 
\end{proof}
 
 Before proving Theorem \ref{thm:lemmedezero}, we 
 need the two following auxiliary results.
The proof of Theorem \ref{thm:lemmedezero} is based on a vanishing theorem due to 
Corvaja and Zannier \cite[Theorem 3]{CZ05}. The latter states that if the set of zeros of a multivariate   
analytic function with algebraic coefficients  contains an infinite sequence of $\mathcal S$-unit points whose height does not grow too fast, then these points all belong to 
a finite number of translates of tori. The goal of the following two lemmas is to show that most points of the form $T_{\k_l}\balpha$ , $l\in \N$, avoid these tori.

\begin{lem}
\label{lem:admislocglob}
We continue with the assumptions of Theorem  \ref{thm:lemmedezero}. 
Then, 
for every nonzero integer $N$-tuple $\bmu$, the set  
$$
\mZ := \left\{ l \in \N : (T_{\k_l}\balpha)^{\bmu} = 1 \right\}\, 
$$
is negligible.
\end{lem}

\begin{proof}
We argue by contradiction, assuming that $\mZ$ is piecewise syndetic. 
For every pair of nonnegative integers $(l,m)$, with $m >0$, we define the $r$-tuple of natural numbers 
$\be:=\be(l,m)$ by  
\begin{equation}
\label{eq:sturm}
\be =\k_{l+m}-\k_l \,
\end{equation}
and we set $\E := \{\be(l,m) :\ l,\, m \in \N  \}$. 
Since by \eqref{eq:bounded} we have $\k_l=l\Theta + \mathcal O(1)$, we obtain that 
\begin{equation}\label{eq:equivalencekl1}
\be(l,m) = m\Theta + \mathcal O(1) \, ,
\end{equation}
which shows that $\E$ is infinite. However, given any positive integer $m_0$, the set 
$\{\be(l,m_0) : \ l\in \N\}$ is finite. 

Let us remark that given any pair 
$(\beta_1,\beta_2)$ of nonzero complex numbers that are not roots of unity, and any pair of natural numbers  
$(i,j)$, $1\leq i<j\leq r$, the set 
$$
\E_1 :=\left\{\be=(e_1,\ldots,e_r) \in \E : \beta_1^{e_i}=\beta_2^{e_j}\right\}
$$ 
is finite.  
Indeed, the set of natural numbers $u$ such that there exists a natural number $v$ for which 
$\beta_1^u=\beta_2^v$ is an ideal of $\Z$. Let $u_0\geq 0$ be a generator of this ideal and let  
$v_0 \in \N$ be such that $\beta_1^{u_0}=\beta_2^{v_0}$. For every $(e_1,\ldots,e_r) \in \E_1$, 
there exists an integer $a$ such that $e_i=a u_0$. 
We obtain that
$$
\beta_2^{e_j}=\beta_1^{e_i}=\beta_1^{au_0}=\beta_2^{av_0}\, .
$$
Since $\beta_2$ is nonzero and is not a root of unity, we have $e_j=a v_0$. 
Hence $e_i/e_j=u_0/v_0 \in \mathbb Q$. Since 
$e_i= k_{i,l+m}-k_{i,l}$ and $e_j=k_{j,l+m}-k_{j,l}$, we get that 
$$
\frac{k_{j,l+m}-k_{j,l}}{k_{i,l+m}-k_{i,l}}=\frac{u_0}{v_0}\in \mathbb Q \, .
$$
Now let us assume by contradiction that  $\E_1$ is infinite.  Then, there exist arbitrarily large integers 
$m$ with this property. 
Letting $m$ tends to infinity, we deduce from \eqref{eq:equivalencekl1} that 
the ratio $\log  \rho(T_i)/\log \rho(T_j)$ is rational. This provides a contradiction 
since by assumption $\rho(T_i)$ and $\rho(T_j)$ are multiplicatively independent. 
Hence $\E_1$ is finite.

Let us recall that, by assumption, none of the eigenvalues of the matrices $T_i$ 
is equal to zero or to a root of unity. The previous reasoning shows that 
 there exists a positive integer $m_0$ such that, for every $m \geq m_0$, every $l \in \N$, 
 every eigenvalue $\lambda_i$ of $T_i$, and every eigenvalue $\lambda_j$ of $T_j$, $i \neq j$, 
 we have  
\begin{equation}\label{eq: valpropres}
\lambda_i^{e_i} \neq \lambda_j^{e_j} \, ,
\end{equation}
where $\be=\be(l,m)=(e_1,\ldots,e_r)$. 
For such a vector $\be$,  set 
\begin{equation*}
T_\be:= T_1^{e_1} \oplus \cdots \oplus T_r^{e_r} \, .
\end{equation*}
Given a vector space $V\subset \C^N$, we have 
$$
V \subset \bigoplus_{i=1}^r \iota_i \circ \pi_i(V)\, ,
$$
where we let $\pi_i : \C^N=\C^{n_1+\cdots + n_r} \rightarrow \C^{n_i}$ 
denote the projection defined by $\pi_i(\x_1,\ldots,\x_r)=\x_i$ and $\iota_i: \C^{n_i} \rightarrow \C^N$ be such that $\pi_i\circ \iota_i$ is the identity on $\C^{n_i}$.  
By \eqref{eq: valpropres}, if $V$ is invariant under $T_\be$, then 
\begin{equation}\label{eq: vi}
V = \bigoplus_{i=1}^r \iota_i \circ \pi_i(V)\, .
\end{equation}

We are now ready to proceed with the proof of the lemma. 
Let us consider the column vector $\x$ whose transpose is the vector 
$$
(\log \alpha_{1,1},\log\alpha_{1,2},\ldots,\log\alpha_{1,n_1},\log\alpha_{2,1},\ldots,\log \alpha_{r,n_r})\ ,
$$
where $\log$ stands for a suitable determination of the logarithm (that is, the corresponding branch 
cut avoids all the coordinates of the points $T_{\k_l}\balpha$, $l \in \N$). 
By assumption, we have 
$$
\langle \bmu \, ,\,T_{\k_l}(\x)\rangle = 0 \, 
$$
for all $l \in \mZ$, where we let $\langle \,, \rangle$ denote the usual scalar product. 
Let  $U$ denote the orthogonal complement to the vector $\bmu$ in $\C^N$. 
This is a proper subspace of $\C^N$ defined over $\mathbb Q$, which contains all  vectors 
 $T_{\k_l} (\x)$, $l \in \mZ$. 
 Given $\mZ' \subset \mZ$, we let $U(\mZ')$ denote the smallest vector subspace of $\C^N$ defined over $\mathbb Q$  
 and containing all 
$T_{\k_l} (\x)$, $l \in \mZ'$. 
It follows that $U(\mZ) \subset U$. Furthermore, if $\mZ'' \subset \mZ'$, then $U(\mZ'') \subset U(\mZ')$. 
The subspace $U(\mZ)$ having finite dimension, there exists a subset $\mZ_1 \subset \mZ$ 
that is piecewise syndetic, 
and such that for all piecewise syndetic set $\mZ' \subset \mZ_1$,  one has 
$U(\mZ')=U(\mZ_1)$. 
Let $B$ denote a bound for $\mZ_1$ and set 
$$
\E_0 := \{ \be(l,m) : m\in [m_0,m_0+B],l\in\mZ_1,l+m \in \mZ_1\}\, , 
$$
where $m_0$ is defined as in the first part of the proof (just before \eqref{eq: valpropres}). 
This is a finite set. 
Let 
$$
\mZ_2 := \{l\in\mZ_1 : \exists m \in [m_0,m_0+B] \mbox{ such that } l+m\in\mZ_1\}\,.
$$
By Lemma \ref{lem:syndetique}, the set $\mZ_2$ is piecewise syndetic. 
Now, given $\be=\be(l,m) \in \E_0$, we set 
$$
\mZ_\be :=\{l \in \mZ_2 : T_{\be}(T_{\k_{l}}(\x))=T_{\k_{l+m}}(\x)\in U(\mZ_1)  \} \, .
$$  
If $l\in\mZ_2$, then there exists $m\in [m_0,m_0+b]$ such that $l\in\mZ_{\be}$ for some $\be=\be(l,m)$. 
Hence, $\mZ_2 \subset \cup_{\be \in \E_0} \mZ_\be$. Since $\mZ_2$ is piecewise syndetic, Lemma \ref{lem:syndetique} 
ensures the existence of $\be \in \E_0$  such that $\mZ_{\be}$ is piecewise syndetic. Furthermore, $\mZ_{\be}\subset \mZ_1$.  
Thus we obtain that 
$$
U(\mZ_{\be})= U(\mZ_1) \, .
$$
Hence, the vector space $U(\mZ_1)$ is invariant under $T_{\be}$. Indeed,  if $T_{\k_{l}}(\x)\in U(\mZ_1)=U(\mZ_{\be})$, then 
$T_{\be}(T_{\k_{l}}(\x))\in U(\mZ_1)$.  
By \eqref{eq: vi}, there is a decomposition of the form 
$$
U(\mZ_1) = \bigoplus_{i=1}^r \iota_i (U_i)\, , 
$$
where, for every $i$, $U_i = \pi_i(U(\mZ_1)) \subset \C^{n_i}$ is a $T_i^{e_i}$-invariant vector space  
defined over $\mathbb Q$.  
Since $U(\mZ_1)$ is a proper subspace of $\C^N$, there exists $i$, $1 \leq i \leq r$, such that 
$U_{i}$ is a proper subspace of $\C^{n_{i}}$. This vector space being defined over $\mathbb Q$, it has  a 
nonzero vector 
$\bnu_0 \in \Z^{n_{i}}$ in its orthogonal complement.  We thus have 
\begin{equation}\label{eq:produitscalaireUnematrice}
\langle \bnu_0\, ,\, T_{i}^{e_{i}k_{i,l}}(\x_i)\rangle = 0\,,
\end{equation} 
for all $l \in \mZ_1$, where $\x_i:=\pi_i(\x)$.    
The set $\mZ_1$ being piecewise syndetic, we infer from the definition of the sequence $\k_l$  
that the set $\mZ_2:= \{k_{i,l} : l\in \mZ_1\}$ 
is also piecewise syndetic.  
By property (iv) of Lemma \ref{lem:syndetique}, it contains arbitrarily long arithmetic progressions. 
Let us consider an arithmetic progression of length $n_i$ in $\mZ_2$, say 
$$
a,a+b,a+2b,\ldots,a+(n_i-1)b\, ,
$$
where $a,b \in \N$. 
Let us also  consider the sequence of  vector spaces 
$$V_0\subset \cdots \subset V_{n_i-1}\subset \bnu_0^\perp$$ 
defined by 
$$
V_j = {\rm Vect}_\mathbb{Q}\left\{T_i^{e_ia}(\x_i),\ldots,T_i^{e_i(a+jb)}(\x_i) \right\} \, .
$$
Since $\dim V_{n_i-1}<n_i$, there exists $j_0$ such that 
$V_{j_0}=V_{j_0+1}$. The vector space $V_{j_0}$ is then invariant under $T^{e_ib}$ 
and we get that 
$\langle \bnu_0 \, ,\, T^{e_i(a+kb)}(\x_i)\rangle =0$, 
or equivalently that 
$$
\left(T_i^{e_i(a+kb)}\balpha_i\right)^{\bnu_0} = 1 \,,  \; \mbox{for all } k\in\N\,.
$$
Hence $\balpha_i$ is not $T_i$-independent.  
By Theorem \ref{thm:masser}, this provides a contradiction with  the assumption 
that the pair $(T_i,\balpha_i)$ is admissible.  
\end{proof}

\begin{lem}\label{lem:toresMahler}
We continue with the notation of Theorem \ref{thm:lemmedezero}.  
Let $\gamma \in \Q^\star$ and $\bmu$ be a nonzero integer $N$-tuple. 
Then  the set 
$$
\mZ :=\left\{ l \in \N : \left(T_{\k_l}\balpha\right)^\bmu = \gamma \right\}
$$
is negligible. 
\end{lem}

\begin{proof}
Let us assume by contradiction that  $\mZ$ is piecewise syndetic and let $B$ be a bound for $\mZ$. 
Set 
$$
\mathcal E := \{\be(l,m) : l \in \mZ, l+m \in \mZ, m\leq B\} \,.
$$
This is a finite set. 
For every $\be \in \mathcal E$, set 
$$
\mZ_\be:=\left\{ l \in \N  : \left(T_{\k_l}\balpha\right)^{\bmu-\bmu T_\be} = 1 \right\}\, 
$$
and 
$$
\mZ' :=\{ l\in \mZ : \exists m\leq B \mbox{ such that } l+m \in \mZ\} \, .
$$ 
Lemma \ref{lem:syndetique} implies that $\mZ'$ is piecewise syndetic. 
For $l\in \mZ'$, there exists $\be=e(l,m) \in \mathcal E$ such that $m \leq B$ and $l+m \in \mZ$. 
Thus,
\begin{equation*}
\left(T_{\k_l}\balpha\right)^{\bmu-\bmu T_\be}= \gamma/\gamma = 1\, 
\end{equation*}
and we get that $\mZ' \subset \cup_{\be \in \mathcal E} \mZ_\be$. Lemma \ref{lem:syndetique} ensures the existence of  
$\be \in \mathcal E$ such that $\mZ_\be$ is piecewise syndetic. 
By Lemma \ref{lem:admislocglob}, it  follows that $\bmu - T_\be\bmu = 0$ for such a vector $\be$, which 
contradicts the assumption that none of the eigenvalues of $T_i$ is a root of unity.  
\end{proof}

 We are now ready to prove Theorem \ref{thm:lemmedezero}. 

\begin{proof}[Proof of Theorem \ref{thm:lemmedezero}] 
Set 
$$
\mZ:=\left\{ l \in \N : g(T_{\k_l}\balpha)=0 \right\}\, .
$$
The assumption that the pairs $(T_i,\balpha_i)$ are admissible allows us to apply  \cite[Theorem 3 ]{CZ05} 
to the sequence of points $(T_{\k_l}\balpha)_{l \in \N}$. 
In order to apply the result of Corvaja and Zannier, we need to prove that the following three conditions are satisfied.  
\begin{itemize}
\item[{\rm (i)}] There exists a finite set of places $\S$ such that the algebraic points $T_{\k_l}\balpha$ are $\S$-units. 
\item[{\rm (ii)}] The sequence $(T_{\k_l}\balpha)_{l \in \N}$ tends to $0$. 
\item[{\rm (iii)}] One has $\log H(T_{\k_l}\balpha) = \mathcal O(-\log \Vert T_{\k_l}\balpha\Vert )$, where we let $H$ denote 
the absolute Weil height (see Section \ref{sec: notation}).
\end{itemize}
Condition (i) is easy to check. Indeed, any finite number of nonzero algebraic numbers are $\S$-units  
for some $\S$. 
The coordinates of the vector $\balpha$ are thus $\S$-units for some fixed $\S$, and it follows directly that  
all $T_{\k_l}\balpha$ are  $\S$-units too. 
Since by assumption the pairs $(T_i,\balpha_i)$ are admissible, Theorem \ref{thm:masser} implies that 
the sequence $(T_{\k_l}\balpha)_{l \in \N}$ tends to $0$, and 
thus (ii) is satisfied. 
Next we check that (iii) holds. We infer from Lemma \ref{lem:uniformiteconvergencefamilles} 
that 
$$
\Vert T_{\k_l} \Vert= \mathcal O(e^l) \ \mbox{ and }\ \log \Vert T_{\k_l}\balpha \Vert \leq -c e^l\, ,
$$
for some positive real number $c$ and all sufficiently large integers $l$. On the other hand,  we have 
$\log H(T_{\k_l}\balpha) = \mathcal O(\Vert T_{\k_l} \Vert)$. It thus follows that 
$$
\log H(T_{\k_l}\balpha) =  \mathcal O(-\log \Vert T_{\k_l}\balpha \Vert)\,, 
$$
which shows that Condition (iii) is satisfied.  

Applying \cite[Theorem 3]{CZ05} to the sequence of algebraic points 
$(T_{\k_l}\balpha)_{l \in \N}$ 
and to the function $g(\z)$,
we obtain the existence of a finite number of $N$-tuples $\bmu_1,\ldots,\bmu_s$ and of algebraic numbers 
$\gamma_1,\ldots,\gamma_s$, 
such that 
$$
\mZ \subset \bigcup_{i=1}^s \mZ_i 
$$ 
where 
$$
\mZ_i:=\left\{ l \in \N : \left(T_{\k_l}\balpha\right)^{\bmu_i} = \gamma_i \right\}\, .
$$
By Lemma \ref{lem:toresMahler}, the sets $\mZ_i$ are all negligible. It thus follows from Lemma \ref{lem:syndetique2} 
that $\mZ$ is also negligible, which proves the theorem.
\end{proof}


\section{Mahler's method in families}
\label{sec: families}

In this section, we state Theorem \ref{thm: families}, a general lifting theorem dealing with families of 
Mahler systems associated with sufficiently independent transformations.  
In Section \ref{sec: final}, Theorems \ref{thm: permanence}, 
 \ref{thm: purete2}, and \ref{thm: purity}, as well as Corollary \ref{cor:2}, will  
be deduced from this result. 


\subsection{Statement of Theorem \ref{thm: families}}

Let $r$ be a positive integer. For every $i$, $1\leq i \leq r$, let us consider a Mahler system
\begin{equation}\stepcounter{equation}
\label{eq:mahleriT}\tag{\theequation .$i$}
\left(\begin{array}{c} f_{i,1}(\z_i) \\ \vdots \\ f_{i,m_i}(\z_i) \end{array} \right) = 
A_i(\z_i)\left(\begin{array}{c} f_{i,1}(T_i\z_i) \\ \vdots \\ f_{i,m_i}(T_i\z_i) 
\end{array} \right) 
\end{equation}  
where $n_i$ and $m_i$ are positive integers,  
$\z_i:=(z_{i,1},\ldots,z_{i,n_i})$ is a vector of indeterminates, $T_i$ is an $n_i\times n_i$ 
matrix with nonnegative integer coefficients and with spectral radius $\rho(T_i)$,  
$A_i(\z_i)$ belongs to  $\in{\rm GL}_{m_i}(\Q(\z_i))$, 
and $f_{i,1}(\z_i),\ldots,f_{i,m_i}(\z_i)$ belong to $\Q\{\z_i\}$.  
We also let $\balpha_i\in(\Q^\star)^{n_i}$ and $\X_i:=(X_{i,1},\ldots,X_{i,m_i})$ denote a vector of indeterminates. 
Set $\z:=(\z_1,\ldots,\z_r)$ and $\balpha:=(\balpha_1,\ldots,\balpha_r)$.

\begin{rem}
Note that one has to replace $A_i(\z_i)$ by $A_i(\z_i)^{-1}$ to obtain a system as in \eqref{eq:mahler}. 
However, it is 
more natural in our proof to work with systems written in the form \eqref{eq:mahleriT}.  
We recall that $\overline{\mathbb Q(\z)}_{\balpha}$ denote the algebraic closure  of $\Q(\z)$ in 
$\Q\{\z-\balpha\}$.  
\end{rem}

\begin{thm}\label{thm: families}
We continue with the above assumptions. 
Let us assume that the two following conditions hold. 
\begin{enumerate}
\item[\rm{(i)}] For every  $i$, $\balpha_i$ is regular w.r.t.  \eqref{eq:mahleriT} and 
$(T_i,\balpha_i)$ is admissible.
\item[\rm{(ii)}] $\rho(T_1),\ldots,\rho(T_r)$ are pairwise multiplicatively independent. 
\end{enumerate}
 Then for every polynomial 
$P \in \Q[\X_1,\ldots,\X_r]$ that is homogeneous with respect to each family of 
indeterminates $\X_1,\ldots,\X_r$,  
and such that  
$$
P(f_{1,1}(\balpha_1),\ldots,f_{r,m_r}(\balpha_r)) = 0\,,
$$
there exists a polynomial $Q\in\overline{\mathbb Q(\z)}_{\balpha}[\X_1,\ldots,\X_r]$, 
homogeneous with respect to each family of indeterminates  $\X_1,\ldots,\X_r$, and such that 
$$
Q(\z,f_{1,1}(\z_1),\ldots,f_{r,m_r}(\z_r)) = 0 \mbox{ and }\ Q(\balpha,\X_1,\ldots,\X_r)=P(\X_1,\ldots,\X_r)\, .
$$
Furthermore, if $\Q(\z)(f_{1,1}(\z_1),\ldots,f_{r,m_r}(\z_r))$ is a regular extension of $\Q(\z)$, 
then there exists such a polynomial $Q$ in $\Q[\z,\X_1,\ldots,\X_r]$.
\end{thm}

\subsection{Notation}\label{subsec: not}
 In order to lighten the notation, we let $\f_i(\z_i)$ denote the column vector formed by  
the functions $f_{i,1}(\z_i),\ldots,f_{i,m_i}(\z_i)$. We also set 
\begin{equation}\label{eq:dimensions}
M := \sum_{i=1}^r m_i \qquad \text{ and } \qquad N := \sum_{i=1}^r n_i.
\end{equation}
Iterating $k$ times the system \eqref{eq:mahleriT}, one obtains the new system 
\begin{equation}\stepcounter{equation}
\label{eq:mahleritere}\tag{\theequation .$i$}
\f_i(\z_i) = A_{i,k}(\z_i)\f_{i}(T_i^k\z_i) \,,
\end{equation}
where    
$$
A_{i,k}(\z_i):=A_i(\z_i)A_i(T_i\z_i) \cdots A_i(T_i^{k-1}\z_i) \,.
$$
Set $\z:=(\z_1,\ldots,\z_r)$ and by abuse of notation $\f_{i}(\z):=\f_{i}(\z_i)$. 
For every $r$-tuple of positive integers $\k=(k_1,\ldots,k_r)$, one can gather 
the systems \eqref{eq:mahleritere} 
into a single one as follows: 
\begin{equation}
\label{eq: blocks}
\left(\begin{array}{c}  \f_{1}(\z)  \\ \vdots \\ 
\f_{r}(\z)\end{array} \right) = \left(\begin{array}{ccc} 
A_{1,k_1}(\z_1) & &  \\ 
& \ddots & \\
&& A_{r,k_r}(\z_r)
\end{array}
\right)
\left(\begin{array}{c}   \f_{1}(T_\k\z) \\ \vdots \\ 
\f_{r}(T_\k\z) \end{array}\right) \, ,
\end{equation}
where $T_\k:=T_1^{k_1}\oplus \cdots \oplus T_r^{k_r}$. 
Finally, we let $\f(\z)$ denote the column vector formed by all functions $f_{i,j}(\z_i)$, and $A_\k(\z)$ 
denote the block diagonal matrix defined so that \eqref{eq: blocks} can be shortened to
\begin{equation}
\label{eq: blockcompact}
\f(\z) = A_\k(\z)\f(T_\k\z) \,.
\end{equation}
We keep this notation for the rest of the paper. 

\subsection{Choice of the sequence $\k_l$} 
In order to prove Theorem \ref{thm: families},  one needs to choose a sequence 
$(\k_l)_{l\in \N} \subset \N^r$ satisfying the asymptotic  
\begin{equation}\label{eq:bounded2}
\k_l = \Theta l + \mathcal O(1),\qquad l \rightarrow \infty\, ,
\end{equation}
where $\Theta$ is defined as in \eqref{eq:Theta}, as well as some additional properties. 
Let us define a partial order over $\Z^r$ by setting $\k_1 \preceq \k_2$ when the vector $\k_2-\k_1$ has nonnegative coordinates. Let $V$ denote the orthogonal complement  to the vector $\Theta$ in $\Z^r$. That is,  
$$
V:=\{ \bmu \in \Z^r\, :\, \langle \bmu,\Theta \rangle = 0\}\, .
$$
Let $V^\perp$ denote the orthogonal complement to $V$ in $\Z^r$. That is,  
$$
V^\perp :=\{ \k \in \Z^r\, :\, \langle \k,\bmu \rangle = 0 \mbox{ for all } \bmu \in V\}\, .
$$ 
We also set $V_+^\perp:=V^\perp \cap \N^r$.

\begin{lem}\label{lem:choixkl}
There exists a sequence of nonnegative $r$-tuples 
$(\k_l)_{l \in \N} \subset V_+^\perp$, 
satisfying the asymptotic \eqref{eq:bounded2}, and such that 
 $\k_{l} \preceq \k_{l+1}$ for all $l \in \N$.
\end{lem}

\begin{proof}
We first note that the set $V$ could possibly be reduced to $\{0\}$; this is the case when the numbers  
$$
1/\log \rho(T_1),\ldots,1/\log \rho(T_r)
$$
are linearly independent over the rational numbers\footnote{When $r >2$, It is not known 
whether the pairwise multiplicative independence of the numbers $\rho(T_i)$ implies  
that their reciprocals are  linearly independent over $\mathbb Q$.}. 
In that case, one can simply choose  
$$
\k_l:=\left(\left\lfloor \frac{l}{\rho(T_1)} \right\rfloor ,\ldots, \left\lfloor\frac{l}{\rho(T_r)} \right\rfloor \right)\, .
$$
In contrast, the set $V^\perp$ is a nonempty $\Z$-module. Indeed, the $\mathbb R$-vector space 
generated by $V^\perp$ in $\mathbb R^r$ contains the vector $\Theta$.  
Let $\boldsymbol e_1,\ldots,\boldsymbol e_s$ be a $\Z$-basis of $V^\perp$. For all $l\in \N$, 
there exist real numbers $\lambda_1(l),\ldots,\lambda_s(l)$ such that 
$$
l\Theta = \lambda_1(l) \boldsymbol e_1 +\cdots + \lambda_s(l)\boldsymbol e_s \, .
$$ 
We deduce that 
\begin{equation}\label{eq:approximationTheta}
\Vert l\Theta - \lfloor \lambda_1(l) \rfloor \boldsymbol e_1 +\cdots + \lfloor \lambda_s(l) 
\rfloor \boldsymbol e_s \Vert \leq \sum_{i=1}^s \Vert \boldsymbol e_i \Vert\, .
\end{equation}
Since all coordinates of $\Theta$ are positive, there exists a nonnegative integer $l_0$ such that, 
for all $l \geq l_0$, the vector 
$\lfloor \lambda_1(l) \rfloor \boldsymbol e_1 +\cdots + \lfloor \lambda_s(l) \rfloor \boldsymbol e_s$ 
has positive coordinates. For every $l\geq l_0$, set 
$$
\a_l := \lfloor \lambda_1(l) \rfloor \boldsymbol e_1 +\cdots + \lfloor \lambda_s(l) \rfloor 
\boldsymbol e_s \in V_+^\perp \, .
$$
The sequence $(\a_l)_{l \geq l_0}$ agrees with the asymptotic \eqref{eq:bounded2}, 
but not necessarily with the order $\preceq$. 
Let $e:= \sum_{i=1}^s \Vert \boldsymbol e_i \Vert$ and let $\theta>0$ denote the minimum of the 
modules of the coordinates of the vector $\Theta$. If $l_1 \geq l_0$ and $l_2 \geq l_1 + 2e/\theta$, then   \eqref{eq:approximationTheta} implies that $\a_{l_1} \preceq \a_{l_2}$. Set $b:= \lceil 2e/\theta \rceil$. 
Let us define the sequence $(\k_l)_{l \in \N} \subset \N^r$ by setting 
$$
\k_{l_0+lb+j} := \a_{l_0+lb} 
$$
for $l \in \N$ and $0 \leq j < b$, and  $\k_l = \boldsymbol 0$ for $l < l_0$. Then the sequence 
$(\k_l)_{l \in \N}$ has all the required properties. 
\end{proof}

 \section{Hilbert's Nullstellensatz and relation matrices}\label{sec: matrix}
 
 In this section, we gather some preliminary results needed for proving Theorem \ref{thm: families}. 
 In particular, we introduce the so-called \emph{relation matrices} and study some of their properties. 
 
 \medskip
 
 Let us fix some notation. 
 Let $r$, and $m_1,\ldots,m_r$, be some positive integers, and let $d_1,\ldots,d_r$ 
 be some nonnegative integers. 
 Set $M=m_1+\cdots+m_r$ and 
 let $\bmu_1,\ldots,\bmu_t$ denote an enumeration of all distinct 
 $M$-tuples 
$\bmu:=(\mu_{1,1},\ldots,\mu_{1,m_1},\mu_{2,1},\ldots,\mu_{r,m_r})$  
such that 
$$
\mu_{i,1}+\cdots + \mu_{i,m_i} = d_i\, ,
$$
for every $i$, $1\leq i \leq r$. 
 For every $i$, $1 \leq i \leq r$, we let $B_i$ denote an 
$m_i \times m_i$ matrix with coefficients in some commutative ring $\R$, 
and we let $B=B_1\oplus \cdots \oplus B_r$. 
Given a vector of indeterminates $\X=(\X_1,\ldots,\X_r)$, 
where $\X_i=(X_{i,1},\ldots,X_{i,m_i})$,  we note that $(B(\X))^{\bmu_j}\in \R[\X]$ is a 
homogeneous polynomial of degree $d_i$ in each set of variables $\X_i$. 
We let $R_{j,l}(B)$ denote the elements of $\R$ defined by 
\begin{equation}\label{eq:RjlB}
\left(B(\X)\right)^{\bmu_j} = \sum_{l=1}^t R_{j,l}(B)\X^{\bmu_l} \,.
\end{equation}
For every $j$ and $l$, $R_{j,l}(B)$ is a polynomial of degree $d:=\max\{d_1,\ldots,d_r\}$ in 
the coefficients of 
the matrix $B$. Set $R(B):=(R_{j,l}(B))_{1\leq j,l\leq t}$.  Let $B_1$ and $B_2$ be two  
$M\times M$ block diagonal matrix. Then 
\begin{equation}\label{eq: RAB}
R(B_1B_2)=R(B_1)R(B_2)
\end{equation} 
 and 
\begin{equation}\label{eq: RB-1}
R(B_1)^{-1}=R(B_1^{-1})
\end{equation} 
when $B_1$ is invertible.

All along this section, we continue with the notation of Section \ref{subsec: not}.  
We let $A_\k(\z)\in {\rm GL}_{m}(\Q(\z))$ and $f_{1,1}(\z),\ldots,f_{r,m_r}(\z)\in \Q\{\z\}$ be 
as in \eqref{eq: blockcompact}. We fix some $\balpha:=(\balpha_1,\ldots,\balpha_r)\in (\Q^\star)^{N}$  such that $\balpha_i\in (\Q^\star)^{n_i}$ is regular with respect to the matrix $A_i$. 
We also fix  a sequence of integer vectors $(\k_l)_{l\in \N}$ satisfying the assumptions of Lemma \ref{lem:choixkl}.    
Then, for every $\k \in \N^r$, we  set 
\begin{equation}\label{eq: rkz}
\bR_\k(\z):=R(A_\k(\z))\, .
\end{equation}
This is a $t\times t$-matrix with coefficients in $\Q(\z)$. 
One has $\bR_{\boldsymbol 0}(\z) = {\rm I}_t$ and, given $\k_1,\k_2 \in \N^r$, it follows from 
\eqref{eq: RAB} that 
\begin{equation}\label{eq:Rk1k2}
\bR_{\k_1+\k_2}(\z)= \bR_{\k_1}(\z)\bR_{\k_2}(T_{\k_1}\z) \,.
\end{equation}
 Furthermore, 
since the points $\balpha_i$ are assumed to be regular, we infer from \eqref{eq: RB-1} that 
the matrix $\bR_\k(\balpha)$ is well-defined and invertible for all $\k \in \N^r$, with inverse equal to 
$R(A_\k(\balpha)^{-1})$.

Let $\Y:=(y_{i,j})_{1\leq i,j\leq t}$ denote a matrix of indeterminates.  
Given a field $\mathbb K$ and a nonnegative integer $\delta_1$, we let $\mathbb K[\Y]_{\delta_1}$ 
denote the set of polynomials of degree at most $\delta_1$ in every indeterminate  
$y_{i,j}$.  Given two nonnegative integers $\delta_1$ and $\delta_2$, we let 
 $\mathbb K[\Y,\z]_{\delta_1,\delta_2}$ denote the set of polynomials $P\in \mathbb K[\Y,\z]$ of degree at most  
$\delta_1$ in every indeterminate $y_{i,j}$ and of total degree at most $\delta_2$ in the indeterminates $z_{i,j}$.  
 By Theorem \ref{thm:lemmedezero}, a polynomial $P\in\Q(\z)[\Y]$ is well-defined at the point  
$(\bR_{\k_l}(\balpha), T_{\k_l} \balpha)$ for all $l$ in a full subset of $\N$.  
Set  
\begin{equation*}
\I := \{P \in \Q(\z)[\Y] \ : P(\bR_{\k_l}(\balpha), T_{\k_l} \balpha)= 0\, , 
\; \forall l \mbox{ in a full subset of $\N$}\} \, .
\end{equation*}
Piecewise syndetic, negligible, and full sets are introduced in Definition \ref{def:full}. 

\subsection{Estimates for the dimension of certain vector spaces}

Let $\delta_1$ and $\delta_2$ be two nonnegative integers. Set 
$\mathcal I(\delta_1) := \mathcal I \cap  \Q(\z)[\Y]_{\delta_1}$ and 
$\I(\delta_1,\delta_2):=\I \cap \Q[\Y,\z]_{\delta_1,\delta_2}$. 
Note that  $\I(\delta_1,\delta_2)$ is a vector subspace of $ \Q[\Y,\z]_{\delta_1,\delta_2}$, and   
let $\I^{\perp}(\delta_1,\delta_2)$ denote a complement to  $\I(\delta_1,\delta_2)$ in 
$ \Q[\Y,\z]_{\delta_1,\delta_2}$.

\begin{lem}\label{dimensionespace} Let $d(\delta_1,\delta_2)$ denote the dimension of  
$\I^\perp(\delta_1,\delta_2)$ over $\Q$. 
There exists a positive real number $c_1(\delta_1)$, that does not depend on $\delta_2$, such that 
$$
d(\delta_1,\delta_2) \sim c_1(\delta_1)\delta_2^N \,, \mbox{ as $\delta_2$ tends to infinity. }
$$
\end{lem}

\begin{proof}
Let $\Y^{\bnu_1},\ldots,\Y^{\bnu_h}$, $h:=(\delta_1+1)^{t^2}$, denote an enumeration of the 
monomials of degree at most $\delta_1$ 
in every indeterminates  $y_{i,j}$. Let  
$P_1,\ldots,P_h$ be polynomials in $\I(\delta_1)$. 
Every $P_i$ has a unique decomposition of the form    
$$
P_i(\Y,\z):= \sum_{j=1}^h p_{i,j}(\z) \Y^{\bnu_j}\, ,
$$
where $p_{i,j}(\z) \in \Q(\z)$, $1 \leq i,j \leq h$. 
Let $C(\z)$ denote the square matrix defined by $C(\z):=(p_{i,j}(\z))_{1\leq i,j \leq h}$. 
By Theorem \ref{thm:lemmedezero}, $C(\z)$ is well-defined at $T_{\k_l}\balpha$ 
for all $l$ in a full set $\mZ_0 \subset \N$. 
For every $i$, $1\leq i \leq h$, let $\mZ_i$ denote the set of nonnegative integers $l$ such that 
 $P_i(\bR_{\k_l}(\balpha), T_{\k_l} \balpha)=0$. Since $P_1,\ldots,P_h \in \I(\delta_1)$, 
 the sets $\mZ_1,\ldots,\mZ_h$ are full.   By Lemma \ref{lem:syndetique2}, 
 the set $\mZ: = \bigcap_{i=0}^h \mZ_i$ is full. For every $l \in \mZ$, one has  
\begin{equation}\label{eq: matnoy}
C(T_{\k_l} \balpha)\left(\begin{array}{c} \bR_{\k_l}(\balpha)^{\bnu_1} \\ \vdots \\
\bR_{\k_l}(\balpha)^{\bnu_h}\end{array}\right)=
 \left(\begin{array}{c} P_1(\bR_{\k_l}(\balpha), T_{\k_l} \balpha) \\ \vdots \\ 
P_h(\bR_{\k_l}(\balpha), T_{\k_l} \balpha)\end{array}\right)=0\, .
\end{equation}
Since the matrix 
$\bR_{\k_l}(\balpha)$ 
is invertible, it is nonzero. 
Hence the vector  $(\bR_{\k_l}(\balpha)^{\bnu_1},\ldots,\bR_{\k_l}(\balpha)^{\bnu_h})$ is also nonzero, 
and we  deduce from \eqref{eq: matnoy} that $\det C(T_{\k_l} \balpha)=0$ for all $l \in \mZ$. 
Since $\mZ$ is not negligible, Theorem \ref{thm:lemmedezero} implies that  
$\det C(\z)=0$. Hence $\I(\delta_1)$  is a strict subspace of $ \Q(\z)[\Y]_{\delta_1}$, 
say of dimension $d<h$. 
There thus exist polynomials $b_{i,j}(z) \in \Q[\z]$, $1\leq i \leq h-d$, $1\leq j \leq h$, such that 
for all $p_1(\z),\ldots,p_h(\z) \in \Q(\z)$:
\begin{equation}\label{eq: conditionI}
\sum_{j=1}^h p_j(\z)\Y^{\bnu_j} \in \I(\delta_1) \Leftrightarrow \sum_{j=1}^h b_{i,j}(\z)p_j(\z) = 0 
\;\;\;\; \forall i,\; 1 \leq i \leq h-d\, .
\end{equation}
Since $d+ (h-d) = h$, these equations are linearly independent.  
Now, let us consider a polynomial 
$P :=\sum_{j=1}^h p_j(\z)\Y^{\bnu_j} \in \Q[\Y,\z]_{\delta_1,\delta_2}$ and set  
$$
p_j(\z) = \sum_{\lambd\,:\,|\lambd|\leq \delta_2} p_{j,\lambd} \z^\lambd\ \text{ and }\ b_{i,j}(\z)=\sum_{\kapp\, :\, |\kapp|\leq \delta_1'} b_{i,j,\kapp} \z^\kapp\, ,
$$
where $\delta_1'$ is a nonnegative integer depending only on $\delta_1$, and where  the numbers 
$p_{j,\lambd}$ and $b_{i,j,\kapp}$ are algebraic for all quadruples $(i,j,\lambd, \kapp)$. 
By  \eqref{eq: conditionI},  $P$ belongs to $ \I(\delta_1,\delta_2)$ if and only if 
$$
\sum_{j=1}^h \sum_{\substack{|\lambd|\leq \delta_2, |\kapp|\leq \delta_1' \\ \lambd + \kapp=\gamm }} b_{i,j,\kapp}  p_{j,\lambd} = 0 \, , \;\;\;\; \forall (\gamm,i)\, .
$$
If we fix the vector $\gamm$, then, as soon as every coordinate of  
 $\gamm$ is larger than $\delta_1'$ and that $\vert \gamm \vert\leq \delta_2$, 
 the number of linearly independent equations does not 
depend on our choice for $\gamm$. 
Thus, 
the total number of linearly independent equations defining the vector space 
$\I(\delta_1,\delta_2)$ is equivalent to  
$$
c_1(\delta_1)\delta_2^N\, ,
$$
when $\delta_2$ tends to infinity, where $c_1(\delta_1)$ is a positive real number depending 
only on  $\delta_1$. This number is precisely the dimension  
of the vector space $\I^\perp(\delta_1,\delta_2)$, which ends the proof. 
\end{proof}

\begin{lem}\label{lem:majorationespacepolynome}
For every pair of nonnegative integers $(\delta_1,\delta_2)$, one has 
$$
\dim \I^{\perp}(2\delta_1,\delta_2) \leq 2^{t^2}\dim \I^{\perp}(\delta_1,\delta_2) \,.
$$
\end{lem}

\begin{proof}
Let  $P(\Y,\z)\in \Q[\Y,\z]_{2\delta_1,\delta_2}$.  Such a polynomial has a decomposition of the form 
\begin{equation}\label{eq: pl}
P(\Y,\z)=\sum_{\ell=1}^{t^2} e_\ell(\Y)^{\delta_1}P_\ell(\Y,\z) \,,
\end{equation}
where we let $e_\ell(\Y)$ denote the $2^{t^2}$ monomials of degree at most one in the indeterminates 
$y_{i,j}$, and where each $P_\ell(\Y,\z)$ belongs to $\Q[\Y,\z]_{\delta_1,\delta_2}$. If, in such a decomposition,  
every polynomial $P_\ell$ belongs to $\I(\delta_1,\delta_2)$, then $P$ belongs to $\I(2\delta_1,\delta_2)$. 
The decomposition \eqref{eq: pl} naturally defines a surjective linear map from  
$\left(\Q[\Y,\z]_{\delta_1,\delta_2}/ \I(\delta_1,\delta_2)\right)^{2^{t^2}}$ to $\Q[\Y,\z]_{2\delta_1,\delta_2}/ \I(2\delta_1,\delta_2)$. 
It follows that  
$$
\dim_{\Q} \I^{\perp}(2\delta_1,\delta_2) \leq 2^{t^2} \dim_{\Q} \I^{\perp}(\delta_1,\delta_2)\, ,
$$
as wanted. 
\end{proof}

\subsection{Nullstellensatz and relation matrices} 

In this section, we show how Hilbert's Nullstellensatz allows us to exhibit a matrix $\bphi$, 
called a \emph{relation matrix}, whose coordinates are all algebraic over $\Q(\z)$, 
and which encodes the algebraic relations over $\Q(\z)$ 
of degree at most $d_i$ in each variables between the functions   
$f_{1,1}(\z),\ldots,f_{r,m_r}(\z)$. These relation matrices are the cornerstone of the proof 
of Theorem 
\ref{thm: families}.

We first prove the following lemma. 

\begin{lem}\label{lem:idealI}
The set $\mathcal I$ is a radical ideal of $\Q(\z)[\Y]$. 
\end{lem}

\begin{proof} 
Checking that $\mathcal I$ is an ideal of $\Q(\z)[\Y]$ is not difficult. If $P_1,P_2 \in \I$, and if 
$\mZ_1$ (resp.\ $\mZ_2$) are full sets of nonnegative integers  $l$ for which 
$P_1(\bR_{\k_l}(\balpha),T_{\k_l}\balpha)=0$ (resp.\ $P_2(\bR_{\k_l}(\balpha),T_{\k_l}\balpha)=0$), 
then $P_1+P_2$ vanishes at the points $(\bR_{\k_l}(\balpha),T_{\k_l}\balpha)$ for all 
$l \in \mZ_1 \cap \mZ_2$. By Lemma \ref{lem:syndetique2}, this set is full. Hence $P_1+P_2 \in \I$. 

 Now let $P_1 \in \I$ and $P_2\in \Q(\z)[\Y]$. On the one hand, $P_1(\bR_{\k_l}(\balpha),T_{\k_l}\balpha)=0$ 
 for all $l$ in a full set $\mZ_1$, while, on the other hand, Theorem \ref{thm:lemmedezero} ensures that 
$P_2(\Y,\z)$ is well-defined at $(\bR_{\k_l}(\balpha),T_{\k_l}\balpha)$ for all nonnegative integers 
$l$ 
outside a negligible set $\mZ_2$. We deduce that 
$$
P_1(\bR_{\k_l}(\balpha),T_{\k_l}\balpha)P_2(\bR_{\k_l}(\balpha),T_{\k_l}\balpha)=0
$$
for all $l \in \mZ_1 \setminus \mZ_2$. By Lemma \ref{lem:syndetique2}, this is a full set. Hence $P_1P_2 \in \I$.

Let $P\in \Q(\z)[\Y]$ be such that $P^r \in \I$ for some $r$. If $l$ is a nonnegative integer such that  $P(\bR_{\k_l}(\balpha),T_{\k_l}\balpha)^r =0$, then $P(\bR_{\k_l}(\balpha),T_{\k_l}\balpha)=0$.  
Hence $P \in \I$ and 
$\I$ is a radical ideal. 
\end{proof}

Let $\mathbb K$ denote an algebraic closure of $\Q(\z)$. 

\begin{lem}\label{lem: phi1}
There exists a matrix $\bphi \in  {\rm GL}_t(\mathbb K)$ such that 
$$
P(\bphi,\z) = 0\, ,
$$ 
for all polynomials $P \in \I$.
\end{lem}

\begin{proof}
Let us consider the affine algebraic set $\mathcal V$ associated with the radical ideal $\I$. That is,   
$$
{\mathcal V}:=\{\bphi \in \mathcal M_t(\mathbb K)\ : \ P(\bphi,\z) = 0 \,, \; \forall P \in \I \}\, .
$$
According to the weak form of Hilbert's Nullstellensatz (see, for instance, \cite[Theorem 1.4, p.\ 379]{Lang}),  
$\mathcal V$ is nonempty as soon as  
$\I$ is a proper ideal of $\Q(\z)[\Y]$. But the definition of $\mathcal I$ clearly implies that nonzero constant polynomials do not belong to $\mathcal I$. Hence $\mathcal V$ is nonempty.  

Now, let us assume by contradiction that $\det \bphi=0$ for all $\bphi$ in $\mathcal V$. 
By Hilbert's Nullstellensatz (see, for instance, \cite[Theorem 1.5, p.\ 380]{Lang}), 
the polynomial $\det \Y$  belongs to the radical of the ideal  
$\I$. Hence  $\det \Y\in \I$  for $\I$ is radical. Thus,  $\det \bR_{\k_l}(\balpha)=0$ 
for all $l$ in a full set. This provides a contradiction since $\bR_{\k_l}(\balpha)$ is invertible 
for all $l$ in $\N$. We thus deduce that there exists an invertible matrix $\bphi$ 
in $\mathcal V$, as wanted. 
\end{proof}

\begin{defi} A matrix $\bphi \in  {\rm GL}_t(\mathbb K)$ satisfying the property of Lemma 
\ref{lem: phi1} is called a \emph{relation matrix}. 
\end{defi}

The next lemma plays a central role in the proof of Theorem \ref{thm: families}.

\begin{lem}\label{lem: phi2}
Let $\bphi \in  {\rm GL}_t(\mathbb K)$ 
be a relation matrix.  
Then  
$$
P\left(\bphi \bR_\k(\z),T_\k\z\right) = 0\, ,
$$ 
for all $P \in \I$ and  all $\k \in V_+^\perp$. 
\end{lem}

\begin{proof} Let $\A$ denote the subring of $\Q(\z)$ formed by all rational functions with no pole  
at the points $T_{\k}\balpha$, $\k \in \N^r$. The set $\mathcal S$ of polynomials $P\in \Q[\z]$ 
that does not vanish at any of the points $T_\k \balpha$, $\k \in \N^r$, being multiplicatively closed, 
the ring $\A$ is the localization of $\Q[\z]$ at $\mathcal S$, \textit{i.e.}, $\A=\S^{-1}\Q[\z]$. 
It follows that $\A$ is a 
Noetherian  ring (see, for instance, \cite[Proposition 1.6, p.\ 415]{Lang}). 
Since by assumption each of the points $\balpha_1,\ldots,\balpha_r$ is regular with respect to the 
corresponding Mahler system \eqref{eq:mahleriT}, the coefficients of the matrices 
$\bR_\k(\z)$, $\k \in \N^r$, belong to the ring $\A$.  
With any piecewise syndetic set $\mZ \subset \N$,  we associate the set 

$$
\I_\mZ:=\left\{P \in \A[\Y]\, : \, P(\bR_{\k_l}(\balpha),T_{\k_l}\balpha) = 0\,, \;  \forall l \in \mZ \right\}\, .
$$

The proof is divided into the following seven simple results, namely Facts \ref{aff:ideal} to \ref{aff:inclusionideaux}.   

\begin{Fact}\label{aff:ideal}
The set  $\I_\mZ$ is an ideal of $\A[\Y]$.
\end{Fact}

Let $P_1,P_2 \in \I_\mZ$. Then $P_1+P_2$ vanishes at $(\bR_{\k_l}(\balpha),T_{\k_l}\balpha)$  
for all $l \in \mZ$. Hence $P_1+P_2 \in \I_\mZ$.  Let $P_1 \in \I_\mZ$ and $P_2\in \A[\Y]$. 
Then $P_1(\bR_{\k_l}(\balpha),T_{\k_l}\balpha)=0$ for all $l \in \mZ$. On the other hand, 
by definition of  $\A$, 
 $P_2(\Y,\z)$ has no pole at $(\bR_{\k_l}(\balpha),T_{\k_l}\balpha)$, $l\in \N$. It follows that 
 $P_1P_2$ vanishes at $(\bR_{\k_l}(\balpha),T_{\k_l}\balpha)$ for all $l \in \mZ$. Hence $P_1P_2 \in \I_\mZ$, which proves Fact \ref{aff:ideal}. 

\medskip

If $\mZ'$ is a piecewise syndetic set such that $\mZ' \subset \mZ \subset \N$, one has 
$\I_\mZ \subset \I_{\mZ'}$. Since $\A$ is Noetherian,  $\A[\Y]$ is Noetherian too 
and any increasing sequence of ideals is stationary. 
Thus, for every piecewise syndetic set $\mZ \subset \N$, there exists a 
piecewise syndetic set $\mZ_0 \subset \mZ$ such that $\I_{\mZ_1}=\I_{\mZ_0}$ for all piecewise syndetic sets  
$\mZ_1 \subset \mZ_0$. For Facts \ref{aff:infinitek} to \ref{aff:inclusionideaux}, we fix such a pair of sets   
$(\mZ,\mZ_0)$.

\begin{Fact}\label{aff:infinitek}
There exist infinitely many $r$-tuples $\k \in V_+^\perp$ such that 
\begin{equation}
\label{eq:defL0}
\mZ_0(\k):=\{l \in \mZ_0\ : \ \exists l'\in \mZ_0, \; \k_{l'}=\k_l + \k \}
\end{equation}
is a piecewise syndetic set. 
\end{Fact}

Let $B$ be a bound for $\mZ_0$ and $e\geq 0$ be an integer. By Lemma \ref{lem:syndetique}, 
 the set 
$$
\mZ_e:=\{l \in \mZ_0\ : \ \exists l' \in \mZ_0, e \leq l'-l \leq e+B\}
$$
is piecewise syndetic. By \eqref{eq:bounded2}, there are only finitely many differences $\k_{l'}-\k_l$ 
for which $e \leq l'-l \leq e+B$. Furthermore, by Lemma \ref{lem:choixkl}, such differences all belong to $V_+^\perp$. 
Let $\K_e \subset V_+^\perp$ denote the finite set formed by these differences. Then 
$$
\mZ_e\subset \cup_{\k \in \K_e} \mZ_0(\k)\, .
$$
By Lemma \ref{lem:syndetique}, at least one of the sets $\mZ_0(\k)$, $\k \in \K_e$, is piecewise syndetic. 
Letting $e$ tend to infinity, this proves Fact \ref{aff:infinitek}. 

\medskip 

Now, let 
$$
\K:=\{\k \in V_+^\perp : \ \mZ_0(\k)\, \mbox{ is piecewise syndetic}\}\, 
$$
denote this infinite set.

\begin{Fact}
\label{aff:Zmodule}
The $\Z$-module generated by $\K$ in $\Z^r$ is equal to $V^\perp$.
\end{Fact}

Let $W$ denote the $\Z$-module generated by $\K$ in $\Z^r$. It is enough to show that $W^\perp$, its orthogonal complement  in $\Z^r$, is equal to  $V$. Since $W \subset V^\perp$, we have  
$V\subset W^\perp$. Let us prove the reverse inclusion. Let $\lambd \in W^\perp$. Then $\lambd$ is orthogonal to all $\k\in \K$. 
By construction, $\K$ remains at bounded distance from $\mathbb R\Theta$. 
Renormalizing and taking the limit for large vectors $\k\in \K$, 
we get that $\lambd$ is orthogonal to $\Theta$. Hence, $\lambd \in V$.  
This proves Fact \ref{aff:Zmodule}. 

\medskip

Given $\k \in V_+^\perp$, we define an action from the monoid $V_+^\perp$ to  $\A[\Y]$ by: 
$$
\sigma_\k:\left\{\begin{array}{rcl} \A[\Y]& \rightarrow & \A[\Y]
\\
P(\Y,\z) & \mapsto & P(\Y \bR_\k(\z),T_\k\z)\, .
\end{array}
\right.
$$
Note that the map $\sigma_\k$ is well-defined. Indeed, we already observed that  the coordinates of 
$\bR_\k(\z)$ belong 
to the ring $\A$ for all $\k \in \N^r$. 

\begin{Fact}
\label{aff:actionsigmak}
For all  $\k \in \K$, $\sigma_\k(\I_{\mZ_0}) \subset \I_{\mZ_0}$. 
\end{Fact}

Let $P \in \I_{\mZ_0}$, $\k \in \K$, and $l \in \mZ_0(\k)$. Let $l' \in \mZ_0$  be such that 
$\k_{l'}=\k+\k_l$. Then, we have  
\begin{eqnarray*}
\sigma_\k(P)(\bR_{\k_l}(\balpha),T_{\k_l}(\balpha))&=&P(\bR_{\k_l}(\balpha)\bR_{\k}(T_{\k_l}\balpha),T_\k T_{\k_l}\balpha)
\\ &= &P(\bR_{\k_{l}+\k}(\balpha),T_{\k+\k_{l}}\balpha)
\\ &= &P(\bR_{\k_{l'}}(\balpha),T_{\k_{l'}}\balpha)
\\& = &0\, .
\end{eqnarray*}
Thus,  $\sigma_\k(P)$ belongs to the ideal $\I_{\mZ_0(\k)}$, which  is equal to $\I_{\mZ_0}$ by minimality. 
This proves Fact \ref{aff:actionsigmak}.

\begin{Fact}
\label{aff:actionreciproquesigmak}
Let $P \in \I_{\mZ_0}$ be such that $P=\sigma_\k(Q)$ for some $Q \in \A[\Y]$ and $\k\in \K$. 
Then $Q \in \I_{\mZ_0}$.
\end{Fact}

Since $\mZ_0(\k)$ is piecewise syndetic, we infer from \eqref{eq:bounded2} that the set 
$$
\mZ_0(\k)^{-1}:=\{l' \in \mZ_0\ : \  \exists l \in \mZ_0, \,\k_{l'}=\k_l + \k\} 
$$
is also piecewise syndetic. By minimality, we obtain $\I_{\mZ_0(\k)^{-1}}=\I_{\mZ_0}$. 
Let $l' \in \mZ_0(\k)^{-1}$ and let $l \in \mZ_0$ be such that $\k_{l'}:= \k_{l} + \k$.  
Then 
\begin{eqnarray*}
Q(\bR_{\k_{l'}}(\balpha),T_{\k_{l'}}\balpha)&=&Q(\bR_{\k_{l}+\k}(\balpha),T_{\k+\k_{l}}\balpha)
\\ &=& Q(\bR_{\k_l}(\balpha)\bR_{\k}(T_{\k_l}\balpha),T_\k T_{\k_l}\balpha)
\\ &=&
\sigma_\k(Q)(\bR_{\k_{l}}(\balpha),T_{\k_{l}}(\balpha))
\\ &=&P(\bR_{\k_{l}}(\balpha),T_{\k_{l}}(\balpha))
\\ &=&0\, .
\end{eqnarray*}
Thus, $Q \in \I_{\mZ_0(\k)^{-1}}=\I_{\mZ_0}$.  
This 
proves Fact \ref{aff:actionreciproquesigmak}.

\begin{Fact}\label{aff:stabVperp}
For all $\k \in V_+^\perp$,  $\sigma_\k (\I_{\mZ_0}) \subset \I_{\mZ_0}$. 
\end{Fact}

Let $\k \in V_+^\perp$ and $P \in \I_{\mZ_0}$. 
By Fact \ref{aff:Zmodule}, there is a decomposition of the form  
$$
\k:=\a_1+\cdots+\a_u-\a_{u+1}-\cdots - \a_v
$$
with $\a_1,\ldots,\a_v \in \K$.  
Using recursively Fact \ref{aff:actionsigmak} with $\sigma_{\a_1},\ldots,\sigma_{\a_u}$, 
we deduce that 
$\sigma_{\a_1+\cdots+\a_u}(P) \in \I_{\mZ_0}$. 
On the other hand, we have 
$
\sigma_{\a_{u+1}+\cdots +\a_v}(\sigma_{\k}(P))=\sigma_{\a_1+\cdots+\a_u}(P)\in \I_{\mZ_0}\, .
$
Using recursively Fact \ref{aff:actionreciproquesigmak} with $\sigma_{\a_{u+1}},\ldots,\sigma_{\a_{v}}$, 
we obtain that $\sigma_{\k}(P) \in \I_{\mZ_0}$. 
This proves Fact \ref{aff:stabVperp}.

\begin{Fact}\label{aff:inclusionideaux}
One has  $\I_{\mZ_0} \subset \I$.
\end{Fact}

Let $l_0\in \N$ denote the smallest element in $\mZ_0$. 
Let $P \in \I_{\mZ_0}$ and let $l \geq l_0$ be an integer. 
By Lemma \ref{lem:choixkl},  
$\k_{l}-\k_{l_0} \in V_+^\perp$.  
Set $Q:=\sigma_{\k_{l}-\k_{l_0}}(P)$. 
By Fact \ref{aff:stabVperp}, we obtain that $Q \in \I_{\mZ_0}$.
Since $l_0 \in \mZ_0$, we have 
$$
P(\bR_{\k_l}(\balpha),T_{\k_l}\balpha) = Q(\bR_{\k_{l_0}}(\balpha),T_{\k_{l_0}}\balpha)=0\, .
$$
Hence $P$ vanishes at $(\bR_{\k_l}(\balpha),T_{\k_l}\balpha)$, 
for all $l \geq l_0$.  
This proves Fact \ref{aff:inclusionideaux}.

\medskip 

We are now ready to conclude the proof of Lemma \ref{lem: phi2}. 
Let $P \in \I$, $\bphi \in  {\rm GL}_t(\mathbb K)$ be a relation 
matrix, and $\k \in V_+^\perp$. Let $b(\z) \in \Q[\z]$ be a nonzero polynomial such that 
$b(\z)P(\Y,\z) \in \Q[\z,\Y]$. In particular, $b(\z)P(\Y,\z)\in \I \cap \A[\Y]$.  Let $\mZ$ be the set of 
 integers $l\geq 0$ for which $b(T_{\k_l}\balpha)P(\bphi\bR_{\k_l}(\balpha),T_{\k_l}\balpha)=0$. 
Since $bP \in \I$, $\mZ$ is full and hence piecewise syndetic. There thus exists a piecewise syndetic set 
 $\mZ_0 \subset \mZ$ satisfying Facts \ref{aff:infinitek} to \ref{aff:inclusionideaux}. By definition,  
 $b(\z)P(\Y,\z)\in \I_{\mZ_0}$. By Fact \ref{aff:stabVperp}, we also have 
 $\sigma_\k(bP) \in \I_{\mZ_0}$, for all $\k \in V_+^\perp$. By Fact \ref{aff:inclusionideaux}, we deduce that 
 $\sigma_\k (bP) \in \I$. Then, we infer from Lemma \ref{lem: phi1} that  
$$
b(T_\k\z)P(\bphi\bR_\k(\z),T_\k\z)= \sigma_\k (bP)(\bphi,\z)=0\, . 
$$
Since $T_\k$ is nonsingular and $b(\z)\not=0$, we obtain  
$P(\bphi\bR_\k(\z),T_\k\z)=0$. 
\end{proof}

\subsection{Analyticity of relation matrices}
 
 Let $\bphi$ be a relation matrix. 
All coordinates of $\bphi$ being algebraic over $\Q(\z)$,  they 
generate a finite extension of $\Q(\z)$. Let $\mathbf k\subset \mathbb K$ denote this extension 
and let $\gamma\geq 1$ be the degree of $\mathbf k$.   
Choosing a primitive element $\phi$ in $\mathbf k$,  
we obtain a decomposition of the form  
\begin{equation}\label{eq:decompositionPhi}
\bphi= \sum_{j=0}^{\gamma-1}  \bphi_j(\z) \phi^{j}\, ,
\end{equation}
where the matrices $\bphi_j(\z)$, $0 \leq j \leq \gamma-1$, have coefficients in $\Q(\z)$. 
The field $\mathbb K$ is {\it a priori} an abstract algebraic closure of $\Q(\z)$, but we can easily  
reduce the situation to the case where the coordinates of $\bphi$ are analytic at some 
suitable point  $T_{\k_{l_0}}\balpha$.  

\begin{lem}
\label{lem:zeroMasserextended} We continue with the previous notation.  
There exist an integer $l_0\geq 0$, a neighborhood $\mathcal V$ of $T_{\k_{l_0}}\balpha$, and a  
function $\varphi(\z)$ that is analytic on $\mathcal V$ and algebraic over $\Q(\z)$ such that  
 the following properties holds.  

\begin{itemize}

\item[{\rm (a)}] $\Vert T_{\k_{l_0}}\balpha\Vert<1$. 

\item[{\rm (b)}] $T_{\k_{l_0}}\balpha$ belongs to the disc of convergence of 
$f_{1,1}(\z),\ldots,f_{r,m_r}(\z)$. 

\item[{\rm (c)}] The matrix  
\begin{equation*}
\bphi(\z):= \sum_{j=0}^{\gamma-1}  \bphi_j(\z) \varphi(\z)^{j}  \in  {\rm GL}_t(\mathcal Mer(\mathcal V))\, 
\end{equation*}
is a relation matrix. That is, it satisfies Lemmas \ref{lem: phi1} and \ref{lem: phi2}.

\item[{\rm (d)}] For every $j$, $1\leq j\leq \gamma-1$, the coordinates of the matrix $\bphi_j(\z)$ 
are analytic on $\mathcal V$ and the 
matrix $\bphi(T_{\k_{l_0}}\balpha)$ is invertible. 
\end{itemize}
\end{lem}

\begin{proof} 
By Theorem \ref{thm:masser}, 
$\lim_{l\to\infty} T_{\k_l} \balpha =0$. This ensures that (a) and (b) are satisfied for all  
sufficiently large integers $l$.

Let $P(\z,y)\in \Q[\z,y]$ denote the minimal polynomial of $\phi$ 
and let $D(\z)\in \Q[\z]$ denote the discriminant of $P$, seen as a polynomial in the variable $y$. 
Since $\bphi$ is a relation matrix,  $\det \bphi$ is nonzero and algebraic over 
$\Q(\z)$.     
There thus exist polynomials $q_0(\z),\ldots,q_\nu(\z) \in \Q[\z]$,  
$q_0(\z)\not=0$, such that    
\begin{equation}\label{eq:minequation}
q_0(\z) = q_1(\z)\det \bphi+q_2(\z)\det \bphi^2+\cdots+q_\nu(\z)\det \bphi^\nu\, .
\end{equation}
Let $d(\z)$ be the least common multiple of the denominators of the coefficients of 
the matrices $\bphi_0(\z),\ldots, \bphi_{\gamma-1}(\z)$. 
Since the polynomial 
$D(\z)q_0(\z)d(\z)$ 
is nonzero, Theorem \ref{thm:lemmedezero} ensures the existence of a full set   
$\mathcal E\subset \mathbb N$ such that 
$$
D(T_{\k_l}\balpha)q_0(T_{\k_l}\balpha)d(T_{\k_l}\balpha)\not=0\, ,\; \forall l\in \mathcal E\,.
$$

Let $l_0\in \mathcal E$ be chosen large enough to guarantee 
that (a) and (b) hold. 
Since $D(T_{\k_{l_0}}\balpha)\not=0$, the 
implicit function theorem (see, for instance, \cite[Proposition 6.1, p.\ 138]{Cartan})  
implies that there exists a function $\varphi(\z)$ that is analytic 
on a neighborhood of $T_{\k_{l_0}}\balpha$, say $\mathcal V_0$, and such that 
$P(\z,\varphi(\z))=0$.  
Note that there is a 
$\Q(\z)$-isomorphism between the field  
$\Q(\z,\phi)$ and $\Q(\z,\varphi(\z))$.  We thus deduce that the matrix 
$\bphi(\z):= \sum_{j=0}^{\gamma-1}  \bphi_j(\z) \varphi(\z)^{j}$ 
satisfies the properties of Lemmas \ref{lem: phi1} and \ref{lem: phi2}. 
Furthermore, as $\det  \bphi\not= 0$,  
we also deduce that $\det  \bphi(\z)\not= 0$. Hence $\bphi(\z)\in  
{\rm GL}_t(\mathcal Mer(\mathcal V_0))$. Finally, we deduce that $\det  \bphi(\z)$  
 also satisfies Equation \eqref{eq:minequation}. 
Since $q_0(T_{\k_{l_0}}\balpha)\not=0$, we get that $\det \bphi(T_{\k_{l_0}}\balpha)\not=0$. Hence 
 the matrix $\bphi(T_{\k_{l_0}}\balpha)$ is invertible. Furthermore, since $d(T_{\k_{l_0}}\balpha)\not=0$, the coordinates of $\bphi_j(\z)$ 
are analytic on some neighborhood  of $T_{\k_{l_0}}\balpha$, say $\mathcal V_1$.  
Finally, setting $\mathcal V=\mathcal V_0\cap \mathcal V_1$,  we obtain that 
Properties (a)--(d) are satisfied. 
\end{proof}

\section{Proof of Theorem \ref{thm: families}}\label{sec: mainproof}

Let $P \in \Q[\X_1,\ldots,\X_r]$ be defined as in Theorem \ref{thm: families}. 
Let $d_i$ denote the total degree of $P$ in the indeterminates $\X_i$. 
Set $\X:=(\X_1,\ldots,\X_r)$. We keep on with the notation of the previous section. 
The monomials $\X^{\bmu_1},\ldots,\X^{\bmu_t}$ are precisely those which are of total degree $d_i$ in the indeterminates $\X_i$, for every $i$. Hence we have a decomposition of the form 
$$
P(\X):= \sum_{j=1}^t \tau_j \X^{\bmu_j}\, ,
$$
where $\tau_1,\ldots,\tau_t \in \Q$. 
Set $\btau:=(\tau_1,\ldots,\tau_t) \in \Q^t$ and let  $\X^\bmu \in \Q[\X]^t$ denote the column vector of indeterminates  whose coordinates are the monomials $\X^{\bmu_1},\ldots,\X^{\bmu_t}$. We write $P(\X)=\btau\X^\bmu$. Given a matrix of indeterminates 
$\Y:=(y_{i,j})_{1\leq i,j\leq t}$, we set 
\begin{equation}
\label{eq:defF}
F(\Y,\z):=\sum_{i,j} \tau_i y_{i,j} \f(\z)^{\bmu_j} = \btau \Y \f(\z)^\bmu \in \Q\{\z\}[\Y]\, ,
\end{equation}
where we let $\f(z) \in \Q\{\z\}^M$ denote the column vector formed by the functions 
$f_{1,1}(\z_1),\ldots,f_{r,m_r}(\z_r)$ and $\f(\z)^{\bmu} \in \Q\{\z\}^t$ denote the column vector formed by the  functions $\f(\z)^{\bmu_1},\ldots,\f(\z)^{\bmu_t}$. Note that $F$ is a linear form in $\Y$. 
Evaluating at $({\rm I}_t,\balpha)$, we obtain 
\begin{equation}\label{eq: annulationF}
F({\rm I}_t,\balpha)=\sum_{j=1}^t  \tau_j\f(\balpha)^{\bmu_j} =P(\f(\balpha)) = 0 \,. 
\end{equation}

\begin{rem}\label{rem: F} 
We have  $F(\Y,\z)  \in \Q[\Y,\f(\z)]\subset\Q\{\z\}[\Y]$. Also,  
$F(\Y,\z)$ can be seen as an element of $\Q[\Y][[\z]]$, as we will sometimes do in what follows. 
\end{rem}

\subsection{Iterated relations} 

Using Equality \eqref{eq: blockcompact} and \eqref{eq:RjlB}, we obtain 
\begin{equation}
\label{eq:blockcompactmonome}
\f(\z)^{\bmu_j} = \left(A_{\k}(\z)\f(T_\k\z)\right)^{\bmu_j} = \sum_{l=1}^t R_{j,l}(A_\k(\z))\f(T_\k\z)^{\bmu_l} \, ,
\end{equation}
 for every  $j$, $1\leq j \leq t$. 
We deduce from \eqref{eq:blockcompactmonome} that 
\begin{equation}
\label{eq:mahlerblock}
\f(\z)^{\bmu} = \bR_\k(\z) \f(T_\k\z)^{\bmu}\, ,
\end{equation}
for all $\k \in \N^r$, where $\bR_\k(\z)$ is defined as in \eqref{eq: rkz}. 

For every $i$, $1 \leq i \leq r$,  let $b_i(\z_i) \in\Q[\z_i]$ denote the least common multiple 
of the denominators of the  coordinates of $A_i(\z_i)$. Hence the matrix $b_i(\z_i)A(\z_i)$ 
has coefficients in $\Q[\z_i]$. For every $\k=(k_1,\ldots,k_r) \in \N^r$, we set 
$$
b_\k(\z):=\prod_{i=1}^r \prod_{j=0}^{k_i-1} b_i(T_i^{j}\z_i)^{d_i}\, ,
$$
so that the matrix $b_\k(\z)\bR_\k(\z)$ has coefficients in  $\Q[\z]$. 
For all $\k \in \N^r$, Equality \eqref{eq:mahlerblock} implies the following equality in $\Q\{\z\}[\Y]$:   
\begin{eqnarray}\label{eq: iterationF}
\nonumber F(\Y b_\k(\z),\z) &=& \btau \Y b_\k(\z) \f(\z)^\bmu \\
&=& \btau \Y b_\k(\z)\bR_\k(\z) \f(T_\k\z)^\bmu \\
\nonumber &=& F(\Y b_\k(\z)\bR_\k(\z),T_\k\z)\, .
\end{eqnarray}
Every point $\balpha_i$ being regular with respect to the system \eqref{eq:mahleriT}, 
the number $b_\k(\balpha)$ is nonzero for all  $\k \in \N^r$. 
From \eqref{eq: annulationF} and the fact that $F$ is linear in $\Y$, we deduce that  
\begin{equation}
\label{eq: annulationFitere}
F(\bR_{\k}(\balpha),T_{\k}\balpha)=0 \,, \; \forall \k \in \N^r\,.
\end{equation}

\subsection{The matrix $\bTheta_l(\z)$} 

From now on, we fix a nonnegative integer $l_0$ and a relation matrix $\bphi(\z)$ 
satisfying the properties of  Lemma 
\ref{lem:zeroMasserextended}.  Set  
$$
\bxi := T_{\k_{l_0}}\balpha \, .
$$
Properties (a) and (b) in Lemma \ref{lem:zeroMasserextended} ensure the existence of a 
real number $r_1$ such that $0<\Vert \bxi \Vert<r_1<1$ and such that all the power series 
$f_{1,1}(\z),\ldots,f_{r,m_r}(\z)$ have a radius of convergence larger than $r_1$.  
Then, by  Properties (c) and (d) in Lemma \ref{lem:zeroMasserextended}, we can choose a real number 
$r_2$ satisfying  $0< \Vert \bxi \Vert +r_2< r_1$ and such that the 
coefficients of the matrix $\bphi(\z)$ are analytic on the polydisc $\mathcal D(\bxi,r_2)$. 
Note that Property (a) in Lemma \ref{lem:zeroMasserextended} and  Lemma \ref{lem:choixkl} imply that 
$\Vert T_{\k_l}\balpha \Vert \leq \Vert T_{\k_{l_0}}\balpha \Vert$  for $l\geq l_0$. 
For every $l\geq l_0$, we set   
\begin{equation}\label{eq: theta}
\bTheta_l(\z):= \bR_{\k_{l_0}}(\balpha)\bphi(T_{\k_{l_0}}\balpha)^{-1}\bphi(\z)\bR_{\k_l-\k_{l_0}}(\z)\, .
\end{equation}
By \eqref{eq:Rk1k2}, we have  $\bTheta_l(\bxi) = \bR_{\k_l}(\balpha)$, for all $l\geq l_0$.

\begin{rem}\label{rem: thetak}
By Lemma \ref{lem:zeroMasserextended}, the coefficients of 
$\bTheta_{l_0}(\z)$ are analytic on the polydisc   
$\mathcal D(\bxi,r_2)$.  On the other hand, one has 
$$
\bTheta_l(\z)=\bTheta_{l-1}(\z)\bR_{\k_l-\k_{l-1}}(T_{\k_{l-1}-\k_{l_0}}\z)\,, \; \forall l\geq l_0\,.$$ 
This  implies that, for every $l\geq l_0$,  the coefficients of 
$\bTheta_{l}(\z)$ are analytic on some neighborhood of $\bxi$, that  is on some polydisc 
$\mathcal D(\bxi,r_l)\subset \mathcal D(\bxi,r_2)$.  
Also, the coefficients of 
$b_{\k_l-\k_{l_0}}(\z)\bTheta_l(\z)$ are analytic on the polydisc $ \mathcal D(\bxi,r_2)$.
In what follows, we will consider the expression  
$F(\bTheta_l(\z),T_{\k_l-\k_{l_0}}\z)$. Formally, it is a polynomial in 
$f_{1,1}(T_{\k_l-\k_{l_0}}\z),\ldots,f_{r,m_r}(T_{\k_l-\k_{l_0}}\z)$ and the coordinates of 
$\bTheta_l(\z)$. 
Note that it also defines an analytic functions on  
$\mathcal D(\bxi,r_l)\subset \mathcal D(\bxi,r_2)$. 
In addition, $F(\bTheta_{l_0}(\z),\z)$ is analytic on 
 $\mathcal D(\bxi,r_2)$. 
Indeed, the functions $f_{1,1}(\z),\ldots,f_{r,m_r}(\z)$ are analytic on 
$\mathcal D(\boldsymbol{0},r_1)\supset\mathcal D(\bxi,r_2)$, while our choice of  
$l_0$ ensures that the coordinates of   
$\bTheta_{l_0}(\z)$ are analytic on $\mathcal D(\bxi,r_2)$. 
\end{rem}

\subsection{The key lemma}\label{sec: cle}

In this section, we prove the following result from which we will deduce 
easily Theorem \ref{thm: families} in Section \ref{sec:endproof}.  

\begin{lem}
\label{lem:nulliteF}
One has  
$F(\bTheta_{l_0}(\z),\z)=0$.
\end{lem}

Let us first briefly describe the general strategy we use for proving this key lemma.   
The scheme of the proof is classical and takes its source in the early work of Mahler \cite{Ma30b}.  
It was gradually improved and refined by Kubota \cite{Ku77}, Loxton and van der Poorten \cite{LvdP82}, 
and Ku.\ Nishioka \cite{Ni94,Ni96,Ni_Liv}. It takes here a more complicate shape, involving the matrices 
$\bTheta_{l}$. This is due to the fact  that we have to deal with a bunch of Mahler systems of the form \eqref{eq:mahleriT} without any restriction on the matrices $A_i(\z_i)$. 

Assuming by contradiction that 
$F(\bTheta_{l_0}(\z),\z)\not=0$, we construct, for every triple of nonnegative integers  
$(\delta_1,\delta_2,l)$, $l\geq l_0$,  
an auxiliary function of the form 
$$
E(\bTheta_l(\z),T_{\k_l-\k_{l_0}}\z) = \sum_{j=0}^{\delta_1} P_j(\bTheta_l(\z),T_{\k_l-\k_{l_0}}\z)F(\bTheta_l(\z),T_{\k_l-\k_{l_0}}\z)^j \,,
$$
where $P_j(\Y,\z)$ is a polynomial of degree at most $\delta_1$ in each indeterminate $y_{i,j}$ 
and of total degree at most  $\delta_2$ in the indeterminates $z_{i,j}$. 
Recall that $E(\bTheta_l(\bxi),T_{\k_l-\k_{l_0}}\bxi)=E(R_{\k_l}(\balpha),T_{\k_l} \balpha)$ and that 
$$F(\bTheta_l(\bxi),T_{\k_l-\k_{l_0}}\bxi)=F(R_{\k_l}(\balpha),T_{\k_l} \balpha)=0\, .
$$ 
Hence $E(R_{\k_l}(\balpha),T_{\k_l} \balpha)=P_0(R_{\k_l}(\balpha),T_{\k_l} \balpha)$. 
As our construction ensures that $P_0\not\in \mathcal I$, we have 
$P_0(R_{\k_l}(\balpha),T_{\k_l} \balpha)\not=0$ for many integers $l$, and we can use Liouville's 
inequality to find a lower bound for $E(R_{\k_l}(\balpha),T_{\k_l} \balpha)$. 
On the other hand, we have many choices for the polynomials $P_i$ in the construction of 
our auxiliary function. This level of freedom is used to show that, 
for a good choice of such polynomials,   
the quantity $E(R_{\k_l}(\balpha),T_{\k_l} \balpha)$ is small enough. More precisely,  we 
obtain a contradiction between the upper and the lower bounds when the parameter $\delta_1$ is sufficiently large, 
the parameter $\delta_2$ is 
sufficiently large with respect to $\delta_1$, and the parameter $l$ is sufficiently large with respect to  
 $\delta_1$ and $\delta_2$.

\subsection*{Warning}
The auxiliary function $E(\bTheta_l(\z),T_{\k_l-\k_{l_0}}\z) $ can be though of as a 
simultaneous Pad\'e approximant of  type I for the first 
$\delta_1$th powers of  $F(\bTheta_l(\z),T_{\k_l-\k_{l_0}}\z)$. However, we have to be careful:  
$F(\bTheta_l(\z),T_{\k_l-\k_{l_0}}\z)$ it is not necessarily a power series in $\z$. It is a linear combination of the power series $f_{1,1}(T_{\k_l-\k_{l_0}}\z),\ldots,f_{r,m_r}(T_{\k_l-\k_{l_0}}\z)$ 
whose coefficients are only known to be algebraic over $\Q(\z)$. We only know that 
it is analytic at $\bxi$.

\medskip

 In what follows, we argue by contradiction, assuming that 
\begin{equation}\label{eq: notzero}
F(\bTheta_{l_0}(\z),\z)\not=0\,.
\end{equation}
We divide the proof of Lemma \ref{lem:nulliteF} into 
 four steps. 
 


\subsubsection{First step: construction of the auxiliary function}\label{prooflemmanullite}

Given a formal power series $E=\sum_{\lambd} e_\lambd(\Y)\z^\lambd \in \Q[\Y][[\z]]$ 
and an integer $p>0$, we let  
$$
E_p:=\sum_{\vert \lambd \vert < p} e_\lambd(\Y)\z^\lambd \in \Q[\Y,\z]
$$
denote the truncation of  $E$ at order $p$ with respect to $\z$. 
We recall that  
$\I^\perp(\delta_1,\delta_2)$ is a complement to  
$\I(\delta_1,\delta_2)$ in $\Q[\Y,\z]_{\delta_1,\delta_2}$. 

\begin{lem}
\label{lem:fonctionauxiliaire}
Let $\delta_1\geq 0$ be an integer. For all integers $\delta_2$, $\delta_2\gg \delta_1$, 
there exist polynomials $P_i\in \I^{\perp}(\delta_1,\delta_2)$, 
$0 \leq i \leq \delta_1$, not all zero,  
 such that the formal power series  
$$
E'(\Y,\z) := \displaystyle\sum_{j=0}^{\delta_1} P_j(\Y,\z)F(\Y,\z)^j \in \Q[\Y][[\z]]
$$ 
satisfies 
$E'_p(\bTheta_l(\z),T_{\k_l-\k_{l_0}}\z)=0$ for all $l\geq l_0$, where    
$p=\left\lfloor \frac{\delta_1^{1/N}\delta_2}{2^{(t^2+2)/N}}\right\rfloor$.
\end{lem}

\begin{proof}
Set  
$$
\J(\delta_1,\delta_2) := 
\{ P \in \Q[\z,\Y]\ : \ P(\bR_{\k_{l_0}}(\balpha)\bphi(T_{\k_{l_0}}\balpha)^{-1}\Y,\z )\in \I(\delta_1,\delta_2)\}\,  .
$$
The $\Q$-vector spaces  $\J(\delta_1,\delta_2)$ and $\I(\delta_1,\delta_2)$ have same dimension.  
This  follows directly from the fact that the map  
$$
\begin{array}{ccc} \Q[\Y,\z]_{\delta_1,\delta_2} & \rightarrow & \Q[\Y,\z]_{\delta_1,\delta_2}
\\ P(\Y,\z) & \mapsto & P(\bR_{\k_{l_0}}(\balpha)\bphi(T_{\k_{l_0}}\balpha)^{-1}\Y,\z )\, 
\end{array} 
$$
is an isomorphism, the matrix $\bR_{\k_{l_0}}(\balpha)\bphi(T_{\k_{l_0}}\balpha)^{-1}$ being invertible.  
Furthermore, we have  
\begin{equation}\label{eq: thetazero}
P(\bTheta_l(\z),T_{\k_l-\k_{l_0}}\z)=0\,, \;  \forall P\in \J(\delta_1,\delta_2),\, \forall l\geq l_0\,.
\end{equation} 
Indeed, if 
 $P \in \J(\delta_1,\delta_2)$, then  
$P(\bR_{\k_{l_0}}(\balpha)\bphi(T_{\k_{l_0}}\balpha)^{-1}\Y,\z )\in \I(\delta_1,\delta_2)$, and Lemma 
\ref{lem: phi2} implies that  
$$
P(\bR_{\k_{l_0}}(\balpha)\bphi(T_{\k_{l_0}}\balpha)^{-1}\bphi(\z)\bR_\k(\z),T_\k\z )=0\, , 
\; \forall \k \in V_+^\perp\,.
$$
For $l\geq l_0$, using the previous equality with $\k:=\k_l-\k_{l_0}$, we obtain that 
$$
P(\bR_{\k_{l_0}}(\balpha)\bphi(T_{\k_{l_0}}\balpha)^{-1}\bphi(\z)\bR_{\k_l-\k_{l_0}}(\z),T_{\k_l-\k_{l_0}}\z )=0\, .
$$
By \eqref{eq: theta}, we thus have $P(\bTheta_l(\z),T_{\k_l-\k_{l_0}}\z)=0$.

Let $p$ be as in the lemma and let us consider  
the three $\Q$-linear maps: 
\begin{eqnarray*}
&\left\{\begin{array}{cc}
\prod_{j=0}^{\delta_1} \I^{\perp}(\delta_1,\delta_2)
\\
(P_0(\Y,\z),\ldots,P_{\delta_1}(\Y,\z))
\end{array} \right.&
\\
& \big{\downarrow}&
 \\
& \left\{
\begin{array}{c} \Q[\Y]_{2\delta_1}[[\z]]
\\ E'(\Y,\z) := \sum_{j=0}^{\delta_1} P_j(\Y,\z)F(\Y,\z)^j
\end{array} \right.&
\\
&\big{\downarrow}&
\\
&\left\{\begin{array}{c} \Q[\Y,\z]_{2\delta_1,p-1} 
\\ 
 E'_p(\Y ,\z)
 \end{array}\right.&
 \\
&\big{\downarrow}&
\\
&\left\{\begin{array}{c}  \Q[\Y,\z]_{2\delta_1,p-1}/\J(2\delta_1,p-1) 
\\ 
 E'_p(\Y,\z) \mod \J(2\delta_1,p-1)
 \end{array}\right.&
\end{eqnarray*}
Note that these maps are well-defined.  
By Lemma \ref{dimensionespace}, the dimension of the $\Q$-vector space $\I^\perp(\delta_1,\delta_2)$
is at least equal to $\frac{c_1(\delta_1)}{2}\delta_2^{N}$, assuming  that $\delta_2$ is large enough.  
It follows that 
\begin{equation}\label{eq:dimev}
\dim_{\Q} \left(\prod_{j=0}^{\delta_1} \I^\perp(\delta_1,\delta_2)\right) \geq 
 \frac{c_1(\delta_1)}{2}(\delta_1+1)\delta_2^{N} \,.
\end{equation}
For every pair of nonnegative integers $(n,m)$, set 
$$
\overline{\J}(n,m):= \Q[\Y,\z]_{n,m}/\J(n,m)\,.
$$  
Since $\J(\delta_1,\delta_2)$ and $\I(\delta_1,\delta_2)$ have same dimension, Lemma  \ref{lem:majorationespacepolynome} implies that 
$$
\dim_{\Q} \overline{\J}(2\delta_1,p-1)\leq 
2^{t^2} \dim_{\Q} \overline{\J}(\delta_1,p-1)\, .
$$
Now, if  $\delta_2$ is sufficiently large, Lemma  \ref{dimensionespace} ensures that 
$$
\dim_{\Q} \overline{\J}(\delta_1,p-1) \leq 2c_1(\delta_1)p^{N} \,.
$$ 
On the other hand, the choice of $p$ ensures that 
$$
2^{t^2}\left(2c_1(\delta_1)p^N\right) <  \frac{c_1(\delta_1)}{2}(\delta_1+1)\delta_2^{N}
$$ 
and \eqref{eq:dimev} implies that   
$$
\dim_{\Q} \left(\prod_{j=0}^{\delta_1} \I^\perp(\delta_1,\delta_2)\right) > \dim_{\Q} \left( \Q[\Y,\z]_{2\delta_1,p-1}/\J(2\delta_1,p-1)\right)\,.
$$
Hence the $\Q$-linear map defined by  
$$(P_0(\Y,\z),\ldots,P_{\delta_1}(\Y,\z))\mapsto E'_p(\Y,\z)\mod \J(2\delta_1,p-1)$$ 
has a nontrivial kernel.  
We deduce the existence of polynomials $P_0,\ldots,P_{\delta_1}$ in $\I^{\perp}(\delta_1,\delta_2)$, 
not all zero, 
such that 
$E'_p \in \J(2\delta_1,p-1)$.  
By \eqref{eq: thetazero}, we obtain that $E'_p(\bTheta_l(\z),T_{\k_l-\k_{l_0}}\z)=0$ for all $l\geq l_0$.  
This ends the proof. 
\end{proof}

Let $E'$ be a formal power series satisfying the properties of  Lemma \ref{lem:fonctionauxiliaire} and 
let $v_0$ be the smallest 
index such that  the polynomial $P_{v_0}$ is nonzero.  Then the formal power series  
\begin{equation}\label{eq: E}
E(\Y,\z) := \sum_{j\geq v_0} P_j(\Y,\z)F(\Y,\z)^{j-v_0} \in \Q[\Y][[\z]]\,
\end{equation}
is the auxiliary function that we were looking for.  Note that  we have 
\begin{equation}\label{eq: EE'}
E(\Y,\z)F^{v_0}(\Y,\z) = E'(\Y,\z)\,.
\end{equation}

\subsubsection{Second step: upper bound for $\vert E(\bR_{\k_l}(\balpha),T_{\k_l}\balpha)\vert$}

The aim of this section is to prove that there exists a real number $c_2>0$ such that  
\begin{equation}\label{eq: majprinc}
\vert E(\bR_{\k_l}(\balpha),T_{\k_l}\balpha) \vert \leq e^{-c_2e^l\delta_1^{1/N}\delta_2 }\, ,
  \;\; \forall l \gg \delta_2 \gg \delta_1.
\end{equation}

\medskip

Let us first observe that 
$$
b_{\k_l-\k_{l_0}}(\z)^{(t^2+1)\delta_1}E'(\bTheta_l(\z),T_{\k_l-\k_{l_0}}\z)
$$ 
defines an analytic function on $\mathcal D(\bxi,r_2)$. Then, on this polydisc, 
it has a power series expansion 
\begin{equation}\label{eq:developpement_E'}
b_{\k_l-\k_{l_0}}(\z)^{(t^2+1)\delta_1}E'(\bTheta_l(\z),T_{\k_l-\k_{l_0}}\z)= 
\sum_{\lambd \in \N^N} \epsilon_{\lambd,l} (\z-\bxi)^{\lambd}\, ,
\end{equation}
where $\epsilon_{\lambd,l} \in \C$. 
Let $p$ be defined as in Lemma \ref{lem:fonctionauxiliaire}. We first prove the following result. 

\begin{lem}\label{lem:majo_epsilonk}
There exists a positive real number $\gamma$ that does not depend on $\delta_1,\delta_2$, $\lambd$, and $l$,  such that 
$$
\vert \epsilon_{\lambd,l}\vert \leq e^{-\gamma e^l p}\, ,\; \forall l \gg \delta_1,\delta_2,\lambd\,.
$$
\end{lem}

\begin{proof}
Set 
$$
G(\Y,\z):=E'(\Y,\z) - E'_p(\Y,\z)\in \Q\{\z\}[\Y] \, . 
$$ 
By Lemma \ref{lem:fonctionauxiliaire}, we have  
\begin{equation}\label{eq: GE'}
G(\bTheta_l(\z),T_{\k_l-\k_{l_0}}\z)=E'(\bTheta_l(\z),T_{\k_l-\k_{l_0}}\z) \,,\; \forall l\geq l_0\,.
\end{equation} 
Let $\bnu_1,\ldots,\bnu_s$ denote an enumeration of the $t \times t$ matrices with nonnegative integer coefficients that can be decomposed as  
$$
\bnu_l = M_l + N_l\, ,
$$
where $M_l$ has coefficients in the set  $\{0,1,\ldots,\delta_1\}$ and  
$N_l$ is a matrix with nonnegative integer coefficients whose sum  is at most $\delta_1$. 
The sum of the coefficients of each matrix  $\bnu_i$ is at most equal to 
 $(t^2+1)\delta_1$. There exists a unique decomposition of the form 
$$
G(\Y,\z) = \sum_{i=1}^s \sum_{\vert \lambd \vert \geq p} g_{\lambd,i} \z^\lambd \Y^{\bnu_i}\, ,
$$
where $g_{\lambd,i} \in \C$. For every $i$, $1\leq i \leq s$, we define the formal power series 
$$
G_i(\z):=\sum_{\vert \lambd \vert \geq p} g_{\lambd,i} \z^\lambd \in \mathbb C[[\z]]\, .
$$
By definition of $F(\Y,\z)$, these series belong to $\Q[\z,\f(\z)]$. In particular, they are analytic on some 
polydisc  $\mathcal D(\boldsymbol{0},r)$ with $r>r_1$. 
By  \eqref{eq:CauchyHadamard}, there exists a positive real number $\gamma_1(\delta_1,\delta_2)$ 
such that 
\begin{equation}\label{eq:majoration_g}
\vert g_{\lambd,i} \vert \leq \gamma_1(\delta_1,\delta_2)r_1^{-\vert \lambd \vert}\, , \; \forall \lambd \in \N^N\,.
\end{equation}
On the other hand, Lemma \ref{lem:uniformiteconvergencefamilles} 
ensures the existence of two positive real numbers 
$\kappa_1$ and $\kappa_2$ such that
\begin{equation}
\label{eq: EncadrementTk}
\kappa_1 e^l \vert \lambd \vert \leq \vert \lambd T_{\k_l} \vert \leq \kappa_2 e^l \vert \lambd \vert\, \;\; \forall l,\forall \lambd.
\end{equation}
For every $l \geq l_0$, $G_i(T_{\k_l-\k_{l_0}}\z)$ can thus be written as 
\begin{equation}\label{eq:developpement_G_Tk_0}
G_i(T_{\k_l-\k_{l_0}}\z)= \sum_{\vert \lambd \vert \geq  \kappa_1 e^l p} g_{\lambd,i,l} \z^\lambd\, ,
\end{equation}
with  $g_{\lambd,i,l} \in \C$. Furthermore, this power series is absolutely convergent 
on the polydisc $\mathcal D(\boldsymbol{0},r_1)$. Since the matrix $T_{\k_l-\k_{l_0}}$ is invertible, 
has nonnegative coefficients, and  
$r_1 \leq 1$, we deduce from  \eqref{eq:majoration_g} that  
\begin{equation}
\label{eq:majoration_gk}
\vert g_{\lambd,i,l} \vert \leq \gamma_1(\delta_1,\delta_2)r_1^{-\vert \lambd  \vert}\, ,
\end{equation}
for all $\lambd$, $i\in\{1,\ldots,s\}$, and $l \geq l_0$. 
On the other hand, every function  
$G_i(T_{\k_l-\k_{l_0}}\z)$, $1\leq i \leq s$, $l \geq l_0$, is analytic on the polydisc $\mathcal D(\bxi,r_2)$. 
The power series expansion of $G_i(T_{\k_l-\k_{l_0}}\z)$ at $\bxi$ is thus absolutely convergent on 
 $\mathcal D(\bxi,r_2)$, and we can write 
\begin{equation}\label{eq:developpement_G_alpha}
G_i(T_{\k_l-\k_{l_0}}\z)= \sum_{\lambd \in \N^N} h_{\lambd,i,l} (\z - \bxi)^\lambd \, ,
\end{equation}
where $h_{\lambd,i,l} \in \C$. 
Since by assumption $\mathcal D(\bxi,r_2) \subset \mathcal D(\boldsymbol{0},r_1)$, 
the two power series expansions \eqref{eq:developpement_G_Tk_0} and \eqref{eq:developpement_G_alpha} 
match on $\mathcal D(\bxi,r_2)$ . Using the equality  
$\z^{\lambd}=((\z-\bxi)+ \bxi)^\lambd$ and identifying, for every $\lambd \in \N^N$, the terms in  
$(\z - \bxi)^{\lambd}$ in \eqref{eq:developpement_G_Tk_0} and \eqref{eq:developpement_G_alpha}, 
we obtain that 
\begin{equation}\label{eq:identit_gk}
h_{\lambd,i,l} = \sum_{\substack{\vert \gamm \vert \geq  \kappa_1 e^l p 
\\ \gamm \geq \lambd}} \binom{\gamm}{\lambd} g_{\gamm,i,l} \bxi^{\gamm-\lambd}  \, .
\end{equation} 
For $\gamm \geq \lambd$, one has
   \begin{equation}\label{eq:majbinom}
   \binom{\gamm}{\lambd} = \prod_{i=1}^N \frac{\gamma_i
   !}{(\gamma_i-\lambda_i)!\lambda_i!}
   \leq \prod_{i=1}^N \gamma_i^{\lambda_i} \leq \vert \gamm
   \vert^{\vert \lambd \vert}\, .
   \end{equation}
   Given $\lambd \in \N^N$, 
    $\vert \lambd \vert < \kappa_1 e^l p$ as soon as $l$ is large enough, and  
   since $ \Vert \bxi \Vert<r_1$, we infer from
   \eqref{eq:majoration_gk} and \eqref{eq:identit_gk}
   the existence of a real number $\gamma_2>0$ that  does not
   depend on $\delta_1$, $\delta_2$, $\lambd$, and $l$,   
   such that
   \begin{equation}\label{eq:majoration_g_alpha}
   \vert h_{\lambd,i,l}\vert \leq   e^{-\gamma_2e^l p}\, , \; \forall l \gg \delta_1,\delta_2,\lambd\,.
   \end{equation}

   By Remark \ref{rem: thetak}, the monomial $(b_{\k_l-\k_{l_0}}(\z)\bTheta_l(\z))^{\bnu_i}$
   is analytic on  $\mathcal D(\bxi,r_2)$ for every $i$, $1\leq i \leq s$,
   and every $l \geq l_0$. 
   Multiplying by $b_{\k_l-\k_{l_0}}(\z)^{(t^2+1)\delta_1 - \vert
   \bnu_i \vert}$, we obtain on
   $\mathcal D(\bxi,r_2)$ a power series expansion of the form
   \begin{equation}\label{eq:developpement_thetak}
   b_{\k_l-\k_{l_0}}(\z)^{(t^2+1)\delta_1 - \vert \bnu_i
   \vert}(b_{\k_l-\k_{l_0}}(\z)\bTheta_l(\z))^{\bnu_i} =
   \sum_{\lambd\in \N^N} \theta_{\lambd,i,l}(\z-\bxi)^{\lambd}\, ,
   \end{equation}
   where $\theta_{\lambd,i,l} \in \C$.

 Given $Q(\z)\in \Q[\z]$ and $l \geq 0$, 
    the polynomial $Q_l(\z):=Q(T_{\k_l-\k_{l_0}}\z)$ can be converted into  
    a polynomial in $(\z-\bxi)$: 
   $$
   Q_l(\z) = \sum_{\lambd } q_{\lambd,l}(\z-
   \bxi)^\lambd\, .
   $$
  Using the relation
   $$
   \z^\gamm = \sum_{\lambd \leq \gamm}
   \binom{\gamm}{\lambd}\bxi^{\lambd-\gamm} (\z - \bxi)^\lambd\, ,
   $$
   \eqref{eq: EncadrementTk} and \eqref{eq:majbinom}, we obtain that,  for every $\lambd$, 
   $\vert q_{\lambd,l} \vert = \mathcal O(e^l)$, where 
   the underlying constant in the $\mathcal O$ notation depends both on $Q(\z)$ and 
   $\lambd$. Set $ \Gamma_\k(\z):=b_{\k}(\z)\bR_{\k}(\z)\in\Q[\z]$ and 
   $$
   \Gamma_{\k,l}(\z):=b_{\k}(T_{\k_l-\k_{l_0}}\z)\bR_{\k}(T_{\k_l-\k_{l_0}}\z)\, .
   $$
   Since there are only finitely many differences $\k_{l+1}-\k_l$, the
   coefficients in $(\z-\bxi)^\lambd$ for the entries of $\Gamma_{\k_{l+1}-\k_l,l}(\z)$, 
   seen as polynomials in $\z-\bxi$, belong to $\mathcal O(e^l)$. Using
   the recurrence relation
   $$
   b_{\k_{l+1}-\k_{l_0}}(\z)\bTheta_{l+1}(\z)=b_{\k_l-\k_{l_0}}(\z)\bTheta_l(\z)\Gamma_{\k_{l+1}-\k_l,l}(\z) \,,
   $$
   we infer from \eqref{eq:bounded2} that, for every $\lambd$, 
   there exists a  real number
   $\gamma_3(\delta_1,\lambd)>0$ that does not depend on $i$, $\delta_2$, and $l$, and such that
   \begin{equation}\label{eq:theta_ik}
   \vert  \theta_{\lambd,i,l} \vert < e^{\gamma_3(\delta_1,\lambd){l}} \,, \;\; \forall i,\, 1\leq i \leq s, \, \forall l \geq l_0\,.
   \end{equation}

From  \eqref{eq:developpement_E'}, \eqref{eq: GE'}, \eqref{eq:developpement_G_alpha}, and \eqref{eq:developpement_thetak}, we deduce  that 
\begin{equation}\label{eq:hthetae}
\sum_{i=1}^s \left(\sum_{\lambd \in \N^N} h_{\lambd,i,l} (\z-\bxi)^\lambd \right)\left(\sum_{\lambd \in \N^N} \theta_{\lambd,i,l} (\z-\bxi)^\lambd \right)  =\sum_{\lambd \in \N^N} \epsilon_{\lambd,l} (\z-\bxi)^{\lambd} \,.
\end{equation}
Moreover, since the power series involved in \eqref{eq:developpement_G_alpha} and  \eqref{eq:developpement_thetak} are absolutely convergent on $\mathcal D(\bxi,r_2)$, 
we obtain 
\begin{eqnarray}\label{eq:developpement_total_G}
\nonumber  \sum_{i=1}^s \left(\sum_{\lambd \in \N^N} h_{\lambd,i,l} (\z-\bxi)^\lambd \right)
\left(\sum_{\lambd \in \N^N} \theta_{\lambd,i,l} (\z-\bxi)^\lambd \right)  
\\ = 
  \sum_{\lambd \in \N^N} \sum_{i=1}^s \sum_{\gamm \leq \lambd} h_{\gamm,i,l} \theta_{\lambd-\gamm,i,l}(\z-\bxi)^\lambd\, .
\end{eqnarray}
Finally, identifying the terms in $(\z-\bxi)^{\lambd}$ in 
\eqref{eq:hthetae} thanks to \eqref{eq:developpement_total_G}, we get that 
$$
\epsilon_{\lambd,l} = \sum_{i=1}^s \sum_{\gamm \leq \lambd} h_{\gamm,i,l}\theta_{\lambd-\gamm,i,l}\, .
$$
Now, we fix $\lambd \in \N^N$. By \eqref{eq:majoration_g_alpha} and \eqref{eq:theta_ik}, 
there exists a real number $\gamma_4>0$ that does not depend on $\delta_1$, $\delta_2$, $\lambd$, and $l$, such that 
$$
\vert \epsilon_{\lambd,l}\vert \leq e^{-\gamma_4e^l p}\, ,\;\; \forall l\gg \delta_1,\delta_2,\lambd\,.
$$
Setting $\gamma=\gamma_4$, this ends the proof. 
\end{proof}

Now, let us observe that the function 
$b_{\k_l-\k_{l_0}}(\z)^{(t^2+1)\delta_1}E(\bTheta_l(\z),T_{\k_l-\k_{l_0}}\z)$ is analytic   
on 
$\mathcal D(\bxi,\r_2)$. Hence it has a power series expansion of the form  
\begin{equation}\label{eq:developpementE}
b_{\k_l-\k_{l_0}}(\z)^{(t^2+1)\delta_1}E(\bTheta_l(\z),T_{\k_l-\k_{l_0}}\z) =\sum_{\lambd \in \N^N} 
e_{\lambd,l} (\z-\bxi)^{\lambd}\, ,
\end{equation}
with $e_{\lambd,l} \in \C$. 
Also, $F(\bTheta_{l_0}(\z),\z)^{v_0}$ is analytic on $\mathcal D(\bxi,r_2)$, so that 
\begin{equation}\label{eq:developpementF}
F(\bTheta_{l_0}(\z),\z)^{v_0} = \sum_{\lambd \in \N^N} a_{\lambd}(\z-\bxi)^\lambd\, ,
\end{equation}
with $a_{\lambd} \in \C$. 
Using \eqref{eq: iterationF}, we get that  
$$
F(\bTheta_l(\z),T_{\k_l-\k_{l_0}}\z) = F(\bTheta_{l_0}(\z),\z)\,, \;\; \forall l\geq l_0\,.
$$ 
 This equality is \textit{a priori} valid for $\z$ in 
$\mathcal D(\bxi,r_l)$ (see Remark \ref{rem: thetak}), but 
it extends to $\mathcal D(\bxi,r_2)$ by analytic continuation.   
By \eqref{eq: EE'},  we thus have 
\begin{equation}\label{eq:factorisationE'}
E(\bTheta_l(\z),T_{\k_l-\k_{l_0}}\z)F(\bTheta_{l_0}(\z),\z)^{v_0}=E'(\bTheta_l(\z),T_{\k_l-\k_{l_0}}\z) \,,
\end{equation}
for all $l\geq l_0$ and all $\z\in\mathcal D(\bxi,r_l)$. 
By \eqref{eq: notzero}, $F(\bTheta_{l_0}(\z),\z)$ is nonzero and there thus exists at least  
one nonzero coefficient $a_{\lambd}$ in \eqref{eq:developpementF}. Let us consider an index $\lambd_0$ 
such that $a_{\lambd_0}\not= 0$ with $\vert \lambd_0\vert$ minimal. Multiplying both sides of  
\eqref{eq:factorisationE'} by $b_{\k_l-\k_{l_0}}(\z)^{(t^2+1)\delta_1}$, and identifying the coefficients in  
$(\z - \bxi)^{\lambd_0}$ in their power series expansion on $\mathcal D(\bxi,r_2)$ with the help of   \eqref{eq:developpement_E'}, \eqref{eq:developpementE}, and \eqref{eq:developpementF}, 
we obtain that  
$$
e_{0,l}a_{\lambd_0} = \epsilon_{\lambd_0,l}\, , \;\; \forall l\geq l_0\,.
$$
Since $\bTheta_{l}(\bxi)=\bR_{\k_l}(\balpha)$, we infer from Lemma \ref{lem:majo_epsilonk} 
and the definition of $p$ (see Lemma \ref{lem:fonctionauxiliaire}), 
the existence of a  real number  $\beta_1>0$ that does not depend on  
$\delta_1$, $\delta_2$, and $l$, 
such that 
\begin{eqnarray}\label{eq:majoEb}
\nonumber & & \left\vert b_{\k_l-\k_{l_0}}(\bxi)^{(t^2+1)\delta_1}E(\bR_{\k_l}(\balpha),T_{\k_l}\balpha)\right\vert 
\\ \nonumber &=& \left\vert b_{\k_l-\k_{l_0}}(\bxi)^{(t^2+1)\delta_1}E(\bTheta_l(\bxi),T_{\k_l-\k_{l_0}}\bxi) \right\vert 
\\ &=& \vert e_{0,l} \vert 
\\ \nonumber & \leq & e^{-\beta_1e^l \delta_1^{1/N}\delta_2}\, ,\;\; \forall l\gg \delta_1,\delta_2\,.
\end{eqnarray}

Now, it just remains to find a lower bound for $\vert b_{\k_l-\k_{l_0}}(\bxi)^{(t^2+1)\delta_1}\vert$. 
By \eqref{eq: EncadrementTk}, there exists a positive real number $\beta_{2}$ that does not depend on 
$l$, $\delta_1$, and $\delta_2$, such that the degree of the polynomial 
$b_{\k_l-\k_{l_0}}(\z)^{(t^2+1)\delta_1}$ is at most equal to  
$\beta_{2}e^l\delta_1$. Each $\balpha_i$ being by assumption regular w.r.t.\ \eqref{eq:mahleriT},  
$b_{\k_l-\k_{l_0}}(\bxi)\neq 0$ for all $l \geq l_0$. Since the numbers $b_{\k_l-\k_{l_0}}(\bxi)$, $l \geq l_0$, 
belong to some fix number field, we infer from Liouville's inequality  the existence of a  real number 
 $\beta_{3}>0$ that does not depend on
$l$, $\delta_1$, and $\delta_2$, such that 
\begin{equation}\label{eq:minorationbk}
\vert b_{\k_l-\k_{l_0}}(\bxi) \vert^{(t^2+1)\delta_1} \geq e^{-\beta_{3} e^l\delta_1}\, , \; \; \forall l,\, \forall \delta_1\, .
\end{equation}
By \eqref{eq:majoEb} and \eqref{eq:minorationbk}, there exists a  real number 
$\beta_{4}>0$ that does not depend on  $l$, $\delta_1$, and  $\delta_2$, such that 
\begin{eqnarray}
\label{eq:majoE}
\nonumber\vert E(\bR_{\k_l}(\balpha),T_{\k_l}\balpha)\vert 
& \leq&  e^{\beta_{3}e^l\delta_1-\beta_{1}e^l\delta_1^{1/N}\delta_2 }\\
\nonumber&\leq &e^{-\beta_{4}e^l\delta_1^{1/N}\delta_2 }\, , \;\; \forall l \gg \delta_2 \gg \delta_1\,.
\end{eqnarray}
 Setting $c_2=\beta_4$, 
this proves the upper bound \eqref{eq: majprinc}.

\subsubsection{Third step: lower bound for $\vert E(\bR_{\k_l}(\balpha),T_{\k_l}\balpha)\vert$}
\label{subsub:min}

Let us first recall that, by \eqref{eq: annulationFitere}, we have 
$$
F(\bR_{\k_l}(\balpha),T_{\k_l}\balpha)=0 \, , \;\; \forall l\,.
$$
By construction of our auxiliary function, we deduce that 
$$E(\bR_{\k_l}(\balpha),T_{\k_l}\balpha)=P_{v_0}(\bR_{\k_l}(\balpha),T_{\k_l}\balpha)\,.$$ 
Furthermore, since this construction ensures that $P_{v_0} \notin \mathcal I$, 
there exists an infinite set of positive integers 
$\mathcal E$ such that 
$$
P_{v_0}(\bR_{\k_l}(\balpha),T_{\k_l}\balpha) \not=0 \, , \;\; \forall l\in\mathcal E\,.
$$
Since $P_{v_0}$ has degree at most $\delta_1$ in each indeterminate $y_{i,j}$, and total degree at most  
$\delta_2$ in $\z$,  
 Liouville's inequality and a computation similar to the previous one ensure the existence of a real 
 number  $c_{3}>0$ that does not depend on $\delta_1$, $\delta_2$, and $l$, such that 
\begin{equation}
\label{eq:minorationE}
\vert E(\bR_{\k_l}(\balpha),T_{\k_l}\balpha)\vert=|P_0(\bR_{\k_l}(\balpha),T_{\k_l}\balpha)| \geq  
e^{-c_{3}e^l\delta_2}\, , \;\; \forall l\in\mathcal E, \delta_2 \geq \delta_1\,. 
\end{equation}


\subsubsection{Fourth step: contradiction}\label{contradiction}
By Inequalities \eqref{eq: majprinc} and \eqref{eq:minorationE}, we obtain that   
$$
e^{-c_{3}e^l\delta_2} \leq \vert E(\bR_{\k_l}(\balpha),T_{\k_l}\balpha)\vert  
\leq e^{-c_{2}e^l\delta_1^{1/N}\delta_2 }\, ,\;\; \forall l\in \mathcal E, \, l\gg \delta_2\gg \delta_1\,.
$$
Finally, we deduce that 
$$
c_{3} \geq c_{2}\delta_1^{1/N}\, .
$$
Since $c_{2}$ 	and $c_{3}$ do not depend on $\delta_1$, 
this inequality provides a contradiction, as soon as $\delta_1$ is large enough.   
This end the proof of Lemma \ref{lem:nulliteF}. \qed

\subsection{End of the proof of Theorem \ref{thm: families}}\label{sec:endproof}

Let us recall that  $d(\z) \in \Q[\z]$ stands for the least common multiple  of the denominators of 
the coefficients  of the matrices  
$\bphi_j(\z)$ defined in \eqref{eq:decompositionPhi}. Hence 
$d(T_{\k_{l_0}}\z)\bTheta_{l_0}(T_{\k_{l_0}}\z)$ 
has coefficients in $\Q[\z,\varphi(T_{\k_{l_0}}\z)]$. Lemma \ref{lem:zeroMasserextended} 
ensures that $d(T_{\k_{l_0}}\balpha)\neq 0$. Let $q(\z)$ denote the least common multiple 
of the denominators of the coefficients of the matrix  
$\bR_{\k_{l_0}}^{-1}(\z)$. Since, for every i, $\balpha_i$ is assumed to be regular w.r.t. \ \eqref{eq:mahleriT}, 
we have that $q(\balpha)\not=0$. 
Similarly, $b_\k(\z)\bR_\k(\z)$ has coefficients in $\Q[\z]$ and  $b_{\k_{l_0}}(\balpha)\not=0$.

By Lemma \ref{lem:nulliteF}, we know that $F(\bTheta_{l_0}(\z),\z)=0$, and substituting 
$T_{\k_{l_0}}\z$ to  $\z$, we obtain that 
$F(\bTheta_{l_0}(T_{\k_{l_0}}\z),T_{\k_{l_0}}\z)=0$. 
The function $F(\Y,\z)$ being linear in $\Y$, we deduce that 
$$
F \left(\frac{b_{\k_{l_0}}(\z)d(T_{\k_{l_0}}\z)q(\z)}
{b_{\k_{l_0}}(\balpha)d(T_{\k_{l_0}}\balpha)q(\balpha)}\bTheta_{l_0}(T_{\k_{l_0}}\z),T_{\k_{l_0}}\z\right)=0 \, .
$$ 
Writing $\bTheta_{l_0}(T_{\k_{l_0}}\z)=\bTheta_{l_0}(T_{\k_{l_0}}\z)\bR_{\k_{l_0}}(\z)^{-1}\bR_{\k_{l_0}}(\z)$, 
and using   \eqref{eq: iterationF}, we get that 
\begin{equation}\label{eq:nulliteF}
F \left(\frac{b_{\k_{l_0}}(\z)d(T_{\k_{l_0}}\z)q(\z)}
{b_{\k_{l_0}}(\balpha)d(T_{\k_{l_0}}\balpha)q(\balpha)}\bTheta_{l_0}(T_{\k_{l_0}}\z)
\bR_{\k_{l_0}}(\z)^{-1},\z\right)=0 \,.
\end{equation}
Set
$$
Q(\z,\X):=\btau \left(\frac{b_{\k_{l_0}}(\z)d(T_{\k_{l_0}}\z)q(\z)}
{b_{\k_{l_0}}(\balpha)d(T_{\k_{l_0}}\balpha)q(\balpha)}
\bTheta_{l_0}(T_{\k_{l_0}}\z)\bR_{\k_{l_0}}(\z)^{-1}\right)\X^\bmu \, ,
$$
where we let $\X^\bmu$ denote the column vector with coordinates 
$\X^{\bmu_1},\ldots,\X^{\bmu_t}$. 
We thus get that  $Q(\z,\X)\in \Q[\z,\varphi(T_{\k_{l_0}}\z),\X]$. 
Since $\varphi\circ T_{\k_{l_0}}$ is analytic at $\balpha$, it follows that   
$Q(\z,\X)\in \overline{Q(\z)}_{\balpha}[\X]$. 
Moreover, since $\bTheta_{l_0}(T_{\k_{l_0}}\balpha)=\bTheta_{l_0}(\bxi)=\bR_{\k_{l_0}}(\balpha)$, 
we deduce that 
$$
Q(\balpha,\X)= \btau  \X^\bmu=P(\X)\, .
$$
Finally, it follows from  \eqref{eq:nulliteF} that 
$$
Q(\z,\f(\z))= 0\,.
$$
This ends the first part of the proof of Theorem  \ref{thm: families}. 

Now, let us assume that $\Q(\z)(\f(\z))$ is a regular extension of $\Q(\z)$. 
Let $\mathbb K$ be an algebraic closure of $\Q(\z)$ containing $\varphi(T_{\k_{l_0}}\z)$. 
By \cite[Chapter VIII]{Lang},  $\Q(\z)(\f(\z))$ and $\mathbb K$ 
are linearly disjoint over $\Q(\z)$.  
Let $\delta$ denote the degree of $\varphi(T_{\k_{l_0}}\z)$ over $\Q(\z)$. Since the functions 
$\varphi(T_{\k_{l_0}}\z)^j$, 
$0\leq j  \leq \delta-1$, are linearly independent over $\Q(\z)$, they remain linearly independent over 
$\Q(\z)(\f(\z))$. Splitting the polynomial $Q$ as  
$$
Q= \sum_{j=0}^{\delta-1} Q_j(\z,\X)\varphi(T_{\k_{l_0}}\z)^j\,,
$$
where $Q_j(\z,\X)\in \Q[\z,\X]$, 
we thus deduce that  
$Q_j(\z,\f(\z))=0$  
for all $j$, $0\leq j\leq \delta-1$. Finally, setting 
$$
R(\z,\X):= \sum_{j=0}^{\delta-1} Q_j(\z,\X)\varphi(T_{\k_{l_0}}\balpha)^j\in \Q[\z,\X]\,,
$$
we obtain that 
$
R(\z,\f(\z))=0$  and  $R(\balpha,\X)=P(\X)$,  
as wanted. 
\qed


\section{Proofs of Theorems \ref{thm: permanence},  \ref{thm: purete2}, \ref{thm: purity},
and of Corollaries \ref{thm: Nishioka} and \ref{cor:2}}
\label{sec: final}

In this section, we complete the proof of our main results.  
Note that we establish Corollary \ref{cor:2} before Theorems \ref{thm: purete2} and \ref{thm: purity}. 
The latter are in fact deduced from Corollary \ref{cor:2}. 

\subsection{Proof of Theorem  \ref{thm: permanence}}
There is nothing more to do.
Theorem  \ref{thm: permanence} just corresponds to the case $r=1$ of Theorem
\ref{thm: families}. \qed

  \subsection{Proof of Corollary \ref{thm: Nishioka}}
   We first note that the inequality
   $$
   {\rm tr.deg}_{\Q}(f_1(\balpha),\ldots,f_m(\balpha)) \leq {\rm
   tr.deg}_{\Q(\z)}(f_1(\z),\ldots,f_m(\z))
   $$
   always holds. 
  Let $\mathbb K$ denote the field of fractions of the ring  
   $\overline{\mathbb Q(\z)}_\balpha$. Since $\mathbb K$ 
   is algebraic over $\Q(z)$, we have
   $$
   {\rm tr.deg}_{\Q(\z)}(f_1(\z),\ldots,f_m(\z)) = {\rm
   tr.deg}_{\mathbb K}(f_1(\z),\ldots,f_m(\z))\, .
   $$
   It thus remains to prove that
   \begin{equation}
   \label{ineg:degtrans}
   {\rm tr.deg}_{\Q}(f_1(\balpha),\ldots,f_m(\balpha)) \geq {\rm
   tr.deg}_{\mathbb K}(f_1(\z),\ldots,f_m(\z))\, .
   \end{equation}
   We  follow the same strategy as the one used for proving the Siegel--Shidlovskii theorem 
   in the framework of $E$-functions (see \cite[Theorem 5.23, p.\,230]{FN98}). 
   We first replace Proposition 5.1 in \cite{FN98} by the following result. 
   
   \begin{lem}
   \label{lem:linearLifting}
   Let $g_1(\z),\ldots,g_\ell(\z)\in \Q\{\z\}$ be related by a linear $T$-Mahler system. 
   Let $\balpha \in (\Q^\star)^n$ be a regular point w.r.t.\ this system and let us assume that the 
   pair $(T,\balpha)$ is admissible. Let $s$ be the maximum number of 
   functions among $g_1(\z),\ldots,g_\ell(\z)$ that are linearly independent
   over $\mathbb K$. Then at least $s$ of the
   numbers $g_1(\balpha),\ldots,g_\ell(\balpha)$ are linearly independent over $\Q$. 
   \end{lem}
   
   \begin{proof}
   Let $t$ denote the dimension of the $\Q$-vector space spanned by the
   numbers $g_i(\balpha)$, $1\leq i \leq \ell$. Without any loss of generality, we assume that
   $g_1(\balpha),\ldots,g_t(\balpha)$ are linearly independent over $\Q$. Then,
   for every $i$, $t<i\leq \ell$, there exist algebraic numbers 
   $\gamma_{i,1},\ldots,\gamma_{i,t}$ such that
   $$
   g_i(\balpha)=\gamma_{i,1}g_1(\balpha)+\cdots +
   \gamma_{i,t}g_t(\balpha)\, .
   $$
    By Theorem \ref{thm: permanence}, there exist
   $p_{i,1}(\z),\ldots,p_{i,\ell}(\z) \in \overline{\mathbb
   Q(\z)}_\balpha$ such that
   $$
   p_{i,1}(\z)g_1(\z)+\cdots +p_{i,\ell}(\z)g_\ell(\z)=0\, ,
   $$
   with $p_{i,i}(\balpha)=1$ and $p_{i,j}(\balpha)=0$ when $t<j\leq \ell$ and
   $j\neq i$. 
   Set 
    $$
    L_i(\z,X_1,\ldots,X_\ell):= p_{i,1}(\z)X_1+\cdots +p_{i,\ell}(\z)X_\ell=0\, , \; t<i\leq \ell\,, . 
    $$
    Note that the linear form $L_i(\balpha,X_1,\ldots,X_m)$ is equal to 
    $$
    X_i + \sum_{j=1}^t \gamma_{i,j} X_j \,.
    $$
    Hence the linear forms $L_i(\balpha,X_1,\ldots,X_m)$, $t<i\leq \ell$, are linearly independent 
    over $\Q$ and  {\it a fortiori}  the corresponding linear forms $L_i(\z,X_1,\ldots,X_\ell)$   
    are 
    linearly independent over $\mathbb K$. It follows that $t\geq s$, as wanted. 
    \end{proof}
   
  \begin{proof}[Proof of Corollary \ref{thm: Nishioka}] 
  Let $D\geq 0$ be an integer. As in \cite[p.\,231]{FN98}, we let
   $\varphi_\balpha(D)$ denote the dimension of the $\Q$-vector space spanned by the 
   monomials of total degree at most $D$ in $f_1(\balpha),\ldots,f_m(\balpha)$. 
   We also let $\varphi_\z(D)$ denote the dimension of the $\mathbb K$-vector space 
   spanned by the monomials of total degree at most $D$ in 
   $f_1(\z),\ldots,f_m(\z)$. Now, we observe that the monomials of total 
   degree at most $D$ in $f_1(\z),\ldots,f_m(\z)$ are also related by a linear 
   $T$-Mahler system for which $\balpha$ remains regular. Indeed, such a system
   can be obtained by first taking the system associated with the matrix $((1)\oplus A(\z))^{\otimes d}$, that is, 
   the $D$th power of Kronecker of
   the matrix $(1)\oplus A(\z)$, where $A(\z)$ is defined as in \eqref{eq:mahler}, and then applying some standard reduction procedure (see 
   \cite[Lemme 5]{Fa}).  
   Then, we infer from Lemma \ref{lem:linearLifting} that 
   \begin{equation}\label{eq: hilbertserre}
   \varphi_\balpha(D) \geq \varphi_\z(D),\, \forall D \in \N\, .
   \end{equation}
    By the Hilbert-Serre Theorem (see for example \cite[Theorem 42, p.\,235]{SZ60}), 
    for sufficiently large $D$,  
    $\varphi_\balpha(D)$ and $\varphi_\z(D)$ are polynomials in $D$ whose degree are respectively 
    ${\rm tr.deg}_{\Q}(f_1(\balpha),\ldots,f_m(\balpha))$ and 
    ${\rm tr.deg}_{\mathbb K}(f_1(\z),\ldots,f_m(\z))$. 
   Hence \eqref{eq: hilbertserre} implies that Inequality \eqref{ineg:degtrans} holds, as wanted. 
   \end{proof}

\subsection{Proof of Corollary \ref{cor:2}}

We first need the two following simple results.

\begin{lem}\label{lem:independantvariables}
Let $\z=(z_{1,1},\ldots,z_{1,n_1},z_{2,1},\ldots,z_{r,n_r})$ be a tuple of $n_1+\cdots+n_r$ distinct variables.
For every $i$, $1\leq i \leq r$, let $f_{i,1}(\z_i),\ldots,f_{i,m_i}(\z_i) \in \Q[[\z_i]]$ be some power series, 
where $\z_i=(z_{i,1},\ldots,z_{i,n_i})$. Then
\begin{eqnarray*}
{\rm tr.deg}_{\Q(\z)}\{f_{i,j}(\z_i) :  1\leq i \leq r,\, 1\leq j \leq m_i\} =\hspace{3.5cm}\\
\hspace{5.5cm}\sum_{i=1}^r {\rm tr.deg}_{\Q(\z_i)}\{ f_{i,j}(\z_i) : 1\leq j \leq m_i\}\, .
\end{eqnarray*}
\end{lem}
\begin{proof}
The result follows directly from the fact that the sets of variables
$\{z_{1,1},\ldots,z_{1,n_1}\},\ldots,\{z_{r,1},\ldots,z_{r,n_r}\}$  are disjoint.
\end{proof}

\begin{lem}\label{lem:transcendancedegree}
Let $\mathcal E_1, \ldots,\mathcal E_r,\mathcal F_1,\ldots, \mathcal F_r$ be nonempty finite sets of complex numbers
such that $\mathcal E_i \subset \mathcal F_i$ for every $i$, $1 \leq i \leq r$. Let us assume that
$$
{\rm tr.deg}_{\Q}\left(\bigcup_{i=1}^r \mathcal F_i\right) = \sum_{i=1}^r {\rm tr.deg}_{\Q}(\mathcal F_i)\, .
$$
Then
$$
{\rm tr.deg}_{\Q}\left(\bigcup_{i=1}^r \mathcal E_i\right) = \sum_{i=1}^r {\rm tr.deg}_{\Q}(\mathcal E_i)\, .
$$
\end{lem}

\begin{proof}
Suppose first that all elements of each $\mathcal F_i$, $1 \leq i \leq r$, are algebraically independent.
By assumption, all elements of the set $\bigcup_{i=1}^r \mathcal F_i$ are algebraically independent.
Hence, all elements of the set $\bigcup_{i=1}^r \mathcal E_i$ are also algebraically independent,
and the lemma is proved.
Let us assume now that some elements of the family $\mathcal F_i$, $1 \leq i \leq r$, are algebraically dependent.
For every $i$, we choose a subset of algebraically independent elements
$\mathcal E'_i \subset \mathcal E_i$ such that ${\rm tr.deg}_{\Q}(\mathcal E'_i)
= {\rm tr.deg}_{\Q}(\mathcal E_i)$. Then, we complete the set $\mathcal E'_i$
in a set of algebraically independent elements
$\mathcal F'_i \subset \mathcal F_i$ such that
${\rm tr.deg}_{\Q}(\mathcal F'_i) = {\rm tr.deg}_{\Q}(\mathcal F_i)$.
From the first part of the proof, we have 
$$
{\rm tr.deg}_{\Q}\left(\bigcup_{i=1}^r \mathcal E'_i\right)
= \sum_{i=1}^r {\rm tr.deg}_{\Q}(\mathcal E'_i)\, .
$$
It follows that
$$
{\rm tr.deg}_{\Q}\left(\bigcup_{i=1}^r \mathcal E_i\right)
= {\rm tr.deg}_{\Q}\left(\bigcup_{i=1}^r \mathcal E'_i\right)
=  \sum_{i=1}^r {\rm tr.deg}_{\Q}(\mathcal E'_i)
= \sum_{i=1}^r {\rm tr.deg}_{\Q}(\mathcal E_i)\, ,
$$
which ends the proof.
\end{proof}

We are now ready to prove Corollary \ref{cor:2}.

\begin{proof}[Proof of Corollary \ref{cor:2}] 
We continue with the notation of Theorem \ref{thm: purete2} and \ref{thm: purity}.

Let us first assume that the assumptions of Theorem \ref{thm: purete2} are satisfied.   
We can gather all the linear Mahler systems \eqref{eq:mahleri}
into a big Mahler system of the form \eqref{eq:mahler}, where $A(\z)=A_1(\z_1)\oplus \cdots \oplus A_r(\z_r)$, 
$\z=(z_{1,1},\ldots,z_{r,m_r})$, and  $T:= T_1 \oplus \cdots \oplus T_r$. 
Then, we infer from assumptions (i) and (ii) of Theorem \ref{thm: purete2} and from Theorem \ref{thm:masser} that the pair $(T,\balpha)$ is admissible and that the point $\balpha=(\balpha_1,\ldots,\balpha_r)$ is regular   with respect to this $T$-Mahler system. Hence we can apply Corollary \ref{thm: Nishioka} to this larger system.  
We obtain that 
\begin{eqnarray}
\label{eq:Nishioka1}
{\rm tr.deg}_{\Q}\{f_{i,j}(\balpha_i) : 1\leq i \leq r,\, 1\leq j \leq m_i\} \hspace{4cm}\\
\nonumber \hspace{4cm} ={\rm tr.deg}_{\Q(\z)}\{f_{i,j}(\z_i) : 1\leq i \leq r,\, 1\leq j \leq m_i\}\, .
\end{eqnarray}
On the other hand,  applying Corollary \ref{thm: Nishioka} to 
the system \eqref{eq:mahleri}, for every $i$, $1\leq i \leq r$, we deduce that   
\begin{equation}
\label{eq:Nishioka2}
{\rm tr.deg}_{\Q}\{f_{i,j}(\balpha_i)\} : 1\leq j \leq m_i\}={\rm tr.deg}_{\Q(\z_i)}\{f_{i,j}(\z_i) : 1\leq j \leq m_i\}\,.
\end{equation}
 It follows from \eqref{eq:Nishioka1}, \eqref{eq:Nishioka2}, and Lemma \ref{lem:independantvariables} that
\begin{eqnarray}
\label{eq:trdegTotal}
{\rm tr.deg}_{\Q}\{f_{i,j}(\balpha_i) : 1\leq i \leq r,\, 1\leq j \leq m_i\} \hspace{3cm}\\
\nonumber \hspace{4cm} = \sum_{i=1}^r {\rm tr.deg}_{\Q}\{f_{i,j}(\balpha_i) : 1\leq j \leq m_i\}\, .
\end{eqnarray}
For every $i$, set $\mathcal F_i:= \{f_{i,j}(\balpha_i) : 1\leq j \leq m_i\}$. 
Using \eqref{eq:trdegTotal}, we can thus apply Lemma \ref{lem:transcendancedegree} to deduce that 
${\rm tr.deg}_{\Q}(\mathcal E)=\sum_{i=1}^r {\rm tr.deg}_{\Q}(\mathcal E_i)$, as wanted.

Now, let us assume that the assumptions of Theorem \ref{thm: purity} are satisfied.  
The proof is essentially the same. The only change occurs when establishing Equality \eqref{eq:Nishioka1}. 
We infer from assumptions (i) and (ii) of Theorem \ref{thm: purity} that we can apply Theorem \ref{thm: families}.
Then, using Theorem \ref{thm: families} and arguing as in the proof of Corollary \ref{thm: Nishioka},  we 
deduce that Equality \eqref{eq:Nishioka1} holds. The rest of the proof remains unchanged. 
\end{proof}

\subsection{Proof of Theorems  \ref{thm: purete2} and \ref{thm: purity}} 
We continue with the notation of Theorem \ref{thm: purete2} and \ref{thm: purity}. 
For every $i$, $1\leq i\leq r$, we set 
$$
\mathcal E_i:=(f_{i,\ell_1}(\balpha_i),\ldots,f_{i,\ell_{s_i}}(\balpha_i))\,,
$$
where $1\leq \ell_1< \ell_2 < \cdots < \ell_{s_i}\leq m_i$. Note that the inclusion
\begin{equation}\label{eq:inclusionindirect}
\sum_{i=1}^r {\rm Alg}_{\Q}(\mathcal E_i \mid \mathcal E) \subset {\rm Alg}_{\Q}(\mathcal E)\,
\end{equation}
is trivial.
Suppose that the assumptions of either Theorem \ref{thm: purete2} or Theorem \ref{thm: purity} hold. 
By Corollary \ref{cor:2}, we have 
\begin{equation}
\label{eq:egalitetranscdeg}
{\rm tr.deg}_{\Q}\left(\mathcal E \right) = \sum_{i=1}^r {\rm tr.deg}_{\Q}(\mathcal E_i)\, .
\end{equation}
Given an ideal $\I$, we let ${\rm ht}(\mathcal I)$ denote its height. Then, we have
\begin{equation}
\label{eq:height_transdeg}
\begin{array}{rcl}
 {\rm ht}\left({\rm Alg}_{\Q}(\mathcal E_i)\right) &=& s_i-{\rm tr.deg}_{\Q}(\mathcal E_i)\, ,\;\;\forall i,\; 1\leq i\leq r,\
 \\ {\rm ht}\left({\rm Alg}_{\Q}(\mathcal E)\right) &=& S-{\rm tr.deg}_{\Q}(\mathcal E)\, ,
\end{array}
\end{equation}
where $S:=s_1+\cdots+s_r$.
From \eqref{eq:egalitetranscdeg} and \eqref{eq:height_transdeg} we deduce that
$$
{\rm ht}\left({\rm Alg}_{\Q}(\mathcal E)\right) =\sum_{i=1}^r {\rm ht}\left({\rm Alg}_{\Q}(\mathcal E_i)\right)\, .
$$
Set $\mathcal I := \sum_{i=1}^r {\rm Alg}_{\Q}(\mathcal E_i \mid \mathcal E)$. Then the  isomorphism\footnote{See, for instance, \cite[Chap. I, Exercise 3.15]{Har}.}
$$
\Q[\X_1]\big/{\rm Alg}_{\Q}(\mathcal E_1) \otimes_{\Q} \cdots \otimes_{\Q} \Q[\X_r]\big /{\rm Alg}_{\Q}(\mathcal E_r)  \cong \Q[\X]\big/{\mathcal I}\,
$$
implies that $\mathcal I$ is a prime ideal. Indeed, the tensor product of integral domains,
over an algebraically closed field, is an integral domain.
Furthermore, this isomorphism also gives that
${\rm ht}(\mathcal I)=\sum_{i=1}^r{\rm ht}({\rm Alg}_{\Q}(\mathcal E_i))$. It follows that ${\rm Alg}_{\Q}(\mathcal E)$ and $\sum_{i=1}^r {\rm Alg}_{\Q}(\mathcal E_i \mid \mathcal E)$ are both prime ideals with the same height. By \eqref{eq:inclusionindirect}, these two ideals are equal.  This ends the proof.
\qed
\section{Proof of Theorem \ref{thm: main}} \label{sec: theorem1}

In this section, we show how to deduce Theorem \ref{thm: main} from  
 the two purity theorems. We first prove the following lemma.

\begin{lem}\label{lem: adapt}
Let $f(z)$ be a $q$-Mahler function and $\alpha$ be a nonzero algebraic number 
such that $f(z)$ is well-defined at $\alpha$. Then there exists a $q$-Mahler function $g(z)$ 
such that the following properties hold. 

\begin{itemize}
\item[{\rm (a)}]  $g(\alpha)=f(\alpha)$.

\medskip

\item[{\rm (b)}] The function $g(z)$ is the first coordinate of a vector solution to a $q^l$-Mahler 
system, say  
\begin{equation*}
\left(\begin{array}{c} g_1(z)=g(z) \\ \vdots \\ g_{m}(z) \end{array}\right)
=B(z)\left(\begin{array}{c} g_{1}(z^{q^l}) \\ \vdots \\ g_{m}(z^{q^l}) \end{array}\right)\, .
\end{equation*}
for some integer $l>0$.

\medskip

\item[{\rm (c)}]  The point $\alpha$ is regular with respect to this system. 

\end{itemize}
\end{lem}

\begin{proof} 
We first note that if $f(\alpha)$ is algebraic, the lemma is trivial for we can choose 
$g(z):= f(\alpha)$ to be constant.  We assume now that $f(\alpha)$ is transcendental. 
Using a minimal $q$-Mahler equation for $f(z)$, we deduce that $f(z)$ is the first coordinate of some 
$q$-Mahler system, say 
\begin{equation}
\label{eq:Mahlerunevariable}
\left(\begin{array}{c} f_{1}(z) \\ \vdots \\ f_{m}(z) \end{array}\right)
=A(z)\left(\begin{array}{c} f_{1}(z^q) \\ \vdots \\ f_{m}(z^q) \end{array}\right)\, ,
\end{equation}
where $f_1(z)=f(z),\ldots,f_m(z)$ are linearly independent over $\Q(z)$. 
Since the functions $f_1(z),\ldots,f_m(z)$ are linearly independent, 
we infer from \cite[Theorem 1.10]{AF1} that there exists an integer $l$ 
such that the numbers $\alpha^{q^l}$ is regular with respect to the system
\begin{equation}
\label{eq:Mahleriteree}
\left(\begin{array}{c} f_{1}(z) \\ \vdots \\ f_{m}(z) \end{array}\right)
=A_l(z)
\left(\begin{array}{c}f_{1}(z^{q^l}) \\ \vdots \\ f_{m}(z^{q^l}) \end{array}\right)\, ,
\end{equation}
where
$$
A_l(z)=A(z)A(z^q)\cdots A(z^{q^{l-1}})\, .
$$
Furthermore, this theorem ensures that $\alpha$ is not a pole of the matrix $A_l(z)$. 
Let $(a_1(z),\ldots,a_{m}(z))$ denote the first row of $A_l(z)$. 
Set  
\begin{equation}\label{eq:g(z)}
g(z)= a_{1}(\alpha)f_1(z^{q^l})+\cdots+ a_{m}(\alpha)f_m(z^{q^{l}})\, . 
\end{equation}
Note that $g(z)$ is a $q$-Mahler function for it is obtained as a linear combination over $\Q$ 
of $q$-Mahler functions\footnote{Indeed, if $f(z)$ is a $q$-Mahler function then $f(z^{q^l})$ is 
clearly a $q$-Mahler function too.}.
Since $f(\alpha)$ is transcendental,  the vector $(a_{1}(\alpha),\ldots,a_m(\alpha))$ 
is nonzero. 
Then applying a suitable constant gauge transformation to the Mahler system associated 
with the matrix $A_l(z^{q^l})$, we can obtain a Mahler system which has a solution vector with 
$g(z)$ as first coordinate.  Furthermore, 
since the point $\alpha^{q^l}$ is regular w.r.t.\  \eqref{eq:Mahleriteree}, 
$\alpha$ is a regular point w.r.t.\ this new system.  
 On the other hand, we infer from \eqref{eq:Mahleriteree} and \eqref{eq:g(z)} that $g(\alpha)=f(\alpha)$, 
 as wanted. 
 \end{proof}

\begin{proof}[Proof of Theorem \ref{thm: main}]
We keep on with the notation of Theorem \ref{thm: main}. 
We assume that no number among $f_1(\alpha_1),\ldots,f_r(\alpha_r)$ belongs to $\mathbb K$, so that it remains to prove that $f_1(\alpha_1),\ldots,f_r(\alpha_r)$ are algebraically independent over $\Q$.  
By \cite[Corollaire 1.8]{AF1}, our assumption implies that the numbers $f_1(\alpha_1),\ldots,f_r(\alpha_r)$ are all transcendental. For every $i$, $1\leq i\leq r$, we let  $z_i$ denote an indeterminate.

By Lemma \ref{lem: adapt}, with each pair $(f_i,\alpha_i)$, we can associate a $q_i$-Mahler function $g_i(z_i)$ 
such that 
\begin{equation}\stepcounter{equation}
\label{eq:mahler1Vi}\tag{\theequation .$i$}
\left(\begin{array}{c} g_{i,1}(z_i)=g_i(z_i) \\ \vdots \\ g_{i,m_i}(z_i) \end{array}\right)
=B_i(z_i)\left(\begin{array}{c} g_{i,1}\left(z_i^{q_i^{l_i}}\right) \\ \vdots \\ g_{i,m_i}\left(z_i^{q_i^{l_i}}\right) \end{array}\right)\, ,
\end{equation}
$g_i(\alpha_i)=f_i(\alpha_i)$ and $\alpha_i$ is regular w.r.t.\ \eqref{eq:mahler1Vi}.  

Let us first prove Case (i) of Theorem \ref{thm: main}.  
Let us divide the natural numbers $1,\ldots,r$ into $s$ classes $\mathcal I_1=\{i_{1,1},\ldots,i_{1,\nu_1}\},\ldots,\mathcal I_s=\{i_{s,1},\ldots,i_{s,\nu_s}\}$, such that 
$i$ and $j$ belong to the same classe if and only if  $q_i$ and $q_j$ are multiplicatively dependent. 
Iterating each system \eqref{eq:mahler1Vi}   
a suitable number of times, we can assume without loss of generality that $q_i^{l_i}=q_j^{l_j}:=\rho_k$ 
whenever 
$i$ and $j$ belong to $\mathcal I_k$.   
Set $\mathcal E := (g_1(\alpha_1),\ldots,g_r(\alpha_r))$ and 
\begin{equation*} 
\mathcal E_k:=(g_{i_{k,1}}(\alpha_{i_{k,1}}),\ldots,g_{i_{k,\nu_k}}(\alpha_{i_{k,\nu_k}}))
 \,,\; \forall k\in\{1,\ldots,s\}\,. 
\end{equation*}  
Given $k\in\{1,\ldots,s\}$, we consider the Mahler system in the variables $z_i,i\in \mathcal I_k$,   
associated with  the matrix $\oplus_{i\in \mathcal I_k}B_i(z_i)$ and 
the transformation $T_k=\rho_k{\rm I}_{\nu_k}$, where we let ${\rm I}_n$ denote the identity matrix of size $n$.   In this way, we have converted our $r$ Mahler systems in one variable into 
$s$ Mahler systems, each having respectively $\nu_1,\ldots,\nu_s$ variables.  
Furthermore, since by assumption the algebraic numbers $\alpha_1,\ldots,\alpha_r$ 
are multiplicatively independent, we deduce that each pair 
$$(T_k,\balpha_k:=(\alpha_{i_{k,1}},\ldots,\alpha_{i_{k,\nu_k}}))\,,\; 1\leq k\leq s\, ,$$
 is admissible. 
Finally, the point $\balpha_k$ is regular for each 
$\alpha_i$ is regular w.r.t.\ \eqref{eq:mahler1Vi}.  Since, by construction, the numbers 
$\rho(T_1)=\rho_1,\ldots,\rho(T_s)=\rho_s$ are pairwise multiplicatively independent, we  
can apply our second purity theorem, Theorem \ref{thm: purity}, to these $s$ Mahler systems. 
We deduce that 
\begin{equation}\label{eq:Es}
{\rm Alg}_{\Q}(\mathcal E) = \sum_{k=1}^s {\rm Alg}_{\Q}(\mathcal E_k \mid \mathcal E)\, .
\end{equation}
Now, let us fix $k\in\{1,\ldots,s\}$. 
Since the numbers $\alpha_i,i\in \mathcal I_k$,    
are multiplicatively independent, we can apply our first purity theorem, Theorem \ref{thm: purete2}, to the 
$\nu_k$ distinct Mahler systems \eqref{eq:mahler1Vi}, with $i\in\mathcal I_k$.  For every $i\in\mathcal I_k$,  
set $\mathcal E_{k,i}:=(g_i(\alpha_i))$. Since  $g_i(\alpha_i)=f_i(\alpha_i)$ is transcendental, we 
have ${\rm Alg}_{\Q}(\mathcal E_{k,i}) =\{0\}$ for every $i\in \mathcal I_k$. We thus 
deduce from Theorem \ref{thm: purete2} 
that 
$${\rm Alg}_{\Q}(\mathcal E_k) = \sum_{i\in\mathcal I_k}
 {\rm Alg}_{\Q}(\mathcal E_{k,i} \mid \mathcal E_k)=\{0\} \,.
 $$ 
Since this holds for every $k$, $1\leq k\leq s$, 
 it follows from \eqref{eq:Es}, that ${\rm Alg}_{\Q}(\mathcal E)  =\{0\}$. That is,  
 $f_1(\alpha_1),\ldots,f_r(\alpha_r)$ are algebraically independent over $\Q$.   

Now, let us prove Case (ii) of Theorem \ref{thm: main}. 
As previously, we associate with each pair $(f_i(z),\alpha_i)$ a function $g_i(z)$ satisfying 
the conditions of Lemma \ref{lem: adapt}.  Since the natural numbers $q_i$ are pairwise multiplicatively independent, 
we can apply our second purity theorem, Theorem \ref{thm: purity}, to the Mahler systems associated with each 
$g_i(z)$ in Lemma \ref{lem: adapt}. 
Setting 
$$\mathcal E:= (g_1(\alpha_1),\ldots,g_r(\alpha_r)) \, 
$$ 
and $\mathcal E_i:=(g_i(\alpha_i))$, $1\leq i\leq r$, we deduce that 
$$
{\rm Alg}_{\Q}(\mathcal E) = \sum_{i=1}^r {\rm Alg}_{\Q}(\mathcal E_i \mid \mathcal E)\, .
$$
Again, since 
by assumption $g_i(\alpha_i)=f_i(\alpha_i)$ is transcendental, we get that 
${\rm Alg}_{\Q}(\mathcal E_i \mid \mathcal E)=0$ for every $i$, $1\leq i\leq r$. 
This shows that ${\rm Alg}_{\Q}(\mathcal E) =\{0\}$, and  
we conclude, as previously, that $f_1(\alpha_1),\ldots,f_r(\alpha_r)$ are algebraically 
independent over $\Q$.  
\end{proof}

\appendix
\section{Representing numbers in independent bases}\label{preamble}

In this appendix, we show how our main results apply to certain problems concerning expansions 
of numbers in multiplicatively independent bases. In particular, we state  and prove 
Conjectures \ref{conj: weakv}, \ref{conj: strongv}, and Corollary \ref{coro: fonctions},  
which were our initial goal.  
All these  results are deduced from Theorem \ref{thm: main}. 

\subsection{The dynamical point of view: Furstenberg's conjecture}

In the late 1960s, Furstenberg \cite{Fu67,Fu70} 
suggested a series of conjectures whose aim is to capture  
the heuristic which has been alluded to only in very vague terms at the beginning of this paper. 
These conjectures, which became famous, take place in a dynamical setting. 
This does not come as a great surprise for there is a well-known dictionary 
transferring combinatorial properties of 
the expansion of a real number $x$ in an integer base $q$ in terms of 
dynamical properties of the orbit of $\{x\}$ under the map $T_q$ defined on 
$\mathbb R/\mathbb Z$ by $x\mapsto qx$.  
We let $\mathcal O_q(x)$ denote the forward orbit of $x$ under $T_q$, that is,  
$$
\mathcal O_q(x):= \left\{x,T_q(x),T^2_q(x),\ldots\right\} \,.
$$
If $X\subset \mathbb R$, we let $\dim_H(X)$ denote the Hausdorff dimension of $X$ 
and $\overline{X}$ its closure.  
One of Furstenberg's conjecture \cite{Fu70} reads as follows. 

\begin{conj}[Furstenberg]\label{conj: F} 
Let $p$ and $q$ be two multiplicatively independent natural numbers, and 
let $x\in[0,1)$ be a real number. Then 
$$
\dim_H \overline{\mathcal O_p(x)} + \dim_H \overline{\mathcal O_q(x)} \geq 1\,,
$$
unless $x$ is rational. 
\end{conj}

This conjecture has wonderful consequences about expansions of both real and natural numbers. 
It beautifully expresses the expected balance between the complexity of expansions of an irrational real number 
in two multiplicatively independent bases: 

\smallskip

 \emph{If $x$ has a simple expansion in base $p$, then it should have a complex expansion in base $q$. }
 
 \smallskip
 
It is easy to see that Conjecture \ref{conj: F} holds true generically.  Indeed, 
endowed with the Haar measure, the topological dynamical system 
$(T_q, \mathbb R/\mathbb Z)$ becomes ergodic, 
and it follows from the ergodic theorem that 
$$
\dim_H \overline{\mathcal O_p(x)} =\dim_H \overline{\mathcal O_q(x)}=1\,, 
$$
for almost all real numbers $x$ in $[0,1)$.  
In fact, all the strength of Conjecture \ref{conj: F} takes shape when $x$ has a 
simple expansion in one of the two bases.  
Defining the entropy of $x$ with respect to the base $q$  
 as the topological entropy of the 
dynamical system $(T_q, \overline{\mathcal O_q(x)})$, 
Conjecture \ref{conj: F} implies that if $x$ has zero entropy in base $p$, then it has a dense orbit under 
$T_q$. 

 Let us illustrate this with a concrete example. The binary Thue-Morse number $\tau$ is defined as follows. 
 Its $n$th binary digit is equal to $0$ if the sum of digits in the binary expansion of $n$ is even, 
and to $1$ otherwise. It is somewhat puzzling that  its decimal expansion 
$$
\langle \tau \rangle_{10} = 0.412\, 454 \, 033 \,640\, 107\, 597\, 783\, 361\, 368\, 258\, 455\, 283\, 089\cdots 
$$
seems unpredictable, while its binary expansion 
$$
\langle \tau \rangle_2 = 0.011\, 010\, 011\, 001\, 011\, 010\, 010\, 110\, 011\, 010\, 011\, 001\, 011 \cdots\, 
$$
is, by definition, so regular. 
This intriguing phenomenon would be nicely explained 
by Conjecture \ref{conj: F}.  
Indeed, since $\tau$ has zero entropy in base $2$, 
 it should have a dense orbit under $T_{10}$, meaning that  
all blocks of digits should occur in its decimal expansion.  

Other astonishing consequences of Conjecture \ref{conj: F} concern expansions of natural numbers.  
For instance, using an elementary construction, Furstenberg shows in \cite{Fu70} how to deduce from  
Conjecture \ref{conj: F} that any finite block of digits 
occurs in the decimal expansion of $2^n$, as soon as $n$ is large enough. 
Note that, in the same vein, a conjecture of Erd\"os claims that the digit $2$ occurs in the ternary 
expansion of $2^n$ for all  $n>8$ (see, for instance, \cite{Lag}).   

Recently, Shmerkin \cite{Sh19} and Wu \cite{Wu19} proved that  
the set of  exceptions to Conjecture \ref{conj: F} has Hausdorff dimension zero. 
Unfortunately, this remarkable results does not tell us anything about expansions of real numbers 
with zero entropy in some base. Indeed, 
the set of all such real numbers 
has Hausdorff dimension zero \cite{MM12}. Though the works of Shmerkin and Wu mark 
significant progress towards Conjecture \ref{conj: F}, the latter remains far out of the reach of current methods.  
Even worse, we are afraid that their result could be essentially the best dynamical methods 
have to say about this conjecture.  

\subsection{The computational point of view: from finite automata to Mahler's method}\label{sec:preambule} 

From a computational point of view, there is another relevant notion of 
 simple number, 
namely the notion of \emph{automatic real number} (see \cite[Chap.\ 13]{AS03}).  
While  computable numbers can be generated by general Turing machines,   
automatic numbers are those whose expansion in some base  
can be generated by a finite automaton. Broadly speaking, a finite automaton is  a Turing machine 
without any memory tape, all its memory being stored in the finite state control. 
This severe restriction justifies that these numbers are considered as especially simple.  
For example, the Thue-Morse number $\tau$ is automatic in base $2$. 
We refer the reader to \cite{ACL} and the references therein for a discussion on these different models 
of computation.  In this new framework, our general heuristic naturally leads to the following conjecture. 

\begin{conj}\label{conj: weakv} Let $p$ and $q$ be two multiplicatively independent natural 
numbers. A real number cannot be automatic in both bases $p$ and $q$,  unless it is rational. 
\end{conj}

This conjecture turns out to be a very special case of Conjecture \ref{conj: F} for  
being automatic in some base implies having zero entropy in that base.   
Nevertheless, Conjecture \ref{conj: weakv} remains quite challenging since, 
to date, not a single real number has been proved to be at once automatic in some base and 
not automatic in another one.  We also mention that a weaker version of Conjecture \ref{conj: weakv}  
appears as Open Problems 7 in \cite[Chap.\ 13]{AS03}. 

With a more Diophantine flavor, Conjecture 
\ref{conj: weakv} can be strongly strengthened as follows.   

\begin{conj}\label{conj: strongv}
Let $r\geq 1$ be an integer. Let $b_1,\ldots,b_r$ be multiplicatively independent positive integers, 
and, for every $i$, $1\leq i \leq r$, 
let  $\xi_i$ be a real number that is automatic in base $b_i$.   
Then the numbers $\xi_1,\ldots,\xi_r$ are algebraically independent over $\Q$, unless one of them is rational. 
\end{conj}

Conjecture \ref{conj: strongv} is not implied by Furstenberg's conjecture. The former 
does not only imply that the 
Thue-Morse number $\tau$ cannot be automatic in base $10$, but also 
that this is the case for any number obtained from $\tau$ by using algebraic numbers and algebraic 
operations 
(addition, multiplication, division, taking $n$th roots...). 
The case $r=1$ was a long-standing conjecture first proved by Bugeaud and the first author  
\cite{AB07} by means of the subspace theorem. 
See also \cite{AF1,PPH} for a recent different proof based on Mahler's method.   
So far, Conjecture \ref{conj: strongv} has only be settled in that particular case.  

\subsubsection{Connection with Mahler's method and Theorem \ref{thm: main}} 

In 1968, Cobham \cite{Co68} first noticed the following fundamental connection 
between automatic sequences and $M$-functions. 
If the sequence $(a_n)_{n\geq 0}$ is $q$-automatic, then the generating function  
$$
f(z):= \sum_{n=0}^{\infty}a_nz^n 
$$
is a {\it $q$-Mahler function}.  
In turn, problems about transcendence and algebraic independence of 
automatic real numbers can be translated and extended to problems concerning 
transcendence and algebraic independence of $M$-functions at algebraic points. 
In particular, Conjecture \ref{conj: strongv}, 
and hence Conjecture \ref{conj: weakv}, easily follow from Part (i) of Theorem \ref{thm: main}.  

\begin{proof}[Proof of Conjecture \ref{conj: strongv}]
Replacing  $\xi_1,\ldots,\xi_r$ by their fractional part if necessary, we can assume without 
any loss of generality that $0\leq \xi_i<1$, for every $i$, $1\leq i\leq r$.  
By assumption, the number $\xi_i$ is automatic in base $b_i$.  
This means that, for some integer $q_i\geq 2$, there exists a $q_i$-automatic sequence 
$(a_{i,n})_{n\geq 0}$ with values in $\{0,1,\ldots,b_i-1\}$ such that $\xi_i=f_i(1/b_i)$, where  
$f_i(z):= \sum_{n=0}^{\infty}a_{i,n}z^n\in \mathbb Q\{z\}$. A discussed in \cite{Co68}, the fact that 
the sequence 
$(a_{i,n})_{n\geq 0}$  is $q_i$-automatic implies that $f_i(z)$ is a $q_i$-Mahler function.    
Now, let us assume that $\xi_1,\ldots,\xi_r$ are all irrational. By \cite[Corollaire 1.8]{AF1},  
we obtain that these numbers are all transcendental. 
Since by assumption the numbers $1/b_1,\ldots,1/b_r$ are multiplicatively independent, 
Part (i) of Theorem \ref{thm: main} implies that the numbers 
$\xi_1=f_1(1/b_1),\ldots,\xi_r=f_r(1/b_r)$ are algebraically independent, as wanted. 
\end{proof}

As with Furstenberg's conjecture, Theorem \ref{thm: main} 
has also valuable consequences about expansions of natural numbers.   
Let us first recall that a set $\mathcal E\subset \mathbb N$ is $q$-automatic if its elements, 
when written in base $q$, 
can be recognized by a finite automaton (See, for instance, \cite[Chapter 5]{AS03}). 
In this framework, there is a famous theorem by Cobham \cite{Co69} that can be stated as follows.
If $\mathcal E\subset \mathbb N$ is both $p$- and 
$q$-automatic, 
where 
$p$ and $q$ are multiplicatively independent, then $\mathcal E$ is a periodic set, meaning that  $\mathcal E$ 
is the union of a finite set and finitely many arithmetic progressions.  

Cobham's theorem can be rephrased in terms of power series. 
Indeed, it is equivalent to the fact that, given any aperiodic $p$-automatic set $\mathcal E_p$ and any 
aperiodic $q$-automatic set $\mathcal E_q$, 
the corresponding generating functions  cannot be equal. That is,  
$$\sum_{n\in\mathcal E_p} z^n=:f_p(z) \not= f_q(z):=\sum_{n\in\mathcal E_q} z^n\, .$$ 
In 1987, Loxton and van der Poorten \cite{vdP87} conjectured the following generalization:   
a power series in $\Q[[z]]$ cannot be both $p$-Mahler and $q$-Mahler, unless it is rational. 
This conjecture was first proved by Bell and the first author in \cite{AB17}, while 
a different proof was given by Sch\"afke and Singer \cite{SS19}. Very recently, the authors of 
\cite{ADHW} even proved a stronger result also conjectured by Loxton and van der Poorten \cite{vdP87}: 
a $p$-Mahler function $f_p(z)\in\Q[[z]]$ and a $q$-Mahler function $f_q(z)\in\Q[[z]]$ are algebraically independent over $\Q(z)$, unless one of them is rational.   
This result refines Cobham's theorem by expressing, in algebraic terms, the discrepancy between 
aperiodic automatic sets associated with 
multiplicatively independent input bases. The proof given in \cite{ADHW} is based on a 
suitable parametrized Galois theory associated with linear Mahler equations and follows the strategy 
initiated in 
\cite{ADH}.  

Part (ii) of Theorem \ref{thm: main} leads to the following significant generalization of  
all the aforementioned results, providing in particular a 
totally new proof of Cobham's theorem.

\begin{coro}\label{coro: fonctions}
Let $r\geq 1$ be an integer. For every integer $i$, $1\leq  i\leq r$, let $q_i\geq 2$ be an integer and 
$f_i(z)\in \Q\{z\}$ be a $q_i$-Mahler function. Assume that $q_1,\ldots, q_r$ are pairwise multiplicatively independent. Then $f_1(z),\ldots,f_r(z)$ are algebraically independent over $\Q(z)$, unless one of them 
is  rational.   
\end{coro}

The case $r=1$ is a classical result (see, for instance, \cite[Theorem 5.1.7]{Ni_Liv}), 
while, as previously mentioned, the case $r=2$ is much harder and was only recently proved in \cite{ADHW}. 

\begin{proof}
Let us assume that the functions $f_1(z),\ldots,f_r(z)$ are all irrational. 
Then, by  \cite[Theorem 5.1.7]{Ni_Liv}, they are all transcendental over $\Q(z)$.  
Combining Nishioka's theorem and \cite[Lemma 6]{Bec94}, we deduce that there exists $r>0$ such that 
for all algebraic numbers $\alpha$, with $0<\vert \alpha\vert<r$, the numbers $f_1(\alpha),\ldots,f_r(\alpha)$ are  all transcendental.  
Picking such $\alpha$ and applying Part (ii) of Theorem \ref{thm: main}, we obtain that the numbers 
$f_1(\alpha),\ldots,f_r(\alpha)$ are algebraically independent over $\Q$. 
Hence the functions $f_1(z),\ldots,f_r(z)$ 
are algebraically independent over $\Q(z)$, as wanted. 
\end{proof}



\begin{thebibliography}{99}

\bibitem{Ad19} B. Adamczewski, {\it Mahler's method}, Doc. Math. Extra Vol. Mahler Selecta (2019), 
95--122.  

\bibitem{AB17} B. Adamczewski and J. Bell, {\it A problem about Mahler functions}, 
Ann. Sc. Norm. Super. Pisa {\bf 17} (2017), 1301--1355. 

\bibitem{ABS} B. Adamczewski, J. Bell, and D. Smertnig, 
{\it A height gap theorem for coefficients of Mahler functions}, 
preprint 2020, arXiv:2003.03429 [math.NT], 42 pp.

\bibitem{AB07} B. Adamczewski and Y. Bugeaud, {\it On the complexity of algebraic numbers I. Expansions in integer bases}, Annals of Math. {\bf 165} (2007), 547--565.

\bibitem{ACL} B. Adamczewski, J. Cassaigne, and M. Le Gonidec, {\it On the computational complexity of algebraic numbers: the Hartmanis--Stearns problem revisited}, 
 Trans. Amer. Math. Soc. {\bf 373} (2020), 3085--3115. 

\bibitem{ADH} B. Adamczewski, T. Dreyfus, and C. Hardouin, {\it Hypertranscendence and linear difference equations}, to appear in J. Amer. Math. Soc. (2020).

\bibitem{ADHW} B. Adamczewski, T. Dreyfus, C. Hardouin, and M. Wibmer, {\it Algebraic independence and linear difference equations}, preprint 2020, arXiv:2010.09266 [math.NT], 31 pp.

 \bibitem{AF1} B. Adamczewski et C. Faverjon, {\it M\'ethode de Mahler: relations lin\'eaires, transcendance et applications aux nombres automatiques}, 
 Proc. London Math. Soc. {\bf 115} (2017), 55--90.  
 
 \bibitem{AF2} B. Adamczewski et C. Faverjon, {\it M\'ethode de Mahler, transcendance et relations lin\'eaires : aspects effectifs}, 
 J. Th\'eor. Nombres Bordeaux  {\bf 30} (2018), 557--573. 
 
   \bibitem{AF3} B. Adamczewski and C. Faverjon, {\it Mahler's method in several variables I: The theory of regular singular systems}, preprint 2018, arXiv:1809.04823 [math.NT], 65 pp.
 
  \bibitem{AF4} B. Adamczewski and C. Faverjon, {\it Mahler's method in several variables II:  
  Applications to base change problems and finite automata}, preprint 2018, arXiv:1809.04826 [math.NT], 47 pp.
 
 \bibitem{AS03} J.-P. Allouche and J. Shallit,
Automatic sequences. 
Theory, applications, generalizations, Cambridge University Press, Cambridge, 2003.
 
 \bibitem{An1} Y. Andr\'e, {\it S\'eries Gevrey de type arithm\'etique I. Th\'eor\`emes de puret\'e et de dualit\'e}, 
 Annals of Math. {\bf 151} (2000), 705--740.   

 \bibitem{An2} Y. Andr\'e, {\it S\'eries Gevrey de type arithm\'etique II. Transcendance sans transcendance}, 
 Annals of Math. {\bf 151} (2000),  
741--756.  


\bibitem{An3} Y. Andr\'e, {\it Solution algebras of differential equations and quasi-homogeneous varieties$:$ a new differential Galois correspondence},  
Ann. Sci. \'Ec. Norm. Sup\'er. {\bf 47}  (2014), 449--467.

\bibitem{Bec94} P.-G. Becker, {\it $k$-regular power series and Mahler-type functional equations}, 
J. Number Theory {\bf 49} (1994), 269--286. 

\bibitem{Beu06} F. Beukers, {\it A refined version of the Siegel--Shidlovskii theorem}, Annals of Math. 
{\bf 163} (2006), 369--379. 

\bibitem{Brown} T.C. Brown, {\it On locally finite semigroups} (In Russian), Ukraine Math. J. {\bf 20} (1968), 732--738.

\bibitem{Cartan} H. Cartan, Th\'eorie \'el\'ementaire des fonctions analytiques d'une ou plusieurs variables complexes, Hermann, 1961.

\bibitem{Co68} A. Cobham, {\it On the Hartmanis-Stearns problem for a class of tag machines}, Conference
Record of 1968 Ninth Annual Symposium on Switching and Automata Theory, Schenectady,
New York (1968), 51--60.

\bibitem{Co69}
 A. Cobham, {\it On the base-dependence of sets of numbers recognizable by finite automata}, Math. Systems Theory {\bf 3} (1969), 186--192. 

\bibitem{CZ05} P. Corvaja and U. Zannier, {\it S-unit Points on Annalytic Hypersurfaces}, Ann. Sci. \'Ec. Norm. Sup. {\bf 38} (2005), 76--92. 

\bibitem{Du11} F. Durand, {\it Cobham's theorem for substitutions}, J. Eur. Math. Soc. {\bf 13} (2011),  
1799--1814.

\bibitem{Fa} C. Faverjon, Contribution \`a la m\'ethode de Mahler, \'equations lin\'eaires et automates finis, 
Th\`ese de l'Universit\'e Claude Bernard Lyon 1, 2020, HAL Id : tel-02977792. 

\bibitem{FN98} N. I. Fel'dman, Yu. V. Nesterenko, Transcendental numbers. Number theory, IV, 1--345, Encyclopaedia Math. Sci. {\bf 44}, Springer, Berlin, 1998. 

\bibitem{Fu67} H. Furstenberg, {\it 
Disjointness in ergodic theory, minimal sets, and a problem in Diophantine approximation},  
Math. Systems Theory {\bf 1} (1967), 1--49. 

\bibitem{Fu70} H. Furstenberg, {\it Intersections of Cantor sets and transversality of semigroups}, 
in Problems in Analysis (Sympos. Salomon Bochner, Princeton Univ., Princeton, N.J., 1969), Princeton Univ. Press, Princeton, N.J., 1970, pp. 41--59.

\bibitem{Grantmacher} F. R. Gantmacher, Applications of the Theory of Matrices, Interscience
              Publishers Ltd., London, 1959.

\bibitem{Har} R. Hartshorne, 
Algebraic geometry,  
Graduate Texts in Mathematics {\bf 52} Springer-Verlag, New York-Heidelberg, 1977.

\bibitem{Ku76} K. K. Kubota,  {\it An application of Kronecker's theorem to transcendence theory}, 
S\'eminaire de Th\'eorie des Nombres de Bordeaux (1975--1976), Exp. 25, 10 pp.  


\bibitem{Ku77} K. K. Kubota, {\it On the algebraic independence of holomorphic solutions of certain functional equations and their values}, 
Math. Ann. {\bf 227} (1977), 9--50. 

\bibitem{Lag} J. C. Lagarias,  {\it Ternary expansions of powers of $2$}, 
J. Lond. Math. Soc.  {\bf 79} (2009), 562--588.

\bibitem{Lang} S. Lang, Algebra, revised third edition, Graduate Texts in Mathematics {\bf 21} Springer-Verlag, New York, 2002.

\bibitem{Lo84} J. H. Loxton, {\it A method of Mahler in transcendence theory and some of its applications}, 
Bull. Austral. Math. Soc. {\bf 29} (1984), 127--136. 

\bibitem{Lo88} J. H. Loxton, {\it Automata and transcendence} in New advances in transcendence theory (Durham, 1986), pp. 215--228, Cambridge Univ. Press, Cambridge, 1988.


\bibitem{LvdP77} 
J. H. Loxton and A. J. van der Poorten,
 {\it Arithmetic properties of certain functions in several variables}, 
 J. Number Theory {\bf 9} (1977), 87--106.
 
  \bibitem{LvdP77II} 
J. H. Loxton and A. J. van der Poorten,
 {\it Arithmetic properties of certain functions in several variables II}, 
 J. Austral. Math. Soc.  {\bf 24} (1977), 393--408.

  \bibitem{LvdP78} 
J. H. Loxton and A. J. van der Poorten,
{\it Algebraic independence properties of the Fredholm series}, 
J. Austral. Math. Soc. {\bf 26} (1978),  31--45. 

\bibitem{LvdP82} 
J. H. Loxton and A. J. van der Poorten, {\it Arithmetic properties of the solutions of a class of
 functional equations}, J. reine angew. Math. {\bf 330} (1982), 159--172.

\bibitem{Ma29}
K. Mahler, 
{\it Arithmetische Eigenschaften der 
L\"osungen einer Klasse von Funktionalgleichungen}, 
Math. Ann. {\bf 101} (1929),  342--367. 

\bibitem{Ma30a}
K. Mahler, 
{\it \"Uber das Verschwinden von Potenzreihen mehrerer Ver\"anderlichen in speziellen Punktfolgen}, 
Math. Ann. {\bf 103} (1930), 573--587 . 


\bibitem{Ma30b}
K. Mahler, 
{\it Arithmetische Eigenschaften einer Klasse transzendental-transzendente Funktionen}, Math. Z. {\bf 32} (1930), 545--585. 

\bibitem{Ma1969b}
K. Mahler, {\it Lectures on transcendental numbers}, in   Summer Institute on Number Theory at Stony Brook, 
1969, Proc. Symp. Pure Math.(Amer. Math. Soc.) XX (1969), 248--274.


\bibitem{Mas82} D. Masser, {\it A vanishing theorem for power series}, 
Invent. Math. {\bf 67} (1982), 275--296. 

\bibitem{Mas99} D. Masser, {\it Algebraic independence properties of the Hecke-Mahler series}, Quart. J. Math. Oxford {\bf 50} (1999), 207--230.

\bibitem{MM12} C. Mauduit and C. G. Moreira, {\it Generalized Hausdorff dimensions of sets of real numbers with zero entropy expansion}, Ergodic Theory Dynam. Systems {\bf 32} (2012), 1073--1089.

\bibitem{NS96} Yu. V. Nesterenko and A. B. Shidlovskii,  {\it On the linear independence of values of $E$-functions},  Mat. Sb. {\bf 187} (1996), 93--108 ;  translation in Sb. Math. {\it 187}  (1996), 1197--1211. 

\bibitem{NS} L. Naguy and T. Szamuely, {\it A general theory of Andr\'e's solution algebras}, to appear in Ann. Inst. Fourier.


\bibitem{Ni90}
Ku. Nishioka, 
{\it New approach in Mahler's method}, 
J. reine angew. Math. {\bf 407} (1990), 202--219.


\bibitem{Ni94} Ku. Nishioka {\it Algebraic independence by Mahler's method and S-unit equations}, 
Compos.  Math.  {\bf 92} (1994), 87--110.

\bibitem{Ni96} Ku. Nishioka {\it Algebraic independence of Mahler functions and their values}, 
Tohoku Math. J. {\bf 48} (1996), 51--70.

\bibitem{Ni_Liv} 
Ku. Nishioka, 
Mahler functions and transcendence,
Lecture Notes in Math. {\bf1631}, Springer-Verlag, Berlin, 1997.

\bibitem{PPH} P. Philippon, {\it Groupes de Galois et nombres automatiques}, J. Lond. Math. Soc. {\bf 92} (2015),  596--614. 

\bibitem{vdP76} A. J. van der Poorten, {\it On the transcendence and algebraic independence of certain somewhat amusing numbers (results of Loxton and van der Poorten)},  
S\'eminaire de Th\'eorie des Nombres de Bordeaux (1975--1976), Exp. 14, 14 pp.  

\bibitem{vdP87} A. J. van der Poorten, {\it Remarks on automata, functional equations and transcendence}, 
S\'eminaire de Th\'eorie des Nombres de Bordeaux (1986--1987), Exp. 27, 11 pp.  



\bibitem{SZ60} P. Samuel, 0. Zariski, Commutative Algebra Vol. II,  Graduate Texts in Mathematics {\bf 29}  Springer-Verlag, New York-Heidelberg, 1975.

\bibitem{SS19} R. Sch\"afke and M. Singer, {\it Consistent systems of linear differential and difference equations}, J. Eur. Math. Soc. {\bf 21} (2019), 2751--2792.

\bibitem{Sh19} P. Shmerkin, 
{\it On Furstenberg's intersection conjecture, self-similar measures, and the $L^q$ norms of convolutions},  
Annals of Math.  {\bf 189} (2019),  319--391. 

\bibitem{Sz} E. Szemer\'edi, {\it On sets of integers containing no $k$ elements in arithmetic progression}, Acta Arith. {\bf 27} (1975), 199--245. 

\bibitem{Miw} M. Waldschmidt, Diophantine approximation on linear algebraic groups. Transcendence properties of the exponential function in several variables, Grundlehren der Mathematischen Wissenschaften  
 {\bf 326}, Springer-Verlag, Berlin, 2000.

\bibitem{Wu19} M. Wu,  {\it A proof of Furstenberg's conjecture on the intersections of $\times p$- and $\times q$-invariant sets}, 
Annals of Math.  {\bf 189} (2019),  707--751.

\end{thebibliography}
\end{document}